\documentclass[10pt]{article}
\usepackage{amsmath, amsfonts, amssymb, verbatim, hyperref}
\usepackage{multirow} 
\usepackage{graphicx}
\usepackage{pstricks}
\usepackage[all,cmtip]{xy}
\usepackage{lscape}
\usepackage{longtable}

\newtheorem{theorem}{Theorem}[section]
\newtheorem{definition}[theorem]{Definition}

\newtheorem{corollary}[theorem]{Corollary}

\newtheorem{lemma}[theorem]{Lemma}

\newtheorem{proposition}[theorem]{Proposition}

\newenvironment{proof}[1][]{ \textbf{Proof#1. }}{$\Box$\medskip}

\newcommand{\gog}{{\mathfrak g}}

\newcommand{\gop}{{\mathfrak p}}
\newcommand{\gom}{{\mathfrak m}}
\newcommand{\goh}{{\mathfrak h}}

\newcommand{\gon}{{\mathfrak n}}
\newcommand{\gol}{{\mathfrak l}}
\newcommand{\gob}{{\mathfrak b}}
\newcommand{\gos}{{\mathfrak s}}
\newcommand{\got}{{\mathfrak t}}
\newcommand{\ad}{\mathrm{ad}}

\newcommand{\CharL}{\mathbf{ch} }
\newcommand{\CharS}{\mathbf{{\overline {ch}}} }
\DeclareMathOperator{\pr}{pr}

\newcommand{\scalarSA}[2]{\langle #1,#2 \rangle_{\bar \gog}}
\newcommand{\scalarLA}[2]{\langle #1,#2 \rangle_{\gog}}
\newcommand\clplus{\hbox{$\subset${\raise0.3ex\hbox{\kern -0.55em ${\scriptscriptstyle +}$}}\ }}
\newcommand\crplus{\hbox{$\supset${\raise1.05pt\hbox{\kern -0.60em ${\scriptscriptstyle +}$}}\ }}

\newcommand{\mV}{\mathbb V}

\newcommand{\mC}{\mathbb C}

\newcommand{\mZ}{\mathbb Z}

\DeclareMathOperator{\Cone}{Cone}
\DeclareMathOperator{\sign}{sign}

\DeclareMathOperator{\weights}{Weights}
\DeclareMathOperator{\supp}{Supp}
\DeclareMathOperator{\PrFin}{PrFin}

\newcommand{\mR}{\mathbb R}

\newcommand{\Tr}{\operatorname{Tr}}

\newcommand{\Hom}{\operatorname{Hom}}

\newcommand{\LieGtwo}{\mathrm{Lie~}G_2}
\newcommand{\LieAlgPair}[2]{#1\stackrel{i}{\hookrightarrow} #2}

\newcommand{\BThreeDiagram}[3]{
\gop_{(#1,#2,#3)}
}

\newcommand{\GTwoDiagram}[2]{
\bar{\gop}_{(#1,#2)}
}

\newcommand{\onlineVersionOnly}[1]{#1}
\newcommand{\offlineVersionOnly}[1]{}

\begin{document}
\baselineskip13pt

\title{The branching problem for generalized Verma modules, with application to the 
pair $(so(7),\LieGtwo)$ \onlineVersionOnly{ \\~\\
Extended version with tables
}
}
\author{\large T. Milev \\
\small Charles University Prague \\
\small and University of Massachusetts Boston
\\
\large P. Somberg
\\
\small 
Charles University Prague
%\thanks{This work was supported by grant GA\v CR No. 201/03/P137}
}

\maketitle

\begin{abstract}
We discuss the branching problem for generalized Verma modules 
$M_\lambda(\mathfrak g, \mathfrak p)$ applied to 
couples of reductive Lie algebras $\bar{\mathfrak g}\stackrel{i}{\hookrightarrow} \mathfrak g$. 
The analysis is based on
projecting character formulas to quantify the branching, and on
the action of the center of $U(\bar{\mathfrak g})$ to explicitly 
construct singular vectors realizing part of the 
branching. We demonstrate the results on the pair 
$\mathrm{Lie~}G_2\stackrel{i}\hookrightarrow{so(7)}$ for both strongly and weakly compatible 
with $i(\mathrm {Lie~} G_2)$ parabolic subalgebras and 
a large class of inducing representations.
\end{abstract}


  
%%%%%%%%%%%%%%%%%%%%%%%%%%%%%%%%%%%%%%%%%%%%%%%%%%%%%%%%%%%%%%%%%%%%%%%%%%%%%%
% trivial repr. Killing connection
%%%%%%%%%%%%%%%%%%%%%%%%%%%%%%%%%%%%%%%%%%%%%%%%%%%%%%%%%%%%%%%%%%%%%%%%%%%%%%

\section{Introduction}
For an embedding of reductive complex algebraic Lie groups, $\bar G\stackrel{ j}{\hookrightarrow} G$ 
with Lie algebras $\bar \gog$ and $\gog$, the branching rules of 
generalized Verma $\gog$-modules over $\bar \gog$ are a central problem of representation theory, harmonic 
analysis and geometry. 

To our best knowledge, the most general treatment of the problem is given in 
\cite{Kobayashi:branching}, \cite{koss}. The article \cite{Kobayashi:branching} restricts to 
a couple of reductive Lie algebras $(\gog,i(\bar{\gog}))$ and a parabolic 
subalgebra $\gop\subset\gog$ compatible with $i(\bar{\gog})$, and the question of discrete decomposability 
of an element in the Bernstein-Gelfand-Gelfand parabolic category 
${\fam2 O}^\gop$ over $\bar\gog$ is 
achieved by employing
the geometrical properties of the double coset $N_G(i(\bar\gog))\backslash G/P$. 
Here $N_G(i(\bar\gog))$ denotes 
the normalizer of $i(\bar\gog)$ in $G$.

We assume that $\gog$ is a semisimple Lie algebra, $i(\bar\gog)$ is reductive in $\gog$ and 
$i(\bar\gob)\subset\gob\subset\gop$, where $\bar{\gob}$ and $\gob$ are Borel 
subalgebras of respectively $\bar{\gog}$ and $\gog$. 
Let $M_{\lambda}(\gog,\gop)$ be the generalized Verma $\gog$-module induced 
from the irreducible finite dimensional 
$\gop$-module with highest weight $\lambda$.
We define the branching problem of $M_{\lambda}(\gog,\gop)$ over 
$\bar{\gog}$ to be the problem of finding all $\bar\gob$-singular 
vectors in $M_{\lambda}(\gog,\gop)$, that is, the set of 
all vectors annihilated by image of the nilradical of $\bar\gob$
on which the image of the Cartan subalgebra of $\bar\gob$ has diagonal action. 

The branching problem can be split into two closely related sub-problems. 
First, prove that $M_{\lambda}(\gog,\gop)$ has (finite or infinite) Jordan-H\"older series over $\bar{\gog}$, 
enumerate the $\bar\gob$-highest weights $\mu$ appearing in the series, and give their multiplicity as a 
function of $\mu$ and $\lambda$. Second, for each $\mu$, compute explicitly 
a basis for the $\bar\gob$-singular vectors of weight $\mu$. 

In the present article, under certain technical assumptions, 
we reduce the first step of the branching problem to the  problem of understanding 
a single central character block in Category $\mathcal O^{\bar{\gop}}$. 
More precisely, we express the character of $M_{\lambda}(\gog,\gop)$ 
over $\bar\gog$ as a sum of characters of 
$M_\mu(\bar{\gog}, \bar{\gop})$ over $\bar\gog$ and so
reduce the branching rules 
of $M_{\lambda}(\gog,\gop)$ over $\bar\gog$ to 
those of $M_\mu(\bar{\gog}, \bar{\gop})$ over $\bar\gog$. 
We also conjecture that our technical assumptions are not necessary, 
and conjecturally our result on multiplicities holds in full generality.

For the second step of the branching problem, with additional assumptions, we 
show how to find explicitly the highest weight $\bar\gob$-singular vectors at a certain ``top level'' 
$\bar{\gog}$-submodule of $M_{\lambda}(\gog,\gop)$, and give sufficient conditions for this 
``top level'' to equal the entire $M_{\lambda}(\gog,\gop)$. 

We illustrate our results on the pilot example $\LieAlgPair{\LieGtwo}{so(7)}$. Finally, we present
tables with $\bar\gob$-singular vectors and multiplicities for a large set of inducing representations 
for the pair $\LieAlgPair{\LieGtwo}{so(7)}$. 

We recall that for an arbitrary $\gog$-module $M$, the Fernando-Kac subalgebra of 
$\gog$ associated to $M$ is the Lie subalgebra of 
elements that act locally finitely on every vector $v\in M$.
As the Fernando-Kac subalgebra associated to 
$M_{\lambda}(\gog,\gop)$ is $\gop$, it follows that the Fernando-Kac subalgebra of
$\bar\gog$ associated to $M_{\lambda}(\gog,\gop)$ equals $i^{-1}(i(\bar{\gog})\cap \gop)$. 
Then our requirement that $\gop$ contains
the image of a Borel subalgebra of $\bar\gog$
implies the discrete decomposability of $M_{\lambda}(\gog,\gop) $ 
over $i(\bar{\gog})$ (Lemma \ref{leDiscreteDecomposability}).

Our viewpoint of the branching problem differs from that of 
\cite{Kobayashi:branching} in that we assume that $i(\bar{\gob})\subset\gop$.
If we drop the requirement $i(\bar{\gob})\subset\gop$, it appears that
there is no good understanding of 
the simple $\bar{\gog}$-modules with 
Fernando-Kac subalgebras of the form  $i^{-1}(i(\bar{\gog})\cap \gop)$. 
Even more, there appears to be
no complete understanding of the 
structure of the Lie algebra $i^{-1}(i(\bar{\gog})\cap \gop)$.
%It appears  that a formulation of the branching problem of $M_{\lambda}(\gog,\gop)$ over $i(\bar\gog)$
%in the case $i(\bar{\gob})\nsubseteq\gop$ requires 
%further investigation of the preceding questions.
We note that if $\gop$ does not contain an image of the Borel subalgebra of $\bar\gog$, 
we can restrict our attention to a maximal reductive in $\bar\gog$ subalgebra $\bar{\gog}_1$ with the property that it 
has a Borel subalgebra whose image is contained in $\gop$. If $\bar{\gog}_1\neq \{0\}$, 
the branching problem of $M_\lambda(\gog,\gop)$ over $\bar{\gog}_1$ is well-posed, 
and the results of our paper apply to $\LieAlgPair{\bar\gog_1}{\gog}$.

Let $\gop\supset \gob$ be a parabolic subalgebra of $\gog$ and define the parabolic subalgebra 
$\bar{\gop}\subset\bar{\gog}$ by $i(\bar\gop)= i(\bar{\gog})\cap \gop$. 
Let $ M_{\lambda}(\gog,\gop)$ be the generalized Verma $\gog$-module equipped 
with $\bar{\gog}$-module structure induced by $i$. Then one can write 
$ M_{\lambda}(\gog,\gop)$  as a direct limit of certain 
$\bar{\gog}$-submodules $M_n$ %(Lemma \ref{leDiscreteDecomposability}) 
so that the $\bar{\goh}$-character 
$\CharS M_n$ of $M_n$ decomposes as a sum of the form
\[
\CharS M_n:=\sum m_n(\mu,\lambda) \CharS(M_{\mu} (\bar{\gog},\bar\gop))\quad .
\]

Under the technical Condition A given in Definition \ref{defConditionA} below, 
we prove that the coefficients $m_n(\mu,\lambda)$ in the above expression 
have limits $m(\mu,\lambda)$ (allowing $m(\mu,\lambda)=+\infty$). 
We do not know of an example where Condition A fails and we may conjecture 
that it is a consequence of our assumption $\gop\supset \gob \supset i(\bar{\gob})$.

If Condition A holds, in Theorem \ref{thBranchingMultsViaSymmetricTensors} 
we prove that either all non-zero $m(\mu,\lambda)$ are 
simultaneously equal to $+\infty$, or they are all simultaneously finite, 
and in the latter case we give a formula computing them. For our pilot example 
$\LieAlgPair{\LieGtwo}{so(7)}$, Theorem \ref{thBranchingMultsViaSymmetricTensors} 
implies that $\gop_{(1,0,0)}$ is the only proper 
parabolic subalgebra of $so(7)$ with finite branching over $i(\LieGtwo)$, i.e., 
branching for which $m(\mu,\lambda)\neq0$ for only finitely many $\mu$. 
We note that this gives an example in which an infinite 
dimensional $\gog$-module has finite branching over a subalgebra $i(\bar \gog)$ 
of rank strictly smaller than the rank of $\gog$.

On Condition A, in Theorem \ref{thTheBranchingRuleIsQP}, we explain how to 
compute $m(\mu,\lambda)$ as a piecewise quasi-polynomial in the coordinates 
of $\mu$ and $\lambda$ and give an upper bound for the degrees of 
the piecewise quasi-polynomials. In our pilot example $\LieAlgPair{\LieGtwo} {so(7)}$, 
the degree in question is 1. 

We prove that Condition A holds for all parabolic subalgebras $\gop$ 
weakly compatible with $i(\bar{\gog}) $. In particular, Condition A holds 
for a parabolic subalgebra $\gop$ compatible with $i({\bar\gog})$ 
in the sense of \cite[Section 3]{Kobayashi:branching}. 
As we have assumed that $i(\bar{\gob})\subset \gob\subset \gop$, it follows that
all parabolic subalgebras $\gop$ of $so(7)$ are weakly compatible 
with $i(\LieGtwo)$ (of them only 4 are strongly compatible, 
Corollary \ref{corParabolicsB3weaklyCompatibleWithG2}), 
and therefore Condition A holds for all parabolic subalgebras 
relative to the pair $\LieAlgPair{\LieGtwo}{so(7)}$. 

%The numbers $m(\mu,\lambda)$ give a quantitative measure of 
%the dimension of the space of $\bar\gob$-singular vector of a fixed weight $\mu$.
In Section \ref{secConstructingSingularVectors} we discuss the problem of 
explicitly constructing $\bar{\gob}$-singular vectors in $M_\lambda(\gog,\gop)$.
Under additional assumptions given in Theorem \ref{thSufficientlyGenericWeightEnablesFDbranching}, 
we use the Harish-Chandra isomorphism theorem and the corresponding
elements in the center of $U(\bar{\gog})$ to find a certain a set of 
singular vectors that realize the ``top level'' (see \eqref{eqTopLevel}) of the branching problem. 
For a fixed dimension of the inducing finite dimensional $\gop$-module,
our assumptions on $\lambda$ exclude a certain Zariski-closed subset of $\goh^*$, 
however allowing the Zariski-closed subset to be the entire $\goh^*$. 
Corollary \ref{corGenVermaDecompo} gives a sufficient criteria for $M_\lambda(\gog,\gop)$ to 
decompose as a direct sum of (in general, reducible)
generalized Verma $\bar{\gog}$-modules.

In our pilot example $\LieAlgPair{\LieGtwo}{so(7)}$, in the case of the 
parabolic subalgebra $\gop_{(1,0,0)}$, the technical assumptions on $\lambda$ 
exclude finitely many values for a fixed dimension of the inducing $\gop_{(1,0,0)}$-module. 
Except for these values, in Theorem \ref{leBranchingExplicit(x_1,0,0),(1,a,b)} we 
give explicit bases of the $\bar\gob$-singular vectors of 
$M_\lambda(so(7), \gop_{(1,0,0)})$ for which 
$m(\mu,\lambda)\neq0$ for the 6 one-parameter families
$\lambda=x_1\omega_1$, $x_1\omega_1+\omega_2$,  $x_1\omega_1+\omega_3$,  
$x_1\omega_1+2\omega_2$,  $x_1\omega_1+\omega_1+\omega_2$,  $x_1\omega_1+2\omega_3$.
In Corollary \ref{corGenVermaDecomposB3G2}, we prove that, except for the (explicitly computed)
exceptional values for $x_1$ from the preceding Theorem, at least 5 out of the 6 families
of $\bar\gob$-singular vectors give a decomposition of
$M_{\lambda}(so(7), \gop_{(1,0,0)})$ as a direct sum of generalized Verma modules 
$M_\mu(\LieGtwo,\bar\gop_{(1,0)})$.

\onlineVersionOnly{
In Tables \ref{tableB3fdsOverG2charsAndHWV}-\ref{tableB3fdsOverG2charsAndHWV(0, 1, 1)},
for $\gop\neq \gop_{(1,0,0)}$, we tabulate the ``top level'' sets of 
$\bar\gob$-singular vectors in $M_\lambda(so(7), \gop)$ for certain 
one or two-parameter families of the highest weights $\lambda$. 
For the finite dimensional branching $\gop\simeq so(7)$, Table \ref{tableB3fdsOverG2charsAndHWV}
gives the vectors realizing the branching of all irreducible finite dimensional 
$so(7)$-modules over $\LieGtwo$ for highest weight with fundamental 
coordinate sum less than or equal to 2.
}

A geometric motivation of the present study may be given as follows.
Let $G, \bar G$ be the connected and simply connected Lie groups with 
Lie algebras $\gog, \bar{\gog}$.
Let $P$ be the parabolic subgroup of $G$ with Lie algebra $\gop$.
Then there is a well-known
equivalence between invariant differential operators acting on induced
representations and homomorphisms of generalized Verma modules,
realized by the natural pairing
\begin{eqnarray}\label{eqInvriantDiffOpActing}
Ind^G_P(\mathbb V_\lambda(L)^*)\times M_{\lambda}(\gog,\gop)\to \mC,
\end{eqnarray}
where $\mathbb V_\lambda(L)$ denotes the finite-dimensional irreducible $L$-module 
obtained by exponentiating $V_\lambda(\gol)$, $\mathbb V_\lambda(L)^*$ is its dual, 
and $Ind^G_P$ denotes induction. As a consequence, the singular vectors
constructed in the article induce invariant differential operators
acting between induced representations of $j(\bar{G})$.
It is quite interesting to construct these invariant differential
operators, in particular their curved extensions as lifts
to homomorphisms of corresponding semiholonomic generalized Verma modules.
More information on the relationship between conformal geometry in
dimension $5$ and filtered structure on the tangent space (i.e., a distribution) 
related to exceptional Lie group $G_2$, 
can be found in \cite{GrahamWillse:ParallelTractor} and references therein.

We end this section by proposing a series of 
(open to our best knowledge) sub-problems that 
could eventually lead to a solution of the branching problem. 
The central $\bar{\gog}$-characters 
(the constants by which the center of $U(\bar\gog)$ acts on cyclic $\bar\gog$-modules) 
that appear in $\bar\gog$-submodules of $M_\lambda(\gog,\gop)$ 
that afford a central character
are  exhausted by the central $\bar\gog$-characters
of $M_\mu(\bar\gog,\bar{\gop})$ for which $m(\mu,\lambda)>0$. 
Therefore, one needs to first find 
all $\bar\gob$-singular vectors $v_\mu$ of $\bar \goh$-weight $\mu$
in $M_\lambda(\gog,\gop)$, for which $m(\mu,\lambda)>0$.
In other words, one needs to extend the results of the current paper 
beyond the ``top-level'' 
branching given in Theorem \ref{thSufficientlyGenericWeightEnablesFDbranching}, and 
in addition explore what happens when the Zariski-open Condition B fails.
To achieve this task, it appears most natural to  use iterated translation functors. 
Here, by translation functors 
we broadly mean the decomposition into indecomposables of the $\bar\gog$-module 
$M_\mu(\bar{\gog}, \bar{\gop})\otimes \gog/\bar{\gog}$ 
as a function of the coordinates of the highest weight $\mu$.

Next, one must solve a problem involving a single central character block in the Bernstein-Gelfand-Gelfand category 
$\mathcal O^{\bar{\gop}}$ - namely, find all $ \bar\gob$-singular vectors of 
$M_{\mu}(\bar{\gog},\bar{\gop})$.
In general, $M_\mu(\bar{\gog},\bar\gop)$ has $\bar{\gob}$-singular vectors other than the highest one, 
whose weight is of the form $w(\mu+\rho_{\bar{\gog}})-\rho_{\bar{\gog}}$, 
where $w$ is in the Weyl group of $\bar{\gog}$.
In the case of $\bar{\gop}=\bar\gob$, the multiplicities of these vectors 
are 0 or 1 by \cite[Chapter 7]{Dixmier}; for a general $\bar\gop$ one could
use Kazhdan-Lusztig polynomials to compute the multiplicities in question.
Next, one should compute the $\bar\gob$-singular vectors in $M_\mu(\bar{\gog},\bar\gop)$  explicitly.
For $\bar{\gop}=\bar\gob$, an algorithm for this has been described in 
\cite[Chapter 4]{HumphreysNewBook}.
For a general $\bar\gop$, one can use the approach of analysis of 
generalized Verma modules developed in \cite{koss}, 
based on distribution Fourier transform. 
Here the analysis should be a lot more involved as one must account 
both standard and non-standard
generalized Verma module homomorphisms (see \cite{Lepowski:GeneralizationBGG}).
Non-standard homomorphisms have been classified only 
for parabolic subalgebras with commutative nilradicals (\cite{Boe}) and maximal
parabolic subalgebras (\cite{Mat}), both cases just for scalar generalized Verma modules. 

In addition to the branching of the 
generalized Verma module $M_\mu(\bar{\gog},\bar{\gop})$ over $\bar{\gog}$,
one has to solve one final problem arising from the embedding $\LieAlgPair{\bar{\gog}}{\gog} $.
More precisely, given a $\bar\gob$-singular vector $v_\mu$ of $\bar{\goh}$-weight $\mu$ in $M_\lambda(\gog,\gop)$, 
one needs to decide whether $v_\mu$ generates a $\bar\gog$-submodule isomorphic to  $M_{\mu}(\bar\gog,\bar\gop)$, 
or a proper quotient module (if the latter is the case, 
one also needs tools to understand which quotients are generated).

%%%%%%%%%%%%%%%%%%%%%%%%%%%%%%%%%%%%%%%%%%%%%%%%%%%%%%%%%%%%%%%%%%%%%%%%%%%%%%%%%%%%%%%%%%%%%%%%%%%%%%%%%%%%%%%%%%%%%%%%%%%%

\section{Notation and preliminaries}\label{secNotationPreliminaries}

We fix the base field to be $\mathbb C$. For an arbitrary Lie algebra $\got$ we denote by $U(\got)$  
its universal enveloping 
algebra and by $\ad$ we denote the adjoint action in $\got$ 
and $U(\got)$. Given a $\got$-module $M$, we say that a vector $v\in M$ is 
$\got$-singular if the one-dimensional vector space $\mathbb C v $ is 
preserved by $\got$ (i.e., if $\mathbb C v$ is a one-dimensional 
$\got$-submodule of $M$). 

For a semisimple Lie algebra $\gog$, the quadratic Casimir element $c_1\in U(\gog)$ 
is defined as $\sum_i e_i f_i $, 
where $\{e_i\}, \{f_i\} $ are two bases of $\gog$, dual with respect to the 
Killing form $K(e,f):= \Tr(\ad e \circ \ad f)$. 

Let $\gog$ and $\bar \gog$ be reductive Lie algebras. 
\onlineVersionOnly{
We recall that a subalgebra 
$\gog'$ of $\gog$ is called reductive in $\gog$ if any irreducible finite-dimensional 
$\gog$-module is a semisimple $\gog'$-module (i.e., decomposes as a direct sum 
of irreducible $\gog'$-modules). 
}
Let $i$ be an embedding $\LieAlgPair{\bar \gog}{\gog}$ 
such that $ i(\bar\gog)$ is reductive in $\gog$. Let $\bar \goh$ be a 
Cartan subalgebra of $\bar \gog$. Since the base field is $\mathbb C$ 
and $i(\bar\gog)$ is reductive in $\gog$, we can extend $i(\bar\goh)$ 
to a Cartan subalgebra $\goh\supset i(\bar\goh)$ of $\gog$.

We denote by $\Delta(\gog)$ the root system of 
$\gog$ with respect to $ \goh$ and by 
$\Delta(\bar \gog)$ the root system of $\bar\gog$ with respect to $\bar\goh$.
Then $\gog$ and $\bar\gog$ have the vector space decompositions
\[
\gog=\goh \oplus\bigoplus_{\beta\in \Delta( \gog)}\gog_\beta, \quad \quad 
\bar\gog=\bar\goh\oplus \bigoplus_{\alpha\in \Delta(\bar \gog)}\bar\gog_\alpha\quad ,
\]
where  $\gog_\beta$ (respectively, $\bar \gog_\alpha$) are the $\Delta(\gog)$- 
(respectively, $\Delta(\bar\gog)$-) root spaces. 
The Killing form on $\gog$ (respectively $\bar \gog$) 
induces a non-degenerate  symmetric bilinear form on $\goh^*$ (respectively $\bar \goh^* $); 
we rescale this form on each simple component of $\gog$ (respectively, $[\bar \gog, \bar \gog]$) 
so that the long roots have length $\sqrt{2}$ for types $A_n$, $B_n$, $D_n$, $E_6, E_7, E_8$, 
length 2 for types $C_n$ and $F_4$, and length $\sqrt{6}$ for type $G_2$. 
We denote the so obtained bilinear form on $\goh^*$ (respectively, $\bar \goh^*$) 
by $\scalarLA{\bullet}{\bullet} $ (respectively, $\scalarSA{\bullet}{\bullet}$). 
For any weight $\beta\in \goh^*$ (respectively $\alpha\in \bar\goh^*$) we define 
$ h_\beta\in \goh$ (respectively $\bar h_\alpha\in \bar \goh$) to be the dual 
element of $\beta$, i.e., we require that 
$[h_\beta,  g_\gamma]= \scalarLA{\beta}{\gamma} g_\gamma$ 
(respectively  $ [\bar h_\alpha, \bar g_\gamma]=\scalarSA{\alpha}{\gamma} \bar g_\gamma$ ) 
for all $ g_\gamma\in \gog_\gamma$ (respectively  $\bar g_\gamma\in \bar \gog_\gamma$).

We define an embedding
\[
\iota: \bar\goh^*\hookrightarrow \goh^*
\] 
by requiring 
$$
[ i(\bar h_{\gamma}), g_{\beta}]=\scalarLA{\iota(\gamma)}{\beta } g_{\beta}
$$ 
for all $\beta\in\Delta(\gog)$ and all $\gamma\in\bar \goh^*$. 
\onlineVersionOnly{
There exists a positive integer $D$ such that  
%\begin{equation}\label{eqDynIndex}
$
D\scalarSA{\alpha_1}{\alpha_2}=\scalarLA{\iota(\alpha_1)}{\iota(\alpha_2)} 
$
%\end{equation}
for all $\alpha_1, \alpha_2\in \goh^*$; the integer $D$ is called 
Dynkin index of the embedding $\LieAlgPair{\bar{\gog}}{\gog}$, \cite{Dynkin:Semisimple}.
} Let \[\pr:\mathfrak h^*\to \bar\goh^*\] 
denote the canonical projection map naturally induced by $i$.

\onlineVersionOnly{
 We say that a $\gog$-module  $M$ is a $\goh$-weight module if it has a vector space decomposition 
into countably many one-dimensional  $\goh$-submodules. 
All $\gog$-modules discussed in the current paper 
are $\goh$-weight modules, in particular so is every finite dimensional completely 
reducible $\gog$-module as well as every generalized Verma $\gog$-module. 
}
Given a  a $\goh$-weight $\gog$-module $M$, denote by $\weights_\goh(M)$ the ordered
$(\dim M)$-tuple of elements of $\goh^*$ consisting 
of the $\goh$-weights of $M$ with multiplicities. 
\onlineVersionOnly{
In the case that $\dim M<\infty$, we order weights in increasing graded lexicographic 
order with respect to coordinates given in a simple basis of the ambient Lie algebra. 
}
Define $\supp_\goh M$ to be the set of $\goh$-weights of $M$. 

Given a $\gog$-module $M$ with action $\cdot$, we equip $M$ with a $\bar\gog$-module 
structure by requesting that an element $\bar g\in \bar \gog$ act on an element $m\in M$ 
by $ i(\bar g)\cdot m $. In particular, $\gog$ is a $\bar \gog$ module under 
the $\ad\circ i$ action of $\bar\gog$.
Given a $\gog$-module $M$ which as a $\bar\gog$-module is a $\bar{\goh}$-weight module, 
the definitions of $\bar\goh$ and $\goh$ imply the following relation:
\[
\beta\in \weights_{\goh} M \Rightarrow \pr \beta\in \weights_{\bar{\goh}} M\quad .
\]

For each $\alpha\in \Delta(\bar\gog)$ and $\beta\in \Delta(\gog)$, 
we fix a Chevalley-Weyl basis $\bar g_\alpha\in \bar \gog_\alpha$ and $g_\beta\in \gog_\beta$. 
\offlineVersionOnly{ 
In a Chevalley-Weyl basis
the structure constants of $\gog$ and  $\bar{\gog}$ are integral; 
their explicit values will be needed for our pilot example $\LieAlgPair{\LieGtwo}{so(7)}$. 
} %offlineVersionOnly
\onlineVersionOnly{
In a Chevalley-Weyl basis we have that there exist integer structure constants
$\bar n_{\alpha_1,\alpha_2}\in \mathbb Z$, $n_{\beta_1,\beta_2}\in \mathbb Z$, 
for which
\begin{equation}\label{eqChevalley-WeylBasis}
\begin{array}{lcllcl}~
[\bar g_{\alpha_1},\bar g_{\alpha_2}]&=&\bar n_{\alpha_1,\alpha_2}\bar g_{\alpha_1+\alpha_2}, &  [g_{\beta_1},g_{\beta_2}]&=&n_{\beta_1,\beta_2}g_{\beta_1+\beta_2}  \\~
[\bar g_{\alpha_1},\bar g_{-\alpha_1}]&=&\frac{2}{\scalarSA{\alpha_1}{\alpha_1}} \bar h_{\alpha_1}, & 
[ g_{\beta_1}, g_{-\beta_1}]&=&\frac{2}{\scalarLA{\beta_1}{\beta_1}} h_{\beta_1}\quad ,
\end{array}
\end{equation}
where we define $\bar g_{\mu}:=0$, $g_\mu:=0$ whenever $\mu$ is not a root 
of the respective root system. 
We note that the relations (\ref{eqChevalley-WeylBasis}) imply that the elements 
$[g_{\alpha}, g_{-\alpha}], g_\alpha, g_{-\alpha}$ form a standard $h,e,f$-triple 
parametrization of the $sl(2)$-subalgebra they generate.
} %onlineVersionOnly

Let $\bar \gob \supset \bar\goh$ be a Borel subalgebra of $\bar\gog$. 
As the base field is $\mathbb C$, $i(\bar{\gob})$ can be extended to a Borel 
subalgebra $\gob\supset\goh$ of $\gog$. In the present paper we 
assume one such choice of Borel subalgebras $\bar\gob \supset\bar \goh $, 
$\gob\supset \goh\supset i(\bar{\goh})$, $i(\bar{\gob})\subset \gob$ to be fixed. 
All parabolic subalgebras $\gop$ of $\gog$ are assumed to contain $\gob$, 
and similarly, all parabolic subalgebras $\bar\gop$ of $\bar{\gog}$ are assumed to contain $\bar \gob$.

We set $\gon_-$ to be opposite nilradical of $\gon$, i.e., the Lie 
subalgebra generated by the root spaces opposite to the root spaces lying in $\gon$. 
Set $\bar{\gop}$ to be the preimage of the intersection of $\gop$ with $i(\bar\gog)$, i.e., set 
\[
\bar \gop:= i^{-1} (i(\bar\gog)\cap \gop)\quad .
\]
As $\bar \gop$ contains $\bar{\gob}$, it is a parabolic subalgebra of $\bar \gog$. 
Set $\bar \gon$ to be the nilradical of $\bar \gop$ and $\bar\gon_-$ to be its opposite nilradical.

If $\gog$ is of rank $n$, we parametrize the subsets of its positive 
simple roots by $n$-tuples of $0$'s and $1$'s. A parabolic subalgebra 
$\gop$ containing $\gob$ is parametrized by indicating which positive 
simple roots of $\gog$ are weights of $\gon$ (those roots are also called 
``crossed roots''). For example, the parabolic subalgebra $\gop_{(1,0,0)}$ 
of $so(7)$ stands for the parabolic subalgebra in which the first simple 
root is crossed out. %, i.e. $\gop_{(1,0,0)}$ has reductive Levi part $\goh+so(5)$ and 5-dimensional nilradical whose root spaces are $\varepsilon_1-\varepsilon_2$, $\varepsilon_1-\varepsilon_3$, $\varepsilon$,  $\varepsilon_1+\varepsilon_3$, $\varepsilon_1+\varepsilon_2$.

Define $ \gol$ to be the reductive Levi part of $ \gop$ that contains 
$\goh$ and similarly define $\bar{\gol}$ to be the reductive part of $\bar{\gop}$ 
that contains $\bar{\goh}$. The fact that $i(\bar \goh)\subset \goh$ implies that 
$i(\bar{\gol})=\gol\cap  i(\bar\gog)$. As $i(\bar\gol)\subset \gol$, it follows 
that $i(\bar \gog)\cap\gon$ is a $\bar \gol$-submodule. We claim that $\gon_-$ 
splits as the direct sum of the $\bar \gol$-module $(\gon_-\cap i(\bar\gog))$ 
and its complement submodule isomorphic to $\gon_-/(\gon_-\cap i(\bar\gog))$ - 
indeed, this follows as $\bar \gol $ is reductive in $\bar \gog$, which is by 
definition reductive in $\gog$.

We set $\gos:=[\gol,\gol]$.  
\onlineVersionOnly{
We say that a weight $\lambda\in \goh^*$ is integral with respect 
to $\gos$ and  dominant with respect to $\gos\cap\gob$  if  $\scalarLA{\lambda}{\beta}$ 
is a positive integer for all positive roots $\beta$ of $\gol$. 
}
Given a weight $\lambda\in \goh^*$ that is integral with respect 
to $\gos$ and dominant with respect to $\gos\cap\gob$, we denote by $V_\lambda (\gol)$ 
the irreducible finite dimensional $\gol$-module of highest weight 
$\lambda$.  $V_\lambda (\gol)$ can be regarded as a $\gop$-module with trivial action of $\gon$. We 
note that as a partial case of our notation, for $\gop\simeq \gog$, 
$V_\lambda(\gog)$ denotes the irreducible finite 
dimensional $\gog$-module of highest weight $\lambda$. In a similar fashion 
we define $\bar{\gos}:=[\bar\gol,\bar{\gol}]$, the irreducible $\bar{\gop}$-module 
$V_\mu(\bar{\gol})$ and the irreducible $\bar{\gog}$-module $V_\mu(\bar{\gog})$.

In order to compute in $V_\lambda(\gog)$, 
for an arbitrary reductive Lie algebra $\gog$,
we need to realize the following steps.
\begin{enumerate}
\item Produce a set of words $u_1,\dots, u_k\in U(\gog)$ such that 
$u_1\cdot v_\lambda,\dots, u_k\cdot v_\lambda$ give a basis of 
$V_\lambda(\gog)$. Let $m_i:=u_i\cdot v_\lambda$. 
\item For each simple generator $g_\alpha\in \gog$, compute the matrix of the action 
of $g_\alpha$ on the $m_i$'s, i.e., compute the numbers $b_{is}$ for which $g_\alpha\cdot m_i=\sum_s b_{is}m_s$.
\end{enumerate}
In order to deal with 1), we use a non-commutative 
monomial basis of $V_\lambda(\gog)$ as described in 
\cite[\S 1]{Littelmann:ConesCrystalsPatterns}. 
\onlineVersionOnly{
More precisely, the elements $u_i$ are chosen to be
non-commutative (non-PBW ordered) monomials in the simple Chevalley-Weyl 
generators, whose exponents are given by the set of adapted strings $S^{\lambda}_w$ 
described in the 
Definition after \cite[Lemma 1.3]{Littelmann:ConesCrystalsPatterns}.
} %onlineVersionOnly
In order to solve 2), we use the Shapovalov form, 
\cite{Shapovalov}, \cite{JacksonMilev:ShapovalovForm}.
The algorithm we used for solving 1) and 2) is described in 
\cite{JacksonMilev:ShapovalovForm}. The 
explicit computations in Section \ref{secG2inB3} were carried out with the help of a C++ 
program written for the purpose, \cite{Milev:vpf}, 
and we omit all computation details.

For $\lambda\in \goh^*$ integral with respect to $\gos$ and dominant with respect to $\gos\cap\gob$, define 
\[
M_\lambda(\gog,\gop):=U(\gog)\otimes_{U(\gop)} V_\lambda(\gol)
\] 
to  be the generalized Verma $\gog$-module induced from $V_\lambda(\gol)$. 
Similarly define $M_\mu(\bar \gog,\bar \gop):=U(\bar \gog)\otimes_{U(\bar \gop)} V_\mu(\bar \gol)$ 
to be the generalized Verma $\bar{\gog}$-module induced from $V_\lambda(\bar\gol)$.
We have the vector space isomorphism
\begin{equation}\label{eqGenVermaIsUnTimesV}
M_\lambda(\gog,\gop) \simeq U(\gon_-) \otimes V_{\lambda}(\gol)\quad .
\end{equation}

\onlineVersionOnly{ Given a vector space $V$, by 
$S^\star(V):=\bigoplus_{i=0}^\infty S^i(V)$ we denote the 
symmetric tensor algebra of $V$, where $S^i(V)$ is generated over $\mC$
by monomials of the form $\sum_{\sigma\in S_i} v_{\sigma(1)} \otimes\dots \otimes v_{\sigma(i)}$. 
%, where the sum varies over the set of all $i$-element permutations and $m_1,\dots, m_i\in V$ are arbitrary vectors.
}

\offlineVersionOnly{ We denote by $S^\star(V)$ the symmetric tensor algebra of $V$.
}

\onlineVersionOnly{ We now fix notation for the Weyl character formula. 
}
We denote the Weyl group of a reductive Lie algebra $\got$ by $\mathbf W(\got)$.
To each element $\beta\in \goh^*$ we assign the 
formal exponent $x^{\beta}$ with multiplication $x^{\beta_1}x^{\beta_2}=x^{\beta_2+\beta_2}$ and denote 
by $\mathbb Q[x^\gamma]_{\gamma\in \goh^*}$  the commutative associative $\mathbb Q$-algebra generated 
by all such elements with no further relations. 
We define $\mathbb Q[[x^\gamma]]_{\gamma\in \goh^*}$  as 
the vector space generated by formal infinite $\mathbb Q$-linear 
combinations of the monomials $x^\alpha$.
%Alternatively, to an arbitrary function $p:\goh^*\to \mathbb Q$ we can assign the formal linear
%combination $\displaystyle P:=\sum_{\gamma\in\goh^*} p(\gamma)x^{\gamma}$. 
%In this way we may regard the elements of $\mathbb Q[[x^\gamma]]_{\gamma\in \goh^*}$ 
%as functions $\goh^*\to \mathbb Q$.
Similarly, to each element $\alpha\in\bar \goh^*$ we assign 
the formal exponent $y^{\alpha}$ with multiplication 
$y^{\alpha_1} y^{\alpha_2}=y^{\alpha_1+\alpha_2}$ and denote 
by $\mathbb Q[y^\gamma]_{\gamma\in \bar \goh^*}$ 
the commutative associative $\mathbb Q$-algebra generated by all 
such elements with no further relations.

We define the $\mathbb Q$-algebra homomorphism 
\begin{eqnarray}
 \Pr: \mathbb Q[x^\gamma]_{\gamma\in \goh^*}&\to &\mathbb Q[y^\gamma]_{\gamma\in \bar \goh^*}\quad ,
\nonumber \\
 x^{\gamma} &\mapsto & y^{\pr(\gamma)}\quad .
\end{eqnarray}

Given a weight $\gog$-module $M$ with finite-dimensional 
$\goh$-weight spaces $M({\beta})$, the character of 
$\CharL M\in \mathbb Q[[x^\gamma]]_{\gamma\in \goh^*}$ is defined via 
\[
\CharL M:=\sum_{\beta\in \weights_{\goh} M} x^{\beta} = \sum_{\beta\in \supp_{\goh} M} \dim M({\beta})  x^{\beta}\quad    .
\]
Similarly, for a $\bar \gog$-module $N$ we define $\CharS$ 
to be the character with respect to $\bar\goh$, 
\[
\CharS N :=\sum_{\alpha\in \weights_{\bar \goh} N} y^\alpha\quad .
\]

\begin{definition}
Let $a:=\sum_{\gamma\in \goh^*}a(\gamma)x^\gamma\in  \mathbb Q[[x^\gamma]]_{\gamma\in\goh^*}$. 
We say that $a$ is $\pr$-finite if for any weight $\delta \in \bar\goh^*$ the set 
$\{\gamma\in \goh^*|\pr(\gamma)=\delta \mathrm{~and~} a(\gamma)\neq 0\}$ is finite for every $\delta$. %This condition can be equivalently characterized as finiteness of fibers of the set theoretical map $\Pr$. 
\end{definition}
The set of $\pr$-finite elements $\PrFin$ of $\mathbb Q[[x^\gamma]]_{\gamma\in\goh^*}$ 
is a vector subspace of $\mathbb Q[[x^\gamma]]_{\gamma\in\goh^*}$ 
containing $\mathbb Q[x^\gamma]_{\gamma\in\goh^*}$. We can extend the 
map $\Pr:  \mathbb Q[x^\gamma]_{\gamma\in \goh^*}\to \mathbb Q[y^\gamma]_{\gamma\in \bar \goh^*}$ 
to a map $\Pr:\PrFin\to \mathbb Q[[y^\gamma]]_{\gamma\in \bar \goh^*}$ via
\begin{equation}\label{eqApplyPr} 
\Pr(a) :=\sum_{\delta\in \bar\goh} \left(\sum_{\pr(\gamma)=\delta}a(\gamma) \right)x^{\delta}\quad .
\end{equation}
We note that if $\CharL(M)$ is $\pr$-finite, then we have the equality
\begin{equation}\label{eqApplyPrOnChars} 
\Pr(\CharL (M)) =\CharS (M)\quad .
\end{equation}

%%%%%%%%%%%%%%%%%%%%%%%%%%%%%%%%%%%%%%%%%%%%%%%%%%%%%%%%%%%%%%%%%%%%%%%%%%%%%%%%%%

\section[Characters and branching of generalized Verma modules]
{Character formulas and branching laws for generalized Verma modules}\label{secWeylCharacterFormulas}

We recall that a finitely generated $\bar{\gog}$-module is of finite length 
if it has finite Jordan-H\"older composition series over $\bar{\gog}$. 
We recall the following definition, \cite{Kobayashi:branching}.
\begin{definition}\label{defDiscretelyDecomposable}
A $\bar\gog$-module $M$ is discretely decomposable if there is an increasing filtration 
$M_m$ by $\bar\gog$-submodules of finite length such that
$\bigcup_n M_n=M$. Let $\bar\gop$ be a parabolic subalgebra of $\bar\gog$. We say that $M$ is 
discretely decomposable in Bernstein-Gelfand-Gelfand parabolic category $\mathcal O^{\bar\gop}$ 
if in addition all modules 
$M_n$ from the filtration can be chosen to lie in $\mathcal O^{\bar\gop}$.
\end{definition}

Definition \ref{defDiscretelyDecomposable} implies the following.
\begin{lemma}\label{leDiscreteDecomposability}
The generalized Verma module $M_\lambda(\gog,\gop)$ is discretely decomposable as a $\bar \gog$-module in the category $\mathcal O^{\bar\gop}$.
\end{lemma}
\begin{proof}
Let $v$ be the $\gob$-highest weight vector of $M_\lambda(\gog,\gop)$. Let 
$\gog\simeq U_1\subset U_2\subset\dots\subset U_n\subset\dots$ be the standard filtration 
of $U(\gog)$ by degree. Let $M_n$ be the $\bar\gog$-module generated by the vector space 
$U_n\cdot v$. It is a straightforward check that $M_n$ belongs to category $\mathcal O^{\bar\gop}$. 
The so constructed modules $M_n$ provide the filtration required by Definition 
\ref{defDiscretelyDecomposable}.
\end{proof}

In \cite[Section 3]{Kobayashi:branching}, the branching laws of 
$M_\lambda(\gog,\gop)$ over $\bar\gog$ are explored under the additional 
assumption that $\gop$ is compatible with $i(\bar\gog)$, i.e., that $\gop$ can be defined using a 
hyperbolic grading element that lies in $i(\bar\goh)$. If $\gop$ is a parabolic subalgebra compatible with $i(\bar\gog)$ 
we have that $i(\bar\gon)  \subset \gon$. Without the assumption that $\gop$ is $i(\bar\gog)$-compatible this no longer has 
to hold, see Lemma \ref{leStructureB3overG2}.3.(b),(c). Nevertheless, 
the following Lemma holds.
\begin{lemma}\label{leLargeNilradicalHasLsubmoduleIsoToSmallerNilradical}
~
\begin{enumerate}
\item
The map $\pr$ maps $\supp_{\goh} \gon$ surjectively onto $\supp_{\bar \goh} \bar\gon$.
\item
$\gon$ has an $\bar \gol$-submodule isomorphic to $\bar \gon$.
\end{enumerate}
Both claims of the Lemma remain true if we interchange $\gon$ with $\gon_-$ and $\bar\gon$ with $\bar \gon_-$.
\end{lemma}
\noindent\begin{proof}
%\begin{enumerate}
%\item
1. Consider a $\bar \goh$-weight $\gamma$ of $\bar \gon$. %Let the corresponding $\bar \goh$-weight space 
%in $\bar \gog$ be $\bar \gog_{\gamma}$, let its opposite $\bar \goh$-weight space be $\bar \gog_{-\gamma}$. 
Then $i(\bar g_\gamma)$ is a $\mathbb C$-linear combination 
of elements of the 
weight spaces $\gog_\beta$ for which $\pr (\beta)=\gamma$. Suppose all $\pr$-preimages of $\gamma$ in $\Delta(\gog)$
are roots of $\gol$. This implies that $i(\bar \gog_{-\gamma})$ is a subspace of $\gol$, 
which implies that $\bar \gog_{\gamma}\subset \bar \gol$. Contradiction. Therefore there exists a weight 
$\beta\in\weights_{\goh}\gon$ satisfying 
$\pr(\beta)=\gamma$. The first part of the Lemma now follows from the fact that the $\bar\goh$-weights of 
$\bar \gon$ have multiplicity one (as the root spaces of $\bar\gog$ have no multiplicities).

%\item
2. Let $\mu\in \bar \goh^*$ be a $\bar\gob\cap\bar \gol$-maximal weight of $\bar \gon$ and let $w$ be a 
$\bar \gob\cap\bar\gol$-singular vector of $\bar\gon$ of weight $\mu$ under the adjoint action of $\bar\gol$ on $\bar{\gon}$. 
Then there exist complex numbers $a_{\beta}$, for $\beta\in \Delta(\gog)$ and $\pr (\beta)=\mu$, such that 
\begin{equation}\label{eqNumbersAbeta}
i(w)=\sum_{\substack{\beta\in\Delta(\gog) \\ \pr(\beta)=\mu}}a_\beta g_{\beta}\quad . 
\end{equation}
For a fixed $\mu$, let
\begin{eqnarray}
& & I_1:=\{\beta\in \supp_{\goh} (\gon) | \pr (\beta)=\mu\}\quad, 
\nonumber \\
& & I_2:=\{\beta\in \supp_{\goh} (\gol) | \pr (\beta)=\mu\}\quad.
\end{eqnarray}
The disjoint union $I_1\sqcup I_2$ equals $\{\beta\in \Delta(\gog)|\pr (\beta)=\mu\}$. 
Consider the element 
\[
v:=\sum_{{\beta\in I_1}}a_\beta g_{\beta} \quad . 
\]
We claim that $v$ is non-zero. Indeed, otherwise we would have that $i(w)\in\gol$ 
and consequently $i(w)\in i(\bar\gol)$, which is impossible. 

We claim next that $v$ is $\bar \gob \cap\bar \gol$-singular. Indeed,  let $\bar\gom$ be the nilradical of $\bar\gob$. 
As $\gon$ is a $\bar \gol$-module under the $\ad \circ i$ action (as it is a $\gol$-module), 
we have that $[i(\bar\gom\cap\bar \gol) ,v]\subset\gon$. On the other hand, 
$[ i(\bar \gom \cap\bar\gol),i(w)]=0$ and therefore 
$[i(\bar\gom\cap\bar \gol) ,v]= [i(\bar\gom\cap\bar \gol) ,v-i(w)+i(w)]=[i(\bar\gom\cap\bar \gol) ,v-i(w)]\subset \gol$, 
where the latter follows as $v-i(w)\in \gol$. 
Thus $[i(\bar\gom\cap\bar \gol) ,v]\in \gon\cap\gol=\{0\}$. The latter implies that
$v$ is $\bar\gob\cap\bar\gol$-singular as  $i(\bar{\goh})\subset{\goh}$.

For each $\bar\gob\cap\bar\gol$-maximal weight of $\bar\gon$, the preceding construction 
yields a $\bar \gob\cap\bar\gol$-singular non-zero vector of the same $\bar\goh$-weight. 
As $\bar\gon$ is multiplicity free as a $\bar\gol$-module (as it is multiplicity free 
as a $\bar\goh$-module), the construction yields one set of $\bar\gob\cap\bar\gol$-singular 
vectors under the $\ad \circ i$ action generating an $\bar\gol$-submodule of $\gon $ 
isomorphic to $\bar\gon$.
\end{proof}

\textbf{Remark. } The complex numbers $a_\beta$ in \eqref{eqNumbersAbeta} depend non-trivially on the embedding map $i$. 
It appears that this dependence is not yet fully understood. 
A discussion of the subject can be found in \cite{DeGraaf:ConstructingSSsubalgebras}.

In view of Lemma \ref{leLargeNilradicalHasLsubmoduleIsoToSmallerNilradical}, the 
following definition of the numbers $m(\mu,\lambda)$ is a straightforward 
generalization of the numbers denoted with the same letter 
in \cite[\S 3.4, (3.3)]{Kobayashi:branching}.
\begin{definition}\label{defm(lambda,mu)} 
Let $\lambda\in \goh^*$ be integral with respect to $\gos$ and dominant 
with respect to $\gos\cap\gob$. Let $N$ be any 
$\bar\gol$-submodule of $\gon_-$  
isomorphic to $\bar\gon_-$ (under the $\ad\circ i$-action). We define
\[
m(\mu,\lambda):= \dim \Hom_{\bar\gol} (V_{\mu} (\bar{\gol}), V_{\lambda}(\gol)\otimes S^\star(\gon_-/N)) \quad  .
\]
\end{definition}
As $i(\bar{\gol})$ is reductive in $ \gog$, the $\bar\gol$-module $N$ 
has a complement $\bar \gol$-submodule $Q\simeq \gon_-/N$.
As $S^n(\gon_-)$ is an $\bar\gol$-module, so is  $S^n (Q)$, and we can 
identify the $\bar\gol$-module $S^n(\gon_-/N)$ with  $S^n (Q)$.

\subsection{The cones $\mathcal C$, $\mathcal C'$ }
For an ordered tuple of weights $X\subset \bar\goh^*$, let $\Cone_{\mathbb Z_{>0}}(X)$ 
denote the set of the  possible finite $\mZ_{> 0}$-linear 
 combinations of the elements of $X$. 

The following Condition A will play a crucial role in our ability to write 
closed form character formulas for $M_{\lambda}(\gog,\gop)$.
\begin{definition}\label{defConditionA} [Condition A]
We recall that $i(\bar\gob)\subset\gob\subset\gop$, $\bar{\gop}:=i^{-1}(\gop)$ and 
that $\gon_-$ is the nilradical opposite to the nilradical of $\gop$.
Let $N$ be as in Definition \ref{defm(lambda,mu)}. 
Define the cones $\mathcal C\subset \mathcal  C'$ by
$$
\begin{array}{rcl}
\mathcal C&:=&\Cone_{\mathbb Z_{>0}}(\weights_{\bar\goh} (\gon_-/N))\\
\mathcal C'&:=& \Cone_{\mathbb Z_{>0}}(\weights_{\bar\goh} (\gon_-)) \quad .\\
%\mathcal C''&:=& \Cone_{\mathbb Z_{>0}}(\weights_{\bar\goh} (\gom_-))
\end{array}
$$
We say that $\gop$ satisfies Condition A if exactly one of the two holds.
\begin{enumerate}
\item $0\in \mathcal C$.
\item $0\notin \mathcal C'$.
\end{enumerate}
\end{definition}
The failure of Condition A is equivalent to $0\in \mathcal C'$ and $0\notin \mathcal C$. 
We immediately note that we don't have an example 
for which Condition A fails, and we may conjecture that it always holds.
The $\bar{\gol}$-module $N$ is not uniquely determined,
but its $\bar{\goh}$-weights are (they equal the $\bar{\goh}$-weights of $\bar\gon_-$). 
Consequently the weights of $\gon_-/N$ are obtained by removing the 
$\bar\goh$-weights of $\bar\gon_-$ from the $\bar\goh$-weights of 
$\gon_-$ accounting for multiplicities; thus the cone $\mathcal C$ 
does not depend on the choice of $N$.

\textbf{Example.} As an example of  $0\notin \mathcal C'$, 
we may pick $\gop\simeq \gog$ to be the full parabolic subalgebra 
and $\bar\gog$ to be any reductive subalgebra (then $0\notin C'=\emptyset$). 
As an example of $0\in\mathcal C'$, we can choose subalgebra $i(\bar\gog)$ such that a 
root $\beta$ of $\gog$ projects to the zero $\bar{\goh}$-weight, and chose the parabolic subalgebra $\gop$
such that $\beta$ is not a root of its Levi part. 
An example for which $0\in\mathcal C'$ but none of the roots of 
$\gog$ projects to zero $\bar{\goh}$-weight can be given as follows. Pick $\gog\simeq so(2n+2)$; 
its root system can be given
in $\varepsilon$-notation as $\Delta(so(2n+1)):=\{\pm\varepsilon_i\pm\varepsilon_j | 1\leq i<j\leq n+1 \}$. 
Declare the roots $\{\varepsilon_i\pm\varepsilon_j| i<j\}$ to be positive, and define $\gob$ accordingly.
Pick $i(\bar\gog)$ to be the subalgebra generated by the root spaces 
$\{g_{\pm\varepsilon_i\pm \varepsilon_j}| 2\leq i<j\leq n+1\}$; this subalgebra is isomorphic to $so(2n)$. 
Pick $\gop\supset \gob$ to be any parabolic subalgebra 
for which the root $\varepsilon_1-\varepsilon_2$ is ``crossed out'', i.e.,  
$\varepsilon_1-\varepsilon_2$ is not a root of $\gop$.
Then the root spaces $g_{-\varepsilon_1-\varepsilon_2}$ and $g_{-\varepsilon_1+\varepsilon_2}$  belong to $\gon_-$.
Therefore the weight $0=\pr(-2\varepsilon_1)=\pr(-\varepsilon_1-\varepsilon_2+(-\varepsilon_1+\varepsilon_2))$ 
belongs to $\mathcal C'$, however none of the roots of $so(2n+2)$ projects to zero. 
Another example of $0\in \mathcal C'$ can be given by $\gog\simeq sl(3)$ and
$\bar{\gog}\simeq sl(2)$, with $i( sl(2))$ 
equal to the regular $sl(2)$-subalgebra of $sl(3)$ and $\gop\simeq\gob$. 
However, if we pick $i(sl(2))$ as the principal 
subalgebra of $sl(3)$, then our prerequisite requirement that $i(\bar{\goh})\subset\goh\subset\gob$ 
implies $0\notin \mathcal C'$ no matter which of the 4 parabolic subalgebras $\gop\supset \gob$ we choose.

\begin{definition}\label{defWeaklyCompatible}
We define $\gop$ to be weakly compatible with $i(\bar \gog)$ if 
there exists an element $\bar h \in \bar\goh$ such that 
$\gamma(i(\bar h))\in \mathbb Q$ for all $\gamma\in \Delta(\gog)$ and
\begin{enumerate}
\item $\alpha(\bar h)>0$ for all $\alpha\in \weights_{\bar{\goh}} \bar\gon$,
\item  $\beta(i(\bar h)) \geq 0$ for all $\beta\in \weights_{\goh} \gon$.
\end{enumerate}
\end{definition}
\begin{lemma}\label{leConeEquivalence}
Suppose $\gop$ is weakly compatible with $i(\bar{\gog})$ (Definition \ref{defWeaklyCompatible}). 
Then Condition A holds, i.e.,
\[
0\notin \mathcal C \Leftrightarrow 0\notin \mathcal C' \quad .
\]
\end{lemma}
\begin{proof}
($\Leftarrow$) is trivial; we are proving the other implication. 
$\weights_{\bar\goh}\gon=-\weights_{\bar\goh}\gon_-$ and therefore
there exists an element 
$\bar{h}\in \bar \goh$ for which $\alpha(\bar h)<0 $ for all 
$\alpha\in  \weights_{\bar{\goh}} \bar\gon_-$ 
and for which $\pr(\beta)(\bar h)\leq 0$ for all $\beta\in \weights_{\goh}\gon_-$. 

Suppose 
\[
0=\sum_{\beta\in \weights_{ \goh}{\gon_-} }n_\beta \pr(\beta)
\] 
for some non-negative numbers $n_\beta>0$. Evaluate both sides of the above 
expression on $\bar h$ to obtain that 
$0=\sum_{\beta\in \weights_{ \goh}{\gon_-} } n_\beta \pr(\beta)(\bar h)$, 
where all summands are non-positive and therefore must be zero. 
Let $\gamma$ be a weight with $\pr \gamma\in \weights_{\bar{\goh}} \bar\gon_-$. 
Then $\pr\gamma(\bar h)<0$ and therefore $n_\gamma=0$. 
Therefore $0\in \mathcal C'$ implies that $0\in \mathcal C$, 
which proves the Lemma.
\end{proof}

The notion of a parabolic subalgebra weakly compatible 
with $i(\bar{\gog})$ is a generalization of the notion of a 
compatible parabolic subalgebra, as the following Proposition shows.
\begin{proposition}\label{corCompatibleIsWeaklyCompatible}
Let $\gop$ be a parabolic subalgebra of $\gog$ compatible with $i(\bar\gog)$ 
in the sense of \cite[Section 3]{Kobayashi:branching}. 
Then $\gop$ is weakly compatible with $i(\bar\gog)$ and Condition A holds.
\end{proposition}
\begin{proof}
By the definition of a parabolic subalgebra compatible 
with $i(\bar{\gog})$ there exists an element $\bar h\in \bar\goh$ for 
which $\beta(i(\bar h))>0$ for all $\beta\in \weights_{\goh}(\gon)$. 
Let $\alpha\in \weights_{\bar{\goh}}\gon$. By Lemma 
\ref{leLargeNilradicalHasLsubmoduleIsoToSmallerNilradical} there exists 
$\beta \in   \weights_{\goh}(\gon)$ with $\pr(\beta)=\alpha$. 
Therefore $\alpha(\bar h)= \pr(\beta)(\bar h)=\beta( i(\bar h))>0$.
\end{proof}

We recall that $\mathbb Q[[y^\gamma]]_{\gamma\in \bar{\goh}^*}$ is not a ring, 
as products are not necessarily well defined: 
for example, the product of $\sum_{n=0}^\infty y^{n\alpha}  $ and 
$\sum_{n=0}^\infty y^{-n\alpha}$. However, if $0\notin \mathcal C'$, 
the vector subspace $\mathbb Q[[y^\gamma]]_{\mathcal C'}$ given by
\begin{equation}\label{eqRingQh}
\mathbb Q[[y^\gamma]]_{\mathcal C'} := \sum_{\mu\in \bar\goh^*} 
\{ \sum_{\gamma\in \mu+ \mathcal  C'} a(\gamma)y^\gamma ~ |~ a(\gamma)\in \mathbb Q\} 
\end{equation}
is a ring, where the outer $\sum$ sign stands vector space sum, i.e., the vector space 
generated by finite linear combinations of the elements of the vector spaces under the $\sum $ sign.

%Indeed, the convex hull of $\mathcal C$ does not contain the origin and therefore there exists an element $h\in \bar \goh^*$ with $\gamma(h)<0$ for all $\gamma\in \mathcal C$. Therefore if $\sum_{\gamma\in \mu+ \mathcal  C} a(\gamma)y^\gamma=\in \mathbb Q[[y^\gamma]]_{\mathcal C}$ and for any constant $t\in \mathbb Z$, there are only finitely many non-zero $a(\gamma)$ for which $\gamma(h)<t$. 
Therefore if $0\notin \mathcal C'$ we can write
\begin{equation}\label{eqGeometricSeriesSum0}
\begin{array}{rcl}
\CharS S^\star(\gon_-/N)&=&\displaystyle\prod_{\alpha\in \weights_{\bar\goh} (\gon_-/N)} \frac{1}{(1-y^{\alpha})}\\
\CharS S^\star(\gon_-)&=&\displaystyle\prod_{\alpha\in \weights_{\bar\goh} (\gon_-)} \frac{1}{(1-y^{\alpha})} \quad,
\end{array}
\end{equation}
where $\frac{1}{(1-y^{\alpha})}$ denotes the multiplicative 
inverse of $(1-y^{\alpha})$ in $\mathbb Q[[y^\gamma]]_{\mathcal C'}$, 
given by the geometric series sum formula, and the products are taken in the ring 
$\mathbb Q[[y^\gamma]]_{\mathcal C'}$.

Let $\lambda\in {\goh}^*$ be dominant with respect to $\gob\cap\gos$ 
and integral with respect to $\gos=[\gol,\gol]$, 
where we recall $\gol$ is the reductive Levi part of $\gog$.
As $\bar \gol$ is reductive in $\gog$, 
the $\bar\gol$-module $V_{\lambda}(\gol)$ decomposes as a direct sum of 
irreducible $\bar \gol$-modules. Therefore for $\mu\in \bar{\goh}^*$ we 
can define the numbers $n(\mu,\lambda)$ as the multiplicity of the module 
$V_{\mu} (\bar\gol)$ in $V_{\lambda}(\gol)$. The numbers $n(\mu,\lambda)$ 
are then computed via the finite sum
\begin{equation}\label{eqSplitCharSVlambdagol}
\CharS V_\lambda (\gol)=:\sum_{\mu\in \bar{\goh}^*} n(\mu, \lambda)\CharS V_\mu(\bar \gol)\quad .
\end{equation}

 For an arbitrary sequence of elements 
$P_1,\dots, P_n,\dots~\in Q[[y^\gamma]]_{\gamma\in \bar\goh^*}$,
let the functions $p_n:\bar\goh^*\to \mathbb Q$ be defined via
$P_n:=\sum_{\gamma\in \goh^*} p_{n}(\gamma) y^\gamma $. 
%If the numbers $p_n(\gamma)$ has a limit as $n\to\infty$, as usual we denote it by$\lim_{n\to \infty} p_n(\gamma)$. 
We say that the sequence 
$P_1,\dots, P_n,\dots~\in Q[[y^\gamma]]_{\gamma\in \bar\goh^*}$
has a limit if for all $\gamma\in\bar\goh^*$ the limit 
$\lim_{n\to \infty} p_n(\gamma)$ exists. 
If so, we define $\displaystyle\lim_{n\to \infty} P_n$ by
 $\displaystyle\lim_{n\to \infty} P_n:=\displaystyle\sum_{\gamma\in \goh^*} \left(\lim_{n\to \infty} p_n(\gamma)\right) y^\gamma $.

The following Lemma follows from the definition of $\CharS$ and the geometric 
series sum formula.
\begin{lemma}~\label{leGeometricSeriesSum}
\begin{enumerate}
\item 
 \begin{equation}\label{eqElementaryCharacterIdentity1}
\begin{array}{rcl}
\displaystyle\sum_{\mu\in \bar{\goh^*}} m(\mu,\lambda) \CharS V_\lambda(\bar\gol) 
&=& \displaystyle\lim_{n\to \infty}\left( \CharS V_{\lambda}(\gol) 
\displaystyle\sum_{i=0}^n \CharS (S^n(\gon_-/N))\right) \\
&=& \displaystyle\lim_{n\to \infty}\left( \CharS V_{\lambda}(\gol) 
\displaystyle \prod_{\alpha\in \weights_{\bar\goh} (\gon_-/N)} 
\frac{1-y^{n\alpha}}{1-y^{\alpha}}\right),
\end{array}
\end{equation}
where the division and multiplication operations are carried out 
in the ring $\mathbb Q[y^{\gamma}]_{\gamma\in \goh^*}$, 
we allow $m(\mu,\lambda) = +\infty$, and for $\alpha=0$ we have 
defined $\frac{1-y^{n\alpha}}{1-y^\alpha}_{|\alpha=0}:=n$.

\item Suppose $0\notin \mathcal C$. Then 
\begin{equation}\label{eqElementaryCharacterIdentity2}
\sum_{\mu\in \bar{\goh^*}} m(\mu,\lambda) \CharS V_\lambda(\bar\gol) 
=\CharS V_{\lambda}(\gol)  
\displaystyle\prod_{\alpha\in \weights_{\bar\goh} (\gon_-/N)} 
\frac{1}{1-y^{\alpha}}\quad ,
\end{equation}
where the inverse and multiplication operations are taken 
in the ring $\mathbb Q[[y^{\gamma}]]_{\mathcal C'}$ defined by (\ref{eqRingQh}).
\end{enumerate}
\end{lemma}

\subsection{Character formulas and multiplicity functions $m(\mu,\lambda)$}

In the following Theorem we explore what happens when we apply the 
map $\Pr$ to the character $\CharL M_\lambda(\gog,\gop)$. The observation of 
the Theorem is that we can separate the $\Pr$-images of the denominators 
$(1-x^\beta)$ in the generating function of the character formula according 
to the (possibly zero) multiplicity with which $\pr( \beta)$ belongs to 
$\weights_{\bar \goh} \bar\gon_- $ and $\weights_{\bar{\goh}}\bar{\gon}_-/N$. 

\begin{theorem}\label{thBranchingMultsViaSymmetricTensors}
%Let $\gop$ be weakly compatible with $i(\bar{\gog})$ and 
Let $\lambda\in \goh^*$ be a weight dominant with respect to $\gob\cap\gos$ and 
integral with respect to $\gos$. 
\begin{enumerate}
\item Suppose $0\notin \mathcal C'$. %Then in the Grothendieck group ${\fam 2 O}^{\gop^\prime}$, we have the isomorphism
%\begin{eqnarray}
%M_\lambda(\gog,\gop)|_{\gog^\prime}\simeq \bigoplus_\mu m(\mu, \lambda)M_\mu({\gog^\prime},{\gop^\prime}),
%\end{eqnarray}
%where the sum runs over the $\bar \gol\cap\gob'$-integral dominant weights of $\gol^\prime$. Equivalently,
Then we have that $\CharL M_\lambda(\gog,\gop)$ is $\pr$-finite and
\begin{equation}\label{eqThBranchingViaSymmetricTensors}
\CharS M_\lambda(\gog,\gop)=\Pr (\CharL M_\lambda(\gog,\gop) )=
\sum_{\mu\in{\bar \goh}^*} m(\mu, \lambda)\CharS M_\mu({\bar\gog},{\bar\gop})\quad .
\end{equation}
In particular $m(\mu, \lambda)\in \mathbb Z_{\geq 0}$.

In the language of \cite{Kobayashi:branching}, in the Grothendieck group of ${\mathcal  O}^{\bar\gop}$
of the Bernstein-Gelfand-Gelfand category $K({\mathcal  O}^{\bar\gop})$ there is equality
\[
\CharS M_\lambda(\gog,\gop)=
\bigoplus_{\mu\in{\bar \goh}^*} m(\mu, \lambda)\CharS M_\mu({\bar\gog},{\bar\gop})\quad .
\]
\item If $0\in \mathcal C$, we have that if 
$m(\mu,\lambda)\neq 0$ then $m(\mu,\lambda)= +\infty$.
\end{enumerate}
\end{theorem}

\begin{proof}
1. By (\ref{eqGenVermaIsUnTimesV}) we have that
\begin{eqnarray} \label{eqGenVermaCharr}
\CharL M_{\lambda}(\gog, \gop) %& =& \left(\sum_{\alpha\in \Cone(\Delta(\gon_-))} Q(\alpha)  x^\alpha\right) \CharL V_\lambda \\\notag
&=&\left( \prod_{\beta\in \weights_\goh(\gon_-)} 
\frac{1}{1-x^{\beta}}\right)\CharL V_\lambda(\gol) \quad,
\end{eqnarray}
and similarly 
\begin{eqnarray} \label{eqGenVermaCharSmall}
\CharS M_{\mu}(\bar\gog, \bar\gop) %& =& \left(\sum_{\alpha\in \Cone(\Delta(\gon_-))} Q(\alpha)  x^\alpha\right) \CharL V_\lambda \\\notag
&=&\left( \prod_{\alpha\in \weights_{\bar\goh}(\bar\gon_-)} 
\frac{1}{1-y^{\alpha}}\right)\CharL V_\lambda(\bar\gol) \quad.
\end{eqnarray}
By \eqref{eqGenVermaCharr} and  \eqref{eqGeometricSeriesSum0} 
$\CharL M_\lambda(\gog, \gop)$ is $\pr$-finite. Therefore we can compute
\begin{eqnarray}\notag
& & \Pr\left( \prod_{\beta\in \weights_\goh(\gon_-)} \frac{1}{1-x^{\beta}}\right)= 
\prod_{\beta\in \weights_\goh(\gon_-)} \frac{1}{1-y^{\pr(\beta)}}\\ 
\label{eqSplitNilradicalPowerSeries}
& & =
\prod_{\alpha\in \weights_{\bar\goh}(\gon_-/N)} \frac{1}{1-y^{\alpha}} 
\prod_{\alpha\in \weights_{\bar\goh}(N)}\frac{1}{1-y^{\alpha}}\quad ,
\end{eqnarray}
where the multiplication is taken in the ring $\mathbb Q[[y^\gamma]]_{\mathcal C'}$.
The computation
\begin{eqnarray*}
\Pr \CharL M_\lambda(\gog, \gop)& \stackrel{(\ref{eqGenVermaCharr}), (\ref{eqSplitNilradicalPowerSeries}) }{=} &
\left(\CharS   V_\lambda(\gol)\prod_{\alpha\in \weights_{\bar\goh}(\gon_-/N)} \frac{1}{1-y^{\alpha}}\right)
\prod_{\alpha\in \weights_{\bar\goh}(\bar\gon_-)}\frac{1}{(1-y^{\alpha})} \\
\nonumber \\
& \stackrel{(\ref{eqElementaryCharacterIdentity2}) }{=} & 
\sum_\mu m(\mu,\lambda) \CharS V_\mu(\bar\gol)\prod_{\alpha\in \weights_{\bar\goh}(\bar\gon_-)}
\frac{1}{(1-y^{\alpha})}\\&\stackrel{(\ref{eqGenVermaCharSmall})}{=}& 
\sum_\mu m(\mu,\lambda) \CharS   M_\mu({\bar\gog},{\bar\gop})
\end{eqnarray*}
completes the proof of 1).

2. By (\ref{eqElementaryCharacterIdentity1}), for any positive integer $n$,
we have that $m(\mu,\lambda)$ is greater than or equal to the coefficient of $y^{\mu}$ 
in the expression
\[
\CharS V_\lambda(\gol) \prod_{\alpha\in \weights_{\bar\goh} (\gon_-/N)} (1+y^\alpha+\dots+y^{n\alpha})\quad .
\]
Consider the expression
\[
A(n):=\prod_{\alpha\in \weights_{\bar\goh} (\gon_-/N)} (1+y^\alpha+\dots+y^{n\alpha})\quad .
\] 
The condition $0\in \mathcal C$ implies there exist non-negative 
integers $t_\alpha$, not simultaneously vanishing, with 
$0=\sum_{\alpha\in \weights_{\bar\goh} (\gon_-/N)} t_\alpha\alpha$. 
If $\mu=\sum_{\alpha\in \weights_{\bar\goh} (\gon_-/N)} m_\alpha \alpha $ 
for some integers $m_\alpha\geq 0$, the coefficient in front of $y^{\mu}$ in $A(n)$ 
is greater than or equal to $(n- \max_\alpha m_\alpha)/(\max_\alpha t_\alpha)$. 
%(the latter follows as for any $\gamma =\gamma+k \cdot 0 = \sum m_\alpha \alpha+ k \sum t_\alpha\alpha$). 
Therefore, as $n\to \infty$, the coefficient in front of $y^\mu$ in $A(n)$ 
tends to $\infty$. As the character $\CharS V_\lambda(\gol) $ is 
polynomial (in the variables $y^{\psi_j}$ for $\psi_j$ - the fundamental 
$\bar{\goh}$-weights of $\bar\gog$) with positive coefficients, the statement follows.
\end{proof}

For an ordered tuple $X\subset\bar\goh^*$, denote by $P_{X}:\bar\goh^*\to \mZ_{\geq 0}\cup \{\infty\}$ 
the vector (Kostant) partition function with respect to $X$. More precisely,
$P_X(\alpha)$ is defined to be the number of ways to write $\alpha$ as a non-negative 
integral combination of the elements of $X$. For overview of the theory of 
vector partition functions we direct the reader to \cite{BBCV} and 
the references therein.

We conclude this section by expressing the 
multiplicities $m(\mu, \lambda)$ via vector partition functions. 
Vector partition functions can be written in closed form as 
piecewise quasi-polynomials, for example using the Szenes-Vergne formula, 
\cite{Milev:PartialFractions}, and therefore so can the function $m(\mu, \lambda)$. 
Here, closed form formula means a formula which can
be evaluated with a constant number of arithmetic operations in 
the coordinates of $\mu$ and $\lambda$.
Although we do not use it in the remainder of the paper, 
we find that Theorem \ref{thTheBranchingRuleIsQP} is important on its own. 
As its proof  is constructive, it will allow the use of computers 
to compute the piecewise quasi-polynomial functions $m(\mu, \lambda)$. 
In addition, Theorem \ref{thTheBranchingRuleIsQP} gives an upper 
bound on the growth of $m(\mu, \lambda)$ as a function of $\mu$ and $\lambda$, 
as the example at the end of the section shows.

We give first a definition of piecewise quasi-polynomials. 
\begin{definition}
We say that a function $p:\mathbb C^n\to \mathbb Q$ is an elementary 
piecewise quasi-polynomial if there exist 
\begin{enumerate}
\item
A finite set of non-strict linear inequalities with rational 
coefficients cutting off a (not necessarily bounded) 
set $C\subset \mathbb R^n\subset \mathbb C^n$,
\item
A rational lattice $\Lambda$ (i.e., a $\mathbb Z$-span of a set of vectors with rational coordinates), 
\item
A rational polynomial $p_{C,\Lambda}(y_1,\dots, y_n)$, such that 
\[p(y_1,\dots, y_n)=\left\{\begin{array}{ll} p_{C, \Lambda}(y_1,\dots, y_n) & \mathrm{if~}(y_1,\dots, y_n)\in \Lambda \cap C\\0 &\mathrm{otherwise\quad .}\end{array}\right.
\] 
\end{enumerate}
We say that the degree of  $p_{C,\Lambda}$ is the degree 
of the elementary piecewise quasi-polynomial $p$.

We say that a function $p$ is piecewise quasi-polynomial if it can be 
written as a finite sum of elementary piecewise 
quasi-polynomials. We say that $p$ is of degree $d$ if $p$ can be 
written as a linear combination of elementary piecewise 
quasi-polynomials, each of degree less 
than or equal to $d$, and $d$ is the 
smallest integer with this property.
\end{definition}

\begin{theorem}
\label{thTheBranchingRuleIsQP}
Assume the notation of Theorem \ref{thBranchingMultsViaSymmetricTensors} 
and $0\notin \mathcal C$. Then the multiplicity function 
$m(\mu, \lambda)$ is piecewise quasi-polynomial (in the coordinates of $\mu$ and $\lambda$, 
i.e., a quasi-polynomial of 
$\dim \goh+\dim \bar\goh$ variables). The piecewise quasi-polynomial 
is of degree not exceeding the integer 
$\frac{1}{2}(\dim \gog- \dim\bar\gog-\dim \goh-\dim \bar\goh )$.
\end{theorem}
\begin{proof}
Let $\lambda\in \goh^*$ be integral with respect to $\gos=[\gol,\gol]$ and dominant with respect to 
$\gob\cap \gos$, where we recall that $\gol$ is the reductive Levi part of $\gop$.
The Weyl character formula asserts that 
\[
\CharL V_\lambda(\gol) = \left(\prod_{-\alpha\in \weights_{\goh}(\gol\cap \gob) } 
\frac{1}{1-x^{\alpha}}\right) \left(\sum_{w\in \mathbf{W}(\gol)} \sign(w)x^{w (\lambda+\rho_\gol)-\rho_\gol}\right),
\]
where $\rho_\gol:=\frac 1 2 \displaystyle\sum_{\beta\in \weights_{\goh}\gol\cap\gob}\beta$. 
By (\ref{eqGenVermaIsUnTimesV}) and the Weyl character formula we have that

\begin{eqnarray}\notag
\CharL M_{\lambda}(\gog, \gop) %& =& \left(\sum_{\alpha\in \Cone(\Delta(\gon_-))} Q(\alpha)  x^\alpha\right) \CharL V_\lambda \\\notag
&=&\left( \prod_{\alpha\in \weights_\goh(\gon_-)} \frac{1}{1-x^{\alpha}}\right)
\CharL V_\lambda(\gol) \\ \label{eqGenVermaChar}
&=&\left( \prod_{-\alpha\in \weights_{\goh}(\gob)}\frac{1}{1-x^{\alpha}}\right) 
\sum_{w\in \mathbf{W}(\gol)} \sign(w)x^{w (\lambda+\rho_\gol)-\rho_\gol}\quad .
\end{eqnarray}
Similarly, we get that 
\begin{eqnarray}\label{eqthBranchingMultsViaSymmetricTensors1}
\CharS M_{\mu}(\bar\gog, \bar\gop) 
&=&\left( \prod_{-\alpha\in \weights_{\bar{\goh}}(\bar\gob)} \frac{1}{1-y^{\alpha}}\right)  
\sum_{w\in \mathbf{W}(\bar \gol)} \sign(w) y^{w (\mu+\rho_{\bar\gol}) -\rho_{\bar\gol}} \quad ,
\end{eqnarray}
where $\rho_{\bar\gol}:=\frac 1 2\sum_{\alpha\in \Delta^+(\bar\gol)}\alpha $. 
Let $\gom_-$ and $\bar\gom_-$ be the nilradicals opposite to the nilradicals of $\gob$ and 
$\bar\gob$ and let $ M$ be the $\bar\goh$-module $M:=\gom_-/i(\bar \gom_-)$ under 
the $\ad\circ i$-action. We substitute \eqref{eqthBranchingMultsViaSymmetricTensors1} 
in (\ref{eqThBranchingViaSymmetricTensors}) from Theorem \ref{thBranchingMultsViaSymmetricTensors}, 
and cancel the common denominator $\prod_{\alpha\in \weights_{\bar \goh}\bar\gom_-} \frac{1 }{1-y^\alpha}$ 
from both sides to obtain
\begin{eqnarray}\label{eqBranchingSeries1}
&&\left( \prod_{\substack{ \alpha \in\weights_{\bar\goh} M }}\frac{1}{1-y^{\alpha}}\right) 
\sum_{w\in \mathbf{W}(\gol)}\notag \sign(w)y^{\pr(w(\lambda+\rho_\gol)-\rho_\gol)} \\&&= 
\sum_\mu m(\mu,\lambda) \sum_{w\in \mathbf{W}(\bar \gol)} 
\sign(w) y^{w (\mu+\rho_{\bar\gol})-\rho_{\bar\gol}}\quad ,
\end{eqnarray}
where $\mu$ runs over the weights integral with respect to $\bar{\gos}=[\bar{\gol},\bar{\gol}]$ 
and dominant with respect to $\bar\gob\cap \bar{\gos}$.

Set $X:=\weights_{\bar\goh} M$ (with multiplicities) 
and let $P_X$ be the vector partition function 
with respect to $X$. 

Fix a  weight $\mu$ dominant with respect to $\bar\gob\cap\bar\gos$. In the 
right hand side of \eqref{eqBranchingSeries1}, among all monomials of the 
form $m(\mu,\lambda) \sign(w) y^{w(\mu+\rho_{\bar{\gol}})-\rho_{\bar{\gol}}}$, the only monomial 
whose exponent is dominant with respect to $\bar \gob\cap\bar\gos$ is $m(\mu,\lambda)y^\mu$, obtained for $w=id$.
Therefore we can compare the coefficients in front of all $y^\gamma$
for which $\gamma$ is  $\bar\gob\cap\bar\gol$-dominant to get that 
\begin{eqnarray*}\label{eqBranchingMultGeneral}
m(\mu,\lambda)&=&\sum_{\substack{w\in \mathbf{W}(\gol) \\ 
\mu \mathrm{~is~}\bar{\gob}\cap\bar\gol-\mathrm{dominant}}}\sign(w)
P_X(\mu - \pr(w (\lambda+\rho_\gol)-\rho_\gol))\notag\\
&=&\sum_{\substack{w\in \mathbf{W}(\gol) \\ 
\mu \mathrm{~is~}\bar{\gob}\cap\bar\gol-\mathrm{dominant}}}\sign(w)
P_X(u_w( \mu, \lambda)+ \tau_w)\quad ,
\end{eqnarray*}
where we have set 
\begin{eqnarray*}
\tau_w:=-\pr( w(\rho_\gol) -\rho_\gol),\, \bar u_w:=-\pr\circ w,\, 
u_w(\mu,\lambda):=\mu+\bar u_w(\lambda)\quad .
\end{eqnarray*} 
By standard results on vector partition functions (see \cite{PartialFractions} 
and the references therein), $P_X$ is a piecewise quasi-polynomial  of degree less than or 
equal to the number of elements of $X$ minus the dimension of the 
ambient vector space, i.e.,
$\dim M - \dim \bar \goh =\dim \gom_-  - \dim \bar \gom_- -\dim \bar \goh=
\frac 1 2(\dim \gog-\dim \goh) - \frac 1 2(\dim \bar\gog-\dim\bar \goh)-\dim \bar \goh=
\frac{1}{2}\left(\dim \gog- \dim\bar\gog-\dim \goh-\dim \bar\goh\right)$.
\end{proof}

\textbf{Example.} For $\LieGtwo\hookrightarrow so(7)$, we have that 
$\frac{1}{2}(\dim \gog- \dim\bar\gog-\dim \goh-\dim \bar\goh)=1$ 
and the piecewise quasi-polynomial $m(\mu, \lambda)$ is linear. 
For $\mathbb C\oplus so(2n+1)\hookrightarrow so(2n+2)$, 
$so(2n)\hookrightarrow so(2n+1)$ and $gl(n)\hookrightarrow sl(n+1)$ 
we have that $\frac{1}{2}(\dim \gog- \dim\bar\gog-\dim \goh-\dim \bar\goh)=0$, 
and the piecewise quasi-polynomial  $m(\mu, \lambda)$  is bounded by a constant; 
in those cases, for finite dimensional representations, we know 
from the classical branching rules that the constant is 1 (``multiplicity free branching'').

%%%%%%%%%%%%%%%%%%%%%%%%%%%%%%%%%%%%%%%%%%%%%%%%%%%%%%%%%%%%%
\section{Constructing $\bar\gob$-singular vectors in $M_\lambda (\gog, \gop)$ }\label{secConstructingSingularVectors}

We recall that $V_{\lambda}(\gol)$ denotes the finite dimensional $\gop$-module inducing $M_\lambda (\gog, \gop)$.
By Weyl's semisimplicity theorem $V_{\lambda}(\gol)$ decomposes as a direct sum of irreducible 
$\bar \gol$-modules (we recall that $\bar \gol= i^{-1}(i(\bar{\gog})\cap \gol)$). 
By \eqref{eqSplitCharSVlambdagol}, the component $V_\mu(\bar\gol)$ of highest weight $\mu$ 
in the $\bar\gol$-decomposition of $V_\lambda(\gol)$ appears with  multiplicity $n(\mu,\lambda)$. 
By Theorem \ref{thBranchingMultsViaSymmetricTensors}, the multiplicity $m(\mu,\lambda)$ of 
$M_{\mu}(\bar\gog, \bar \gop)$ in $M_{\lambda}(\gog,\gop)$ is equal to the $\bar \gol$-multiplicity 
of $ V_\mu (\bar\gol)$ in $S^\star(\gon_-/N)\otimes V_{\lambda}(\gol)$. 
Therefore $m(\mu,\lambda)\geq n(\mu,\lambda)$.

With some assumptions, in the present Section we assign 
to each $\bar\gob\cap\bar\gol $-singular vector $v$ in $V_\lambda(\gol)$ 
a $\bar \gob$-singular vector $v'$ in $M_\lambda(\gog,\gop)$.
More precisely, assuming that $\lambda$ satisfies Condition B 
(Definition \ref{defSufficientlyGenericWeight} below), given a weight $\mu\in \bar\goh^*$
we show a method for constructing $n(\mu,\lambda)$ linearly independent $\bar{\gob}$-singular vectors 
in $M_\lambda (\gog, \gop)$. In the particular case that $\gop$ has a finite branching problem over $i(\bar{\gog})$
(Definition \ref{defFiniteBranchingProblem} below), it follows that $n(\mu,\lambda)=m(\mu,\lambda)$ and
the construction exhausts all $\bar\gob$-singular vectors of weight $\mu$ 
for which $m(\mu,\lambda)\neq 0$.

\begin{definition}\label{defFiniteBranchingProblem}
We say that the parabolic subalgebra $\gop\subset\gog$ has a finite branching problem over $i(\bar{\gog})$ if
for $\lambda\in \goh^*$ there are only finitely many $\mu\in \goh^*$ with  $m(\mu, \lambda)\neq 0$. 
\end{definition}
The above definition is equivalent to requesting that $M_{\lambda}(\gog,\gop)$ 
has finite Jordan-H\"older series as a $\bar\gog$-module. 
In view of Theorem \ref{thBranchingMultsViaSymmetricTensors}, 
it is also equivalent to requesting that $\dim \gon=\dim\bar\gon $. 
Theorem \ref{thBranchingMultsViaSymmetricTensors} furthermore implies that the definition 
does not depend on the choice of inducing representation. 
A non-trivial example of Definition \ref{defFiniteBranchingProblem} is given by 
$\LieAlgPair{\LieGtwo}{so(7)}$ for $\gop\simeq \gop_{(1,0,0)}$.

Our construction is carried out in two steps. First, we decompose the inducing finite 
dimensional representation $V_\lambda(\gol)$ over $\bar\gol$ by producing a 
spanning set for the $\bar\gob\cap\bar{\gol}$ singular vectors in 
$V_\lambda(\gol)$. Second, we project the obtained 
$\bar\gob\cap\bar{\gol}$-singular vectors to $\bar\gob$-singular, 
using an element  in the center of $U(\bar{\gog})$. 
%For our pilot example $\Li Under the additional condition eAlgPair{\LieGtwo}{so(7)}$, we do not use the entire center of $U(\gog)$  but just the $\bar\gog$- quadratic Casimir operator. 

We need the following definition.
\begin{definition}\label{defgobmaximalWeight}
Let $P\in \mathbb Q[[y^\gamma]]_{\gamma\in\goh^*}$. We say that 
$\mu\in \bar{\goh}^*$ is $\bar{\gob}$-maximal in $P$ if $y^\mu$ 
has non-zero coefficient in $P$ and we have that 
$y^{\mu+\alpha}$ has coefficient zero in $P$ for all positive roots $\alpha$ of $\bar{\gog}$.
\end{definition}
We describe a procedure for computing
the decomposition (\ref{eqSplitCharSVlambdagol}).
Consider the $\gop$-module $V_\lambda(\gol)$ inducing $M_\lambda(\gog,\gop)$. 
We compute $\CharL V_\lambda(\gol)$ using 
the Freudenthal formula with respect to 
$\gol$ (see e.g., \cite[\S 22.3, 22.4]{Humphreys}). Then we compute the projection 
$\Pr(\CharL V_\lambda(\gol)) = \CharS V_{\lambda}(\gol)$. 
Next we find a $\bar\gob$-maximal 
weight $\mu_1$ in $\Pr(\CharL V_\lambda(\gol))$. 
Then $V_{\mu_1} (\bar\gol)$ appears as a summand in the 
$\bar\gol$-decomposition of $V_\lambda(\gol)$. 
Next we compute $\CharS V_{\mu_1} (\bar\gol)$ using the 
Freudenthal formula with respect to $\bar \gol$ and 
subtract $\CharS V_{\mu_1}(\bar\gol) $ from 
$\CharS  V_\lambda(\gol)$. If we obtain zero, we are done, 
otherwise we find another $\bar\gob$-maximal 
weight $\mu_2$ in the remaining character 
$\CharS V_\lambda(\gol)-\CharS V_{\mu_1}(\bar{\gol})$ 
and subtract $\CharS V_{\mu_2}(\bar{\gol})$, and so on. 
On each step of the algorithm, the remaining character 
is a finite sum of monomials with non-negative integer coefficients. 
Furthermore, on each step the sum of the coefficients of all monomials 
decreases strictly, and therefore the algorithm 
terminates after a finite number of steps. 
By recording the intermediate summands $\CharS V_\mu (\bar \gol)$ appearing 
in the algorithm, we obtain the decomposition \eqref{eqSplitCharSVlambdagol}.

We are now in a position to compute the $\bar\gob\cap\bar\gol$-singular 
vectors of $V_\lambda(\gol)$
using linear algebra and the remarks on finite 
dimensional representations in 
Section \ref{secNotationPreliminaries}. Indeed, for each weight $\mu\in \bar\goh^*$ 
with $n(\mu,\lambda)>0$, we compute the intersection of 
the eigenspaces of the linear operators
given by the $\cdot$ action of $ i( \bar g_\alpha), i(\bar h_\alpha)-\mu(\bar h_\alpha) $ 
on  $V_\lambda(\gol)$,
where $\alpha$ runs over
the simple positive roots of $\bar \gol $.

%Before proceeding with the second step of the algorithm 
We discuss shortly 
the center of $U(\bar{\gog})$.
For an arbitrary semisimple Lie algebra $\bar \gog$, 
the center of $U(\bar\gog)$ is a polynomial algebra 
(of rank equal to $\dim\bar{\goh}$), as follows from 
the Harish-Chandra homomorphism Theorem 
(\cite[\S 23.3]{Humphreys}) and a Theorem of Chevalley 
(\cite[Chapter 3]{Humphreys:ReflectionGroups}). 
Let $\bar c_1,\dots, \bar c_k$ with $k=\dim \bar \goh$ 
be a basis of the center of $U(\bar\gog)$ consisting 
of homogeneous elements, where $\bar c_1$ is chosen to 
be the quadratic Casimir element. The elements 
$\bar c_j$ are known, a list of their degrees for each type can 
be found in \cite[page 260]{PopovVinberg:InvariantTheory}. 
A discussion of how to construct a basis 
$\bar c_1,\dots, \bar c_k$ can be found in \cite{BincerRiesselmann:CasimirG2} 
and the references therein.

%Let $\mu\in \bar{\goh^*}$, and let $E_\mu$ be the vector space spanned by the $\bar{\gob}$-singular vectors of weight $\mu$ in $M_\lambda(\gog, \gop)$. 
For an arbitrary $\mu\in\bar{\goh}^*$, the elements of the center of $U(\bar \gog)$ act by 
scalars on the Verma module $M_\mu(\bar{\gog}, \bar{\gob})$. 
Let the constant by which $\bar c_j$ acts on 
$M_\mu(\bar{\gog}, \bar{\gob})$ be $p_j(\mu)$, i.e.,
\[
p_{j}(\mu) v=\bar c_j\cdot v \quad\mathrm{for~all~}v\in M_\mu(\bar{\gog}, \bar{\gob})\quad .
\]
We say that two weights $\mu$ and $\nu$ in $\bar\goh^*$ are \emph{linked} 
if there exists an element $w$ of 
the Weyl group of $\bar{\gog} $ such that $w(\mu+\rho_{\bar\gog})-\rho_{\bar\gog}=\nu$, 
where $\rho_{\bar\gog}$ 
is the half-sum of the positive roots of $\bar\gog$. 
By the Harish-Chandra homomorphism Theorem, $\mu$ and $\nu$ 
are linked if and only if $p_j(\mu)= p_j( \nu)$ for $j=1,\dots, \dim \bar{\goh}$. 

Let $Q_\lambda$ denote the ``top level'' $\bar{\gog}$-submodule of 
$M_\lambda(\bar\gog,\bar{\gop})$ given by
\begin{equation}\label{eqTopLevel}
Q_\lambda:=i(U(\bar{\gog}))\otimes_{i(U(\bar\gop))} V_\lambda(\gol) \quad ,
\end{equation}
where $V_\lambda(\gol)$ is equipped with the $\bar{\gop}$ action induced from the embedding $i$. 
We note that $Q_\lambda$ lies in the BGG category $\mathcal O^{\bar\gop}$.

Before we proceed with the construction of $\bar\gob$-singular 
vectors in $Q_\lambda$, we note the following.
\begin{proposition}\label{propFiniteBranchingVermaEqualsTopPart}
Suppose the parabolic subalgebra $\gop\subset\gog$ has finite branching problem over $i(\bar{\gog})$. 
Let $\lambda\in \goh^*$ be dominant with respect to
$\gob\cap \gos$ and integral with respect to $\gos$. 
Then we have the $\bar\gog$-module isomorphism 
\[
Q_\lambda \simeq M_{\lambda}(\gog, \gop)\quad .
\]
\end{proposition}
\begin{proof} By Lemma \ref{leLargeNilradicalHasLsubmoduleIsoToSmallerNilradical}, 
the map $\pr$ maps $\weights_{\goh}\gon_-$ 
bijectively onto  $\weights_{\bar\goh}\bar\gon_-$. Let  $\beta_1,\dots, \beta_l$ 
be elements of the ordered tuple
$\weights_{\goh}\gon_-$, and let $\alpha_1,\dots, \alpha_l$ be the 
elements of $\weights_{\bar{\goh}}\bar\gon_-$, such that $\pr(\beta_i)=\alpha_i$. 
Let $v_1,\dots, v_k$ be a basis of $V_\lambda(\gol)$.
By \eqref{eqGenVermaIsUnTimesV} the set 
\[
A:= \{g_{\beta_1}^{n_1}\dots g_{\beta_l}^{n_l}\cdot v_s| n_i\in \mathbb Z_{\geq 0}, s=1,\dots, k\}
\]
is a vector space basis of $M_{\lambda}(\gog,\gop)$. 
Define a partial order $\succ $ 
among the monomials $g_{\beta_1}^{n_1}\dots g_{\beta_l}^{n_l}\cdot v_s$ by requesting 
that $g_{\beta_1}^{n_1}\dots g_{\beta_l}^{n_l}\cdot v_s\succ g_{\beta_1}^{m_1}\dots g_{\beta_l}^{m_l}\cdot v_t$
if and only if $\sum_{i=1}^l n_i >\sum_{i=1}^l m_i$.
The partial order does not depend on $s$ and $t$.
Consider the element 
\[
m(n_1,\dots, n_l, s):= i(\bar g_{\alpha_1}^{n_1}\dots \bar g_{\alpha_l}^{n_l})\cdot v_s\quad .
\]
Lemma 
\ref{leLargeNilradicalHasLsubmoduleIsoToSmallerNilradical} 
and the finiteness of the branching problem
imply the vector space isomorphism 
$i(\bar{\gon})\simeq (\gon+\gol)/\gol$. Therefore 
by the proof of Lemma \ref{leLargeNilradicalHasLsubmoduleIsoToSmallerNilradical}, for each index $j$ 
there exists a non-zero complex number $a_j$ and an element $u_j\in \gol$ such that 
$i(\bar g_{\alpha_j})\in a_j g_{\beta_j}+u_j $ and so
\begin{equation}\label{eqFiniteBranchingPBWbasis}
m(n_1,\dots, n_l, s)= (a_1 g_{\beta_1}+u_1)^{n_1}\dots (a_l g_{\beta_l}+u_l)^{n_l}\cdot v_s\quad .
\end{equation}
Rewrite $m(n_1,\dots, n_l, s)$ in the monomial basis $A$. As $\gon_-$ is 
an ideal in $\gol \oplus \gon_-$, \eqref{eqFiniteBranchingPBWbasis} 
and the PBW theorem imply that $m(n_1,\dots, n_l, s)$ has 
unique $\succ$-maximal monomial in the basis $A$, proportional to 
$g_{\beta_1}^{n_1}\dots g_{\beta_l}^{n_l}\cdot v_s$.
Therefore a straightforward filtration argument shows that, 
for fixed $P\in \mathbb Z_{\geq 0}$, 
the linear span of 
\[\{m(n_1,\dots, n_i,s)| \sum_{i=1}^l n_i\leq P, s=1,\dots, k\}
\]
equals the linear span of 
\[\{g_{\beta_1}^{n_1}\dots g_{\beta_l}^{n_l}\cdot v_s|\sum_{i=1}^l n_i\leq P, s=1,\dots, k\}
\] which proves the statement.
\end{proof}

%The following condition will be sufficient for $Q_\lambda$ to decompose as a direct sum of generalized Verma modules $M_{\mu}(\gog,\gop)$.
In order to associate $\gob\cap \gol$-singular vectors of $V_{\lambda}(\gol)$ to 
$\gob$-singular vectors in $M_\lambda(\gog,\gop)$, 
we need the following condition on $\lambda$.
\begin{definition}[Condition B]\label{defSufficientlyGenericWeight}
Let $\lambda\in \goh^*$ be integral with respect to  $\gos=[\gol,\gol]$ 
and dominant with respect to $\gob\cap \gos$. 
We say that $\lambda$ satisfies Condition B if for each pair of 
$\bar{\goh}$-weights $\mu \neq \nu$ for which 
%\begin{enumerate}
$n(\mu,\lambda)>0$ and  $n(\nu,\lambda)>0$,
%\item  $\nu-\mu$ is a non-negative non-zero linear combination of the positive simple roots of $\bar \gog$,
%\end{enumerate}
we have that $\nu$ and $\mu$ are not linked.

We say that $\lambda$ satisfies the strong Condition B if for all pairs 
$\mu \neq \nu$ with $n(\mu,\lambda)>0$ and  $n(\nu,\lambda)>0$, 
we have that $p_1(\mu)\neq p_1(\nu)$, where $p_1(\mu), p_1(\nu)$ 
are the constants given by the action of the quadratic Casimir element $\bar c_1$ of $\bar{\gog}$ on 
$M_\mu(\bar\gog, \bar{\gob})$, $M_\nu(\bar\gog, \bar{\gob})$.
\end{definition}

We note that, for the two extreme cases, the full parabolic 
subalgebra  $\bar\gop\simeq\bar\gog$, and 
$\bar \gop\simeq \bar \gob$, every possible weight $\lambda$ 
satisfies Condition B. Indeed, for the 
full parabolic $\bar{\gop}=\bar\gog$, all $\mu$ for which $n(\mu,\lambda)>0$ 
are dominant with respect 
to $\bar\gob$ and integral with respect to $\bar{\gog}$ 
and therefore cannot be pairwise 
linked. For the minimal parabolic $\bar\gop\simeq \bar\gob$, 
the inducing module is one-dimensional  and the condition 
is trivially satisfied. 

\begin{theorem}\label{thSufficientlyGenericWeightEnablesFDbranching}
Let $\lambda\in \goh^*$ be a integral with respect to $\gos=[\gol,\gol]$ 
and dominant with respect to $\gob\cap\gos$. 
Suppose $\lambda$ satisfies Condition B (Definition \ref{defSufficientlyGenericWeight}). 
Let $\mu\in \bar{\goh}^*$ such that $n(\mu,\lambda)> 0$.
Then there exists an element $\bar d$ in the center of $U(\bar \gog)$, 
such that $i(\bar d)\cdot ( 1\otimes_{U(\gop)}  v_\mu) $ is a non-zero $\bar\gob$-singular vector 
in $M_{\lambda}(\gog,\gop)$ for any $\bar\gob\cap\bar\gol$-singular vector $v_\mu$ 
in $V_\lambda(\gol)$ of $\bar \goh$-weight $\mu$.
\end{theorem}
\begin{proof}
Let $\bar h$ be a grading element that defines $\bar\gop$, i.e., 
an element such that $\alpha(\bar{\goh})>0$ for 
$\alpha\in \weights_{\bar{\goh}}\bar\gon$ and $\alpha(\bar{\goh})=0$ 
for $\alpha\in \weights_{\bar{\goh}}\bar\gol$. 
%Then $\bar h$ defines a partial order $<$ on $\bar \goh^*$ by declaring $ \nu_1> \nu_2$ whenever $\nu_1(h)>\nu_2(h)$. 
Let 
\[
a:=\max_{n(\nu,\lambda)>0} \nu(\bar h)\quad .
\]

For each $s\in \mathbb Z$, let $W_{s}$ be the sum of all $\bar{\goh}$-weight 
spaces of $V_\lambda(\gol)$ of weight $\nu$ for which $\nu(\bar h)\geq a-s $.
Then $W_s$ is a $\bar\gop$-module. Set 
\[
Q_s:=i(U(\bar\gog) )\otimes_{i(U(\bar{\gop}))} W_s\quad , \quad R_s:= Q_{s}/Q_{s-1}\quad ,
\]
where $Q_{-1}:=\{0\}$, where $Q_s$ are realized as $\bar{\gog}$-submodules of $M_{\lambda}(\gog,\gop)$.

Let 
\[
\mu(\bar{h})=a-m\quad .
\] 
We will prove by induction on $s$ that 
for each $s<m$ there exists an element $\bar d_s$ in the center 
of $U(\bar\gog)$ such that $i(\bar d_s) \cdot  Q_s =0$ and 
\begin{eqnarray}
i(\bar d_s)\cdot ( 1\otimes_{U(\bar{\gop})} v_\mu ) = x( 1\otimes_{U(\bar{\gop})} v_\mu )\mod Q_{m-1}
\end{eqnarray}
for some non-zero complex number $x$, where the latter computation is carried out in $R_{m}$. 

Our induction hypothesis holds trivially for the base case 
$s=-1$ with $d_{-1}:=1$. Suppose that there exists an element 
$\bar d_{s-1}$ that satisfies the induction hypothesis. 
Let $v_\nu\in V_\lambda(\gol)$ be a $\bar\gob\cap\bar\gol$-singular 
vector of $\bar\goh$-weight $\nu$ for which $\nu(\bar{h})=a-s$. 
As Condition B holds, there exists an element $\bar  d_\nu$ such 
that $\bar d_\nu\cdot M_{\nu}(\bar\gog, \bar\gob)=\{0\}$ and 
$\bar d_\nu \cdot M_{\mu}(\bar\gog, \bar{\gob})\neq \{0\}$.

Let $\bar\gom$ be the nilradical of $\bar\gob$ and $\bar{\gom}_-$ 
its opposite nilradical. 
By the PBW theorem we have the vector space isomorphism 
$U(\bar{\gog})\simeq U(\bar\gom_-)\otimes U(\bar\goh)\otimes U(\bar{\gom})$ 
and we can write 
\begin{equation}\label{eqmminushmDecompo}
\bar d_\nu=\sum_k \bar m^-_k \bar h_k \bar m^+_k
\end{equation}
for some monomials 
$\bar m^-_k\in U(\bar\gom_-), \bar h_k\in U(\bar\goh) ,\bar m^+_k\in U(\bar{\gom})$, 
where the summands are assumed linearly independent. 
Consider the action of a summand $\bar m^-_k \bar h_k \bar m^+_k$ 
in \eqref{eqmminushmDecompo} on $1\otimes_{i(U(\bar{\gop}))} v_{\nu}$. 
As $i(\bar\gom\cap \bar\gol)\cdot v_{\nu}=0 $, it follows that whenever 
the monomial $m^+_k$ is not constant it maps $1\otimes_{i(U(\bar{\gop}))} v_{\nu}$ to $Q_{s-1}$. 
If $m^-_k$ is a non-constant monomial, $m^+_k$ is a non-constant 
monomial too (as elements of the center of $U(\bar\gog)$ have 
zero weight and the summands in \eqref{eqmminushmDecompo} are linearly independent). 
Let $u$ be the sum of the summands in \eqref{eqmminushmDecompo} 
that lie in $U(\bar{\goh})$. As we have chosen $\bar d_\nu$ to 
annihilate $M_{\nu}(\bar\gog, \bar\gob)$, we have that $ u \cdot v_\nu=0 $. 
The preceding discussion now implies that 
$i(\bar d_\nu) \cdot( 1\otimes_{i(U(\bar{\gop}))} v_\nu)\in Q_{s-1}$ 
and therefore the product $\bar d_{s-1} \bar d_\nu$ annihilates 
$1\otimes_{i(U(\bar{\gop}))} v_\nu$. In addition, the above discussion shows 
\begin{eqnarray}
i(\bar d_\nu) \cdot (1\otimes_{i(U(\bar{\gop}))} v_\mu) 
= (1\otimes_{i(U(\bar{\gop}))} u\cdot v_\mu) \mod Q_{m-1}
\end{eqnarray} 
and $\left((1\otimes_{i(U(\bar{\gop}))} u\cdot v_\mu) \mod Q_{m-1} \right)$ 
is a non-zero multiple of $((1\otimes_{i(U(\bar{\gop}))} v_\mu)\mod Q_{m-1})$ 
as $\bar d_\nu \cdot M_{\mu}(\bar\gog, \bar{\gob})\neq \{0\}$. 
Set $\bar d_s$ to be the product  
\[
\bar d_s:=\bar d_{s-1}\prod_{ n(\nu, \lambda)>0, \nu(\bar{h})=a-s } \bar d_\nu\quad .
\] 
Then $\bar d_s$ is an element with the desired properties, 
which completes our induction argument.

We claim $\bar d:=\bar d_{m-1}$ satisfies the requirements of the Theorem. 
Indeed, $i(\bar d)\cdot (1\otimes_{U(\gop)} v_\mu)\neq 0$ by the 
construction of $\bar d$, and if $\bar m\in\bar\gom $, then 
$i(\bar m) \cdot(i(\bar d)\cdot (1\otimes_{U(\gop)} v_\mu)) =i(\bar d)\cdot (i(\bar m) \cdot(1\otimes_{U(\gop)} v_\mu) )=0$.
\end{proof}

The proof of Theorem \ref{thSufficientlyGenericWeightEnablesFDbranching} and 
Proposition \ref{propFiniteBranchingVermaEqualsTopPart} imply the following.
\begin{corollary}\label{corGenVermaDecompo}
Suppose the parabolic subalgebra $\gop\subset\gog$ has finite branching problem over $i(\bar{\gog})$. 
Let $\lambda\in \goh^*$ be dominant with respect to 
$\gob\cap \gos$ and integral with respect to $\gos$. 
Suppose $\lambda$ satisfies Condition B. 
Suppose in addition $n(\mu,\lambda)\leq 1$ for all $\mu\in \goh^*$.
Then we have the $\bar{\gog}$-module isomorphism
\[
M_\lambda(\gog,\gop)\simeq \bigoplus_{n(\mu,\lambda)=1} M_{\mu}(\bar\gog, \bar\gop)\quad .
\]
\end{corollary}
\begin{proof}
By Condition B and the fact that $n(\mu,\lambda)=1$, all summands on 
the right hand side have pairwise zero intersections. 
By Proposition \ref{propFiniteBranchingVermaEqualsTopPart} and 
Theorem \ref{thSufficientlyGenericWeightEnablesFDbranching} each of 
the modules on the right hand side is contained as a subset in the left 
hand side. By Proposition \ref{propFiniteBranchingVermaEqualsTopPart} 
the right hand side contains a vector space basis of $M_\lambda(\gog,\gop)$. 
This proves the statement.
\end{proof}

We conclude this section with a discussion on Condition B. 
Let $\mathbf x:=$ $(x_1, $ $\dots,$ $x_{\dim \goh})$ be indeterminate 
variables and let $\mathbf z:=(z_1,\dots, z_{\dim \goh})$, 
$z_i\in \mathbb Z_{\geq 0}$ be non-negative integers. 
Let $I$ be the set of simple roots of $\gog$. Let the fundamental weight that 
corresponds to the simple root $\eta_i$ be denoted by $\omega_i$. 

Let $I_\gol$ be the subset the elements of $I$ that are roots of 
$\gol$ and define
\begin{equation}\label{eqFDpartOfTheHighestWeight}
\begin{array}{rcl}
a(\mathbf x)&:=&\displaystyle\sum_{\alpha_i \in I\backslash I_\gol} x_{i} \omega_{i} \\
b(\mathbf z)&:=&\displaystyle\sum_{\alpha_i \in I_\gol} z_{i} \omega_{i}\\
\lambda(\mathbf x, \mathbf z)&:=& a(\mathbf x)+b(\mathbf z)
\quad .
\end{array}
\end{equation}

For a fixed value of $\mathbf z$, we claim that the set of 
the $\lambda$'s satisfying Condition B is a Zariski 
open set in the variables $x_1,\dots, x_{\dim \goh}$. Indeed, as a $\gos$-module, 
$V_{\lambda(\mathbf x, \mathbf y)} (\gol)$ is isomorphic to 
$V_{ b(\mathbf z)} (\gos)$, and therefore the number of values of $\alpha$ for which 
$n( \pr(\lambda(\mathbf x, \mathbf z) ) - \alpha,\lambda(\mathbf x, \mathbf z))\neq 0$ 
is a function independent of $x_1,\dots, x_{\dim \goh}$. 
For each $\mu\neq \nu\in \bar{\goh}^*$ with 
$n( \mu,\lambda(\mathbf x, \mathbf z))\neq 0$, $n( \nu,\lambda(\mathbf x, \mathbf z))\neq 0$, 
the condition that $\mu$ and $\nu$ are not linked is a finite (as $ \mathbf W(\bar{\gog})$ 
is finite) set  of linear $\neq$-inequalities in the coordinates of $\mu$ and $\nu$, which in 
turn are  linear functions in the $x_i$'s. 
Therefore for a fixed value of $\mathbf z$, Condition B is determined by a finite 
set of linear $\neq$-inequalities in the variables $x_1,\dots, x_{\dim \goh}$ 
and is therefore Zariski open. We note that it is possible the Zariski open set 
on the variables $x_1,\dots, x_{\dim \goh}$ is the empty set:
this could happen if one of the linear $\neq$-inequalities is of the 
form $a\neq 0$ for a non-zero constant $a$.

For a fixed value of $\mathbf z$, we claim that the set of $\lambda$ 
satisfying the strong Condition B is also a Zariski 
open set in the variables $x_1,\dots, x_{\dim \goh}$ (which may, again, be the empty set). 
Indeed, let $\alpha_i$ be the simple roots of $\bar{\gog}$, and define $q_i(\mathbf x, \mathbf z)$ via 
$\pr (\lambda(\mathbf x, \mathbf z))=:\sum q_i(\mathbf x, \mathbf z)\alpha_i$. 
The elements  $q_i(\mathbf x, \mathbf z)$ are linear 
functions of the $x_{i}$'s. By the preceding discussion, a weight 
$\mu\in \bar{\goh}^*$ with $n(\mu,\pr(\lambda(\mathbf x, \mathbf z)))>0$ is of 
the form $\pr(\lambda(\mathbf x, \mathbf z))-\alpha_i$ where $\alpha_i\in \bar{\goh}^*$ 
runs over some finite set of weights with coordinates that do not depend on the $x_i$. 
Therefore for $n(\mu,\pr(\lambda(\mathbf x, \mathbf z)))>0$, the function $p_1(\mu)$
is a quadratic polynomial in the variables  $x_1,\dots, x_{\dim \goh}$ 
with quadratic term depending only on $\lambda(\mathbf x,\mathbf z)$. 
Therefore the function $p_1(\mu)-p_1(\nu)$ is linear function in the 
variables $x_1,\dots, x_{\dim \goh}$, which proves our claim. 

%%%%%%%%%%%%%%%%%%%%%%%%%%%%%%%%%%%%%%%%%%%%%%%%%%%%%%%%%%%%%%%%%%%%%%%%%%%%%%%%%%%%%%%%%%%%%%%%%%%%%%%
\section{The pair $\LieAlgPair{\LieGtwo}{so(7)}$ }\label{secG2inB3}
In this section we study the branching problem for the pair 
$\LieAlgPair{\LieGtwo}{so(7)}$. 

Our original motivation for studying the example of
$\LieAlgPair{\LieGtwo}{so(7)}$ comes from a natural conformal
geometry problem in dimension $5$ (notice that $so(7)$ is the complexification of the
conformal Lie algebra in dimension $5$). There is a remarkable
connection between the geometry of generic $2$-plane fields ${\fam2 D}$ on
a manifold $M$ of dimension $5$ and pseudo-Riemannian metrics of signature $(3,4)$
in dimension $7$ whose holonomy is the split real form of ${G_2}$. 
The identification proceeds in two stages. The first
step, due to Cartan, is the equivalence of such fields
${\fam2 D}$ to a conformal class of metrics of
signature $(2,3)$ on $M$. The second step
canonically associates to a conformal structure of signature $(2,3)$ on $M$ a
Ricci-flat metric $\tilde{g}$ of signature $(3,4)$ on an open subset of
$(\mR_+\times M\times\mR)$ containing $(\mR_+\times M\times \{0\})$. 
Furthermore, the metrics constructed in this way satisfy 
$Hol({\tilde{g}})\subset\LieGtwo$, where $Hol({\tilde{g}})$ is 
the holonomy of the metric ${\tilde g}$. 
We note that metrics with holonomy reduced 
to ${\LieGtwo}$ in dimension $7$ are not easy to construct.

On the other hand, a geometrical characterization of the reduction of
the structure group $SO(7)$ down to ${G_2}$ for a given inducing representation
$\mV_\lambda$ is given by invariant differential operators acting on sections of the
associated vector bundles, intertwined by actions of $so(7)$ and ${\LieGtwo}$.

Finally, by \eqref{eqInvriantDiffOpActing}, we can transform the problem 
of finding differential invariants for
$(so(7),\LieGtwo),\mV_\lambda$ into an algebraic question
on homomorphisms between generalized Verma modules, which correspond
to solutions of the branching problem. For the homomorphisms constructed in the 
present Section, we need to further construct lifts of homomorphisms to the category of
semiholonomic generalized Verma modules. 
We hope to address this geometrically interesting topic in a future work.

The example $\LieAlgPair{\LieGtwo}{so(7)}$ is of special
algebraic interest as well. 
In the branching problem examples in 
which the subalgebra is of
rank 1 (such as, for example, the principal $sl(2)$-subalgebra of $sl(3)$)
the quasipolynomials that govern the branching via Theorem
\ref{thTheBranchingRuleIsQP}
are in one variable. In our point of view these examples  are not 
sufficiently ``generic'' - here, we note that there are fundamental differences in the
vector partition function theory in one and two or more variables. These differences
stem from the fact that polynomials in one variable
are a principal ideal domain, which in turn implies ``non-generic'' properties 
such as uniqueness of the partial fraction decompositions 
used to compute the vector partition functions. 

The pairs of type $B_2\simeq C_2\hookrightarrow A_3\simeq D_3$,
$A_3\hookrightarrow B_3$, $A_3\hookrightarrow C_3$ are,
by Theorem \ref{thTheBranchingRuleIsQP},
of quasipolynomial degree $0=\frac{1}{2}(\dim \gog- \dim\bar\gog-\dim
\goh-\dim \bar\goh )$,
and we do not consider them sufficiently ``generic''.
Of the remaining low-dimensional examples
there is only one in which the subalgebra is simple,
namely that of $\LieAlgPair{\LieGtwo}{so(7)}$.

We note that the finite-dimensional branching laws for the 
pair $(so(7),\LieGtwo)$ were studied in \cite{McGovernG2InB3Branching}. However, 
\cite{McGovernG2InB3Branching} does not address the task of finding 
explicit formulas for the $\gob$-singular vectors realizing the branching, which is ultimately needed 
for the geometric applications described in the preceding paragraph. 
It is important to note that formulas for the finite dimensional 
branching can be extracted using the infinite-dimensional branching problem.
We make use of this idea in Corollary 
\ref{corFormulasEigenVectorsRemainValid} which states that 
the $\bar\gob$-singular vector formulas derived  in an infinite-dimensional setting in 
Theorem \ref{leBranchingExplicit(x_1,0,0),(1,a,b)}
remain valid for finite dimensional representations as well.

\subsection{Structure of the embedding $\LieAlgPair{\LieGtwo}{so(7)}$ and the corresponding parabolic subalgebras}\label{secso7G2LieAlgstructure}\label{secG2inB3def}
We describe first the structure of $so(7)$ as a $\LieGtwo$-module 
and the parabolic subalgebras $\gop$ of $so(7)$ relative to the parabolic 
subalgebras $i(\bar\gop)$ of $i(\LieGtwo)$.

\offlineVersionOnly{
The structure constants of $so(7) $ and $\LieGtwo$ relative to a 
Chevalley-Weyl basis are classically known, and available through a number 
of computer algebra systems, including \cite{Milev:vpf}. We use abbreviations  
$g_{\pm j}$, $j\in \{\pm 1, \dots, \pm 9\}$, $h_{1}, h_2, h_3$ for the
the Chevalley-Weyl generators of $so(7)$ and 
$g_{\pm j}$, $j\in \{\pm 1, \dots, \pm 6\}$, $\bar h_1, \bar h_2$ for the 
Chevalley-Weyl generators of $\LieGtwo$. In our labeling
$g_{\pm 1}, g_{\pm 2}, g_{\pm 3}$ and $\bar g_{\pm 1}, \bar g_{\pm 2}$ are the simple generators of 
$so(7)$ and $\LieGtwo$. We denote by $\alpha_1$ the short positive simple root of $\LieGtwo$ and by 
$\alpha_2$ the long one. We denote by $\eta_1=\varepsilon_1-\varepsilon_2, \eta_2:=\varepsilon_2-\varepsilon_3, \eta_3:=\varepsilon_3 $ the simple positive roots of $so(7)$. We denote by $\omega_1:=\varepsilon_1$, $\omega_2:=\varepsilon_1+\varepsilon_2$ and  $\omega_3:=\frac12(\varepsilon_1+\varepsilon_2+\varepsilon_3)$ 
the fundamental weights of  $so(7)$ and by $\psi_1:=2\alpha_1+\alpha_2$, $\psi_2:=3\alpha_1+2\alpha_2$ 
the fundamental weights of $\LieGtwo$.

For a more detailed explanation of our notation 
we direct the reader to the extended version of this 
text available online, \cite{MilevSomberg:BranchingExtended}.

} %offlineVersionOnly

\onlineVersionOnly{
We start by fixing a Chevalley-Weyl basis of the Lie algebra $so(2n+1)$. 
Let the defining vector space $V$ of $so(2n+1)$ have a basis 
$e_{1},\dots e_{n}, e_0, e_{-1}, \dots e_{-n}$. Let the defining symmetric bilinear 
form $B$ of $so(2n+1)$ be given by the matrix 
$\mathbf B:=\left(\begin{array}{c|c|c}\mathbf 0_{n\times n} & \mathbf 0_{n\times 1} 
& \mathbf I_{n\times n} \\ \hline \mathbf 0_{1\times n} & 1 & \mathbf  0_{1\times n}\\ 
\hline \mathbf I_{n\times n}& \mathbf{0}_{n\times 1} & \mathbf 0_{n\times n} \end{array} 
\right)$, alternatively defined by $ B(e_i, e_j):=0, i\neq -j$, 
$ B(e_i,e_{-i})=1$, $B(e_i,e_0):=0$, $B(e_0,e_0):=1$, or alternatively defined as an 
element of $S^2(V^*)$,
\begin{equation}\label{eqDefiningSymmetricBilinearFormTypeB}
B:=\sum_{i=-n}^{n} e_i^*\otimes e_{-i}^*=(e_0^*)^2+ 2\sum_{i=1}^{n} e_i^* e_{-i}^*\quad ,
\end{equation}
under the identification $v^*w^*:=\frac {1}{2!} \left(v^*\otimes w^*+w^*\otimes v^*\right)$.

In the basis $e_{1},\dots e_{n}, e_0, e_{-1}, \dots e_{-n}$, the 
matrices of the elements of $so(2n+1)$ are of the form
\[
\left(\begin{array}{c|c|c}A&\begin{array}{c}v_1\\ \vdots \\ v_n\end{array} &C=-C^T 
\\\hline \begin{array}{ccc}w_1 &\dots&  w_n\end{array} &0& \begin{array}{ccc}-v_1 
&\dots&  -v_n\end{array} \\\hline D=-D^T&\begin{array}{c}-w_1\\ \vdots \\ -w_n\end{array} 
& -A^T\end{array}\right)\quad , \]
i.e., all matrices $\mathbf C$ such that $\mathbf C^t\mathbf B+\mathbf B\mathbf C=0$. 
We fix $e_{1}^*,\dots e_{n}^*, e_0^*, e_{-1}^*, \dots e_{-n}^*$ to be basis of $V^*$ 
dual to $e_{1},\dots e_{n}, e_0, e_{-1}, \dots e_{-n}$. We identify elements of $End(V)$ 
with elements of $V\otimes V^*$.
In turn, we identify elements of $End(V)$ with their matrices in the basis 
$e_{1},\dots, e_{n}$, $e_0$, $e_{-1}, \dots, e_{-n}$.

Fix the Cartan subalgebra $\goh$ of $so(2n+1)$ to be the subalgebra of
diagonal matrices, i.e., the subalgebra spanned by the vectors 
$e_i\otimes e_i^*-e_{-i}\otimes e_{-i}^* $. 
Then the basis vectors $e_{1},\dots e_{n}, e_0, e_{-1}, \dots e_{-n}$ are a basis for 
the $\goh$-weight vector decomposition of $V$. Let the $\goh$-weight of $e_i, i>0$, be 
$\varepsilon_i$. Then the $\goh$-weight of $e_{-i}, i>0$ is $-\varepsilon_i$, and an 
$\goh$-weight decomposition of $so(2n+1)$ is given by the elements 
$g_{\varepsilon_i-\varepsilon_j}:= e_i\otimes e_j^*-e_{-j}\otimes e_{-i}^*$, 
$g_{\pm(\varepsilon_i+\varepsilon_j)}:=e_{\pm i}\otimes e_{\mp j}^*-e_{\pm j}\otimes e_{\mp i}^*$ 
and $g_{\pm \varepsilon_i}:=\sqrt{2}\left(e_{\pm i}\otimes e_0^*-e_0\otimes e_{\mp i}^*\right)$, 
where $i,j>0$.

Define the symmetric bilinear form $\scalarLA{\bullet}{\bullet}$ on $\goh^*$ by 
$\scalarLA{\varepsilon_i}{\varepsilon_j}=1$ if $i=j$ and zero otherwise. 
%This form is induced by the defining symmetric bilinear form $B$ (\ref{eqDefiningSymmetricBilinearFormTypeB}).
%and is proportional to the one induced by the Killing form. 
%For an element $\alpha\in \goh^*$, let $h_{\alpha}\in \goh$ denote the element defined by 
%$[h_{\alpha}, g_\gamma]=\scalarLA{ \alpha}{\gamma} g_\gamma $ for all $\gamma\in \Delta(\gog)$. 
%The elements $\frac{2}{ \scalarLA{\eta_i}{\eta_i}}h_{\eta_i}$ (for $1\leq i\leq n$) together with 
%the elements $g_{\alpha}, \alpha\in \Delta(\gog)$ give a Chevalley-Weyl basis with respect to $\goh$. 

The root system of $so(2n+1)$ with respect to $\goh$ is given by $\Delta(\gog):=\Delta^+(\gog)\cup\Delta^-(\gog)$, 
where we define 
\begin{equation}\label{eqPosRootSystemB3}
\Delta^+(\gog):=\{\varepsilon_i\pm\varepsilon_j|1\leq i< j\leq n\}\cup\{\varepsilon_i|1\leq i\leq n\}
\end{equation}
and $\Delta^-(\gog):=-\Delta^{+}(\gog)$. We fix the Borel subalgebra $\gob$ of $so(2n+1)$ to be the 
subalgebra spanned by $\goh$ and the elements $g_{\alpha}, \alpha\in \Delta^+(\gog)$. The simple positive 
roots corresponding to $\gob$ are then given by 
\[
\eta_1:=\varepsilon_1-\varepsilon_2,\dots,  \eta_{n-1}:=\varepsilon_{n-1}-\varepsilon_{n},\eta_{n}:=\varepsilon_{n}\quad .
\]
We recall that the $i^{th}$ fundamental weight $\omega_i\in \goh^*$ is defined by $\scalarLA{\omega_i}{\eta_j}:=\delta_{ij}\frac{\scalarLA{\eta_j}{\eta_j}}2$.
The fundamental  weights of $so(2n+1)$ are then given by 
\begin{eqnarray*}
\omega_1&:=&\varepsilon_1 \\
\omega_2&:=&\varepsilon_1+\varepsilon_2 \\ 
\vdots \\
\omega_{n-1}&:=&\varepsilon_1+\dots+\varepsilon_{n-1} \\
\omega_n&:=&1/2(\varepsilon_1+\dots+\varepsilon_n)\quad.
\end{eqnarray*}

For the remainder of this Section we fix the odd orthogonal Lie algebra to be $so(7)$. 

We abbreviate the  Chevalley-Weyl generator $g_{\alpha}\in so(7)$ as $g_i$ 
according to the Table below. 
The rows are sorted by the graded lexicographic order on the 
simple coordinates of the roots. We furthermore set 
$h_1:=[g_1,g_{-1}]$, $h_2:=[g_2,g_{-2}]$, $h_3:=1/2[g_{3}, g_{-3}]$.

\noindent \begin{longtable}{ccc} \caption{Generators of so(7)}\\generator & root simple coord. & root $\varepsilon$-notation \\\hline \endhead
$g_{-9}$&$(-1, -2, -2)$&$-\varepsilon_{1}-\varepsilon_{2}$\\ $g_{-8}$&$(-1, -1, -2)$&$-\varepsilon_{1}-\varepsilon_{3}$\\ $g_{-7}$&$(0, -1, -2)$&$-\varepsilon_{2}-\varepsilon_{3}$\\ $g_{-6}$&$(-1, -1, -1)$&$-\varepsilon_{1}$\\ $g_{-5}$&$(0, -1, -1)$&$-\varepsilon_{2}$\\ $g_{-4}$&$(-1, -1, 0)$&$-\varepsilon_{1}+\varepsilon_{3}$\\ $g_{-3}$&$(0, 0, -1)$&$-\varepsilon_{3}$\\ $g_{-2}$&$(0, -1, 0)$&$-\varepsilon_{2}+\varepsilon_{3}$\\ $g_{-1}$&$(-1, 0, 0)$&$-\varepsilon_{1}+\varepsilon_{2}$\\ $h_{1}$&$(0, 0, 0)$&$0$\\ $h_{2}$&$(0, 0, 0)$&$0$\\ $h_{3}$&$(0, 0, 0)$&$0$\\ $g_{1}$&$(1, 0, 0)$&$\varepsilon_{1}-\varepsilon_{2}$\\ $g_{2}$&$(0, 1, 0)$&$\varepsilon_{2}-\varepsilon_{3}$\\ $g_{3}$&$(0, 0, 1)$&$\varepsilon_{3}$\\ $g_{4}$&$(1, 1, 0)$&$\varepsilon_{1}-\varepsilon_{3}$\\ $g_{5}$&$(0, 1, 1)$&$\varepsilon_{2}$\\ $g_{6}$&$(1, 1, 1)$&$\varepsilon_{1}$\\ $g_{7}$&$(0, 1, 2)$&$\varepsilon_{2}+\varepsilon_{3}$\\ $g_{8}$&$(1, 1, 2)$&$\varepsilon_{1}+\varepsilon_{3}$\\ $g_{9}$&$(1, 2, 2)$&$\varepsilon_{1}+\varepsilon_{2}$\\ \end{longtable}

Let now $\bar\gog= \LieGtwo$. One way of defining the positive 
root system of $\LieGtwo$ is by setting it to be the set of vectors
\begin{equation}\label{eqPosRootSystemG2}
\Delta(\bar \gog):=\{\pm(1, 0), \pm(0, 1), \pm(1, 1), \pm(2, 1), \pm(3, 1), \pm(3, 2)\}\quad .
\end{equation}
We set $\alpha_1:=(1,0)$ and $\alpha_2:=(0,1)$. 
We fix a bilinear form $\scalarSA{\bullet}{\bullet}$ on $\bar\goh$, proportional
\footnote{For the bilinear form induced by the Killing form of $\LieGtwo$, 
the long root has squared length $36$, so the coefficient of proportionality is $6$} 
to the one induced by Killing form by setting 
$\left(\begin{array}{cc}\scalarSA{\alpha_1}{\alpha_1} & \scalarSA{\alpha_1}{\alpha_2} \\ 
\scalarSA{\alpha_2}{\alpha_1} & \scalarSA{\alpha_2}{\alpha_2}\\\end{array}\right):= 
\left(\begin{array}{cc}2 & -3\\ -3 & 6\\\end{array}\right)$. 
In an  $\scalarSA{\bullet}{\bullet}$-orthogonal basis the root system of $\LieGtwo$ 
is often drawn as \begin{tabular}{c}
\begin{pspicture}(-1,1)(1,-1)
\psline(0,0)(0.5,0.87)
\psline(0,0)(-0.5,0.87)
\psline(0,0)(0.5,-0.87)
\psline(0,0)(-0.5,-0.87)
\psline(0,0)(1,0)
\psline(0,0)(-1,0)
\psline(0,0)(0.57,0.32) %0.58, 0.33
\psline(0,0)(-0.57,0.32)
\psline(0,0)(0.57,-0.32)
\psline(0,0)(-0.57,-0.32)
\psline(0,0)(0,0.67)
\psline(0,0)(0,-0.67)
\put(1,0.1){\tiny{$\alpha_2$}}
\put(-0.62,0.38){\tiny{$\alpha_1$}}
\end{pspicture}\end{tabular}. 
The fundamental weights of $\LieGtwo$ relative to the 
simple basis $\{\alpha_1, \alpha_2\}$ are then given by 
$\psi_1:=2\alpha_1+\alpha_2$, $\psi_2:=3\alpha_1+2\alpha_2$.
We fix a basis for the Lie algebra $\LieGtwo$ by giving a 
set of Chevalley-Weyl generators  $\bar g_{i}$, $i\in \{\pm 1, \dots \pm 6\}$, 
and by setting $ \bar h_{1}:=[\bar g_1, \bar g_{-1}]$, $ \bar h_{2}:=3[\bar g_2, \bar g_{-2}]$ 
according to the following table.

\begin{longtable}{ccc} \caption{Generators of $\LieGtwo$\label{tableG2ChevalleyGeneratorsNames} }\\
generator & root simple coord.\\\hline\endhead 
$\bar{g}_{-6}$&$(-3, -2)$\\ $\bar{g}_{-5}$&$(-3, -1)$\\ $\bar{g}_{-4}$& $(-2, -1)$\\ $\bar{g}_{-3}$&$(-1, -1)$\\ $\bar{g}_{-2}$&$(0, -1)$\\ $\bar{g}_{-1}$&$(-1, 0)$\\ $\bar{h}_{1}$&$(0, 0)$\\ $\bar{h}_{2}$&$(0, 0)$\\ $\bar{g}_{1}$&$(1, 0)$\\ $\bar{g}_{2}$&$(0, 1)$\\ $\bar{g}_{3}$&$(1, 1)$\\ $\bar{g}_{4}$&$(2, 1)$\\ $\bar{g}_{5}$&$(3, 1)$\\ $\bar{g}_{6}$&$(3, 2)$\\ \end{longtable}
} %onlineVersionOnly

All embeddings $\LieAlgPair{\LieGtwo}{so(7)}$ are conjugate over $\mathbb C$. 
Following  \cite{McGovernG2InB3Branching}, one such embedding is given via 
\begin{equation}\label{eqEmbedG2B3accordingToMcGovern}
i(\bar{g}_{\pm 2}):= g_{\pm 2}, \quad i(\bar g_{\pm 1}):= g_{\pm 1}+g_{\pm 3}\quad.
\end{equation}
\onlineVersionOnly{
As $\bar g_{\pm 1}, \bar g_{\pm 2}$ generate the Lie algebra $\LieGtwo$, 
the preceding data determines the map $i$. The structure constants of $\LieGtwo$ 
are readily available from a number of computer algebra systems, including ours, 
and one can directly check that the map $i$ is a Lie algebra homomorphism. 
Alternatively, we can use $i(\bar{g}_{\pm 1}), i(\bar{g}_{\pm 2})$ to generate 
a Lie subalgebra of $so(7)$, verify that this subalgebra is indeed 
14-dimensional  and simple, and finally use this 14-dimensional image 
to compute the structure constants of $\LieGtwo$.

Let $\pr: \goh^*\to \bar\goh^*$ be the map naturally induced by $i$. Then
\begin{equation}\label{eqProjectionCartanB_3toCartanG_2}
\pr(\underbrace{\varepsilon_1-\varepsilon_2}_{\eta_1})=\pr(\underbrace{\varepsilon_3}_{\eta_3})=\alpha_1 \quad \pr(\underbrace{\varepsilon_2-\varepsilon_3}_{\eta_2})=\alpha_2\quad ,
\end{equation}
or equivalently
\begin{equation*}
\pr(\omega_1)=\pr(\omega_3)=\psi_1 \quad \pr(\omega_2)=\psi_2\quad.
\end{equation*}
Conversely, $\iota: \bar\goh^* \to \goh^*$ (see Section \ref{secNotationPreliminaries}) is the map 
\begin{equation*}\label{eqG2rootSystemEmbeddingInB3}
\iota(\alpha_2)=3\eta_2= 3\varepsilon_2-3\varepsilon_3 \quad \iota(\alpha_1)=\eta_1+2\eta_3= \varepsilon_1-\varepsilon_2+2\varepsilon_3\quad .
\end{equation*}
The following Lemma gives the pairwise inclusions between the 
parabolic subalgebras of $so(7)$ and the embeddings of the parabolic 
subalgebras of $\LieGtwo$.
} %onlineVersionOnly

%\offlineVersionOnly{The proof of the following theorem is a straightforward computation and we omit it.}
\begin{lemma}\label{leParabolicsG2inParabolicsB3}
For the pair $G_2\stackrel{i}{\hookrightarrow} so(7)$, 
let $\goh, \gob,\gop,\bar \goh, \bar{\gob}, \bar{\gop}$ denote respectively Cartan, 
Borel and parabolic subalgebras with the assumptions that 
$i(\bar\goh)\subset\goh\subset \gob$, $i(\bar{\gob})\subset \gob\subset\gop$, 
$\bar{\gob}\subset\bar{\gop}$.
Then we have the following inclusion diagram for all possible values of $\gop, \bar\gop$.
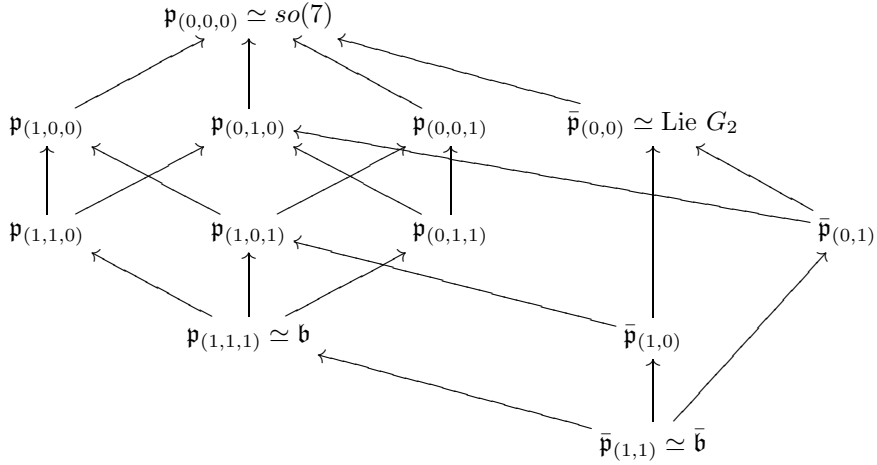
\[
\xymatrix{
&{\BThreeDiagram{0}{0}{0}}\simeq so(7) & &  \\
{\BThreeDiagram{1}{0}{0}} \ar[ur] &{\BThreeDiagram{0}{1}{0}}  \ar[u] &{\BThreeDiagram{0}{0}{1}} \ar[ul] &
{\GTwoDiagram{0}{0}} \simeq\LieGtwo \ar[ull] \\
{\BThreeDiagram{1}{1}{0}} \ar[u]  \ar[ur] &{\BThreeDiagram{1}{0}{1}}  \ar[ul] \ar[ur]  &{\BThreeDiagram{0}{1}{1}} \ar[ul]  \ar[u] &
& {\GTwoDiagram{0}{1}}  \ar[ul] \ar[ulll]  \\
&{\BThreeDiagram{1}{1}{1}}\simeq \gob \ar[u]  \ar[ur]  \ar[ul]  &  &
{\GTwoDiagram{1}{0}} \ar[ull] \ar[uu]  \\
&&& {\GTwoDiagram{1}{1}}\simeq \bar\gob \ar[u] \ar[ull]  \ar[uur]\\
}
\]
%In the diagram we label each root in the Dynkin diagram of $so(7)$ (respectively $\LieGtwo$) by 0 if the 
%corresponding $\goh$-(respectively, $\bar\goh$-) root space is a subset of the Levi part $\gol$ of $\gop$ (respectively 
%$\bar \gol$ of $\bar \gop$) and label by 1 otherwise. 
If a path of arrows exists from one node of the diagram 
to the other, then the corresponding parabolic subalgebras lie inside one another. 
If in the diagram a direct arrow exists from a parabolic subalgebra $\bar \gop$ of 
$\LieGtwo$ to a parabolic subalgebra $\gop$ of $so(7)$, 
then $\bar \gop= i^{-1}(i(\bar{\gog})\cap\gop)$.
\end{lemma}
\begin{proof}
Since a parabolic subalgebra of a reductive Lie algebra is a direct sum of root spaces, 
\eqref{eqEmbedG2B3accordingToMcGovern} implies that $i(\bar g_{\pm 1})$ belongs to $\gop$ 
if and only if both $ g_{\pm 1}$ and $ g_{\pm 3}$  
belong to $\gop$. Similarly we have that $i(\bar g_{\pm 2})$ belongs to $\gop$ if and only 
if $g_{\pm 2}$ belongs to $\gop$. This proves the Lemma.
\end{proof}

The following Lemma describes the structure of $so(7)$ as a 
module over the Levi part of 
a parabolic subalgebra of $\LieGtwo$. The proof of the Lemma is a straightforward computation.
Recall from Section \ref{secNotationPreliminaries} that $\goh$ stands for 
Cartan subalgebra, $\gop$ stands for parabolic subalgebra, $\gob$ stands for 
Borel subalgebra, 
$\gop$ stands for parabolic subalgebra, $\gol $ stands for the 
reductive Levi part of $\gop$, $\gos:=[\gol,\gol]$
and $\gon_-$ stands for the nilradical opposite to the nilradical of $\gop$. 

\begin{lemma}\label{leStructureB3overG2}
Let $\bar \goh, \bar \gob, \bar \gos, \bar{\gol}, \bar{\gop}, \bar \gon_-\subset \LieGtwo$ 
depend on the subalgebras $\goh, \gon_-, \gob, \gos, \gol,  \gop\subset so(7)$ as in Section \ref{secNotationPreliminaries}. 
We recall the key requirements that $\bar \gop= i^{-1}(i(\bar{\gog})\cap\gop)$ and $i(\bar{\gob})\subset\gob$.
\begin{enumerate}
\item
Suppose $\bar \gop\simeq\bar{\gop}_{(0,0)}\simeq \LieGtwo$. Then $so(7)$ decomposes over $\LieGtwo$ as the direct sum of $\LieGtwo$-modules
$V_{\psi_2}(\LieGtwo)\oplus V_{\psi_1}(\LieGtwo)$, where $V_{\psi_1}(\LieGtwo)$ is the 7-dimensional simple $\LieGtwo$-module and $V_{\psi_2}(\LieGtwo)\simeq \LieGtwo$ is the adjoint 14-dimensional  $\LieGtwo$-module. 
A weight vector basis of $V_{\psi_1}(\LieGtwo)$ is given by the following table.
\begin{longtable}{rp{1.8cm}l}Monomial form & embedding in so(7) &$\bar{\goh}$-weight\\
$\ad i(\bar g_{-1}) \ad i(\bar g_{-2})\ad i(\bar g_{-1})^{2}\ad i(\bar g_{-2})\ad i(\bar g_{-1}) \cdot v$ &$2g_{-6}-4g_{-7} $ &$-\psi_1$\\
$\ad i(\bar g_{-2})\ad i(\bar g_{-1})^{2}\ad i(\bar g_{-2})\ad i(\bar g_{-1})\cdot v$& $ 4g_{-4}+2g_{-5}$  & $\psi_1-\psi_2 $\\
$\ad i(\bar g_{-1})^{2}\ad i(\bar g_{-2})\ad i(\bar g_{-1}) \cdot v$& $4g_{-1}-2g_{-3} $& $-2\psi_1+\psi_2$ \\
$\ad i(\bar g_{-1})\ad i(\bar g_{-2})\ad i(\bar g_{-1})\cdot v$&$ -2h_{3}+2h_{1}$ &$0$\\
$\ad i(\bar g_{-2})\ad i(\bar g_{-1})\cdot v$&$g_{3}-2g_{1} $ & $2\psi_1-\psi_2$ \\
$\ad i(\bar g_{-1} ) \cdot v$& $g_{5}+2g_{4} $&$-\psi_1+\psi_2$\\
$v$ &  $2g_{7}-g_{6}$&$\psi_1$\\
\end{longtable}

\item
Suppose $\bar\gop\simeq \bar\gop_{(0,1)}$. Then $\bar \gol\simeq sl(2)+\bar\goh$ 
and $\bar\gos=[\bar\gol,\bar\gol]\simeq sl(2)$. The $\bar\gob\cap \bar\gos$-positive 
root of $\bar\gos$ can be identified with $\alpha_1$ and $so(7)$ 
decomposes over $\bar \gos$ as 
$2V_{\frac 3 2\alpha_1}(\bar\gos)\oplus 2V_{\alpha_1}(\bar\gos)\oplus 
2V_{\frac{\alpha_1}{2}}(\bar\gos)\oplus 3V_{0}(\bar\gos)$ 
(the summands are of respective dimensions $2\times 4+2\times 3+2\times 2+3\times 1$). 

By Lemma \ref{leParabolicsG2inParabolicsB3} $\gop\simeq \gop_{(0,1,0)}$. 
Therefore $\dim\gon_-=7$, $i(\bar\gon_-)\subset\gon_-$, and $\gon_-$ equals the 
maximal $\bar \gol$-stable subspace of the nilradical of the opposite 
Borel subalgebra of $so(7)$. In addition, $\gon_-$ is isomorphic as an 
$\bar \gos$-module to 
$\underbrace{V_{\frac 3 2 \alpha_1}(\bar\gos)}_{\subset i(\bar \gon)}\oplus 
V_{\frac{\alpha_1}{2}}(\bar\gos)\oplus \underbrace{V_0(\bar\gos)}_{\subset i(\bar \gon)}$ 
and is spanned respectively by the $\bar\goh$-weight vectors 
$\underbrace{\{\underbrace{6g_{-8}}_{\mathrm{lowest}},-2g_{-6}-2g_{-7}, -g_{-4}+g_{-5},  g_{-2}\}}_{\subset i(\bar\gon_-)}$,  $\{\underbrace{g_{-6}-2g_{-7}}_{\mathrm{lowest}}, 2g_{-4}+g_{-5}\}$ 
and $\underbrace{\{g_{-9} \}}_{\subset i(\bar\gon_-)} $. 
The coefficients of the preceding elements are chosen 
so that $\ad (i(\bar {g}_{-1}))$  sends the highest weight vectors to the lower ones.

The set $\weights_{\bar\goh} (\gon_-/N)$ equals $\{ -\psi_1, \psi_1-\psi_2\}$.
\item
Suppose $\bar\gop\simeq \bar{\gop}_{(1,0)}$. Then $\bar \gol\simeq sl(2)+\bar\goh$ and 
$\bar\gos=[\bar \gol,\bar \gol]\simeq sl(2)$. 
The $\bar\gob\cap\bar\gos$-positive root of $\bar \gos$ can 
be identified with $\alpha_2$ and $so(7)$ decomposes over $\bar\gos$ as 
$V_{\frac 3 2\alpha_2}(\bar\gos) \oplus 
6V_{\frac{\alpha_2}{2}}(\bar \gos)\oplus 6V_{0}(\bar \gos)$ 
(the summands are of respective dimensions $1\times 3+6\times 2+6\times 1$). 
The nilradical of the opposite Borel subalgebra of $so(7)$ has a maximal 
$\bar \gol$-stable subspace isomorphic as a $\bar \gos$-module 
to $ 3V_{\frac{\alpha_2} {2}}(\bar\gos)\oplus 2V_0(\bar\gos)$, 
spanned respectively by the $\bar\goh$-weight vectors 
$\underbrace{\{ \underbrace{-g_{-9}}_{\mathrm{lowest}}, g_{-8}  \}}_{\subset i(\bar\gon_-)}$, $\underbrace{\{\underbrace{g_{-4}-g_{-5}}_{\mathrm{lowest}}, g_{-1}+g_{-3}\}}_{\subset i(\bar\gon_-)}$, $\{\underbrace{2g_{-4}+g_{-5}}_{\mathrm{lowest}},2g_{-1}-g_{-3}\} $, 
$\underbrace{\{g_{-6}+g_{-7}\}}_{\subset i(\bar\gon_-)}$, $\{g_{-6}+2g_{-7}\}$. 
The coefficients of the preceding 
elements are chosen so that $\ad (i(\bar {g}_{-2}))$  sends the 
highest weight vectors to the lower ones.

By Lemma \ref{leParabolicsG2inParabolicsB3} $\gop\simeq$ 
$\gop_{(1,0,1)}$, $\gop_{(0,0,1)}$ or $\gop_{(1,0,0)}$.
\begin{enumerate}
\item
Let $\gop\simeq \gop_{(1,0,1)}$. Then $\dim\gon_-=8 $ and $i(\bar\gon_-)\subset \gon_-$. 
The set $\weights_{\goh} (\gon_-/N)$ equals $\{ -\psi_1, \psi_1-\psi_2, -2\psi_1+\psi_2\}$.
\item
Let $\gop\simeq \gop_{(0,0,1)}$. Then $\dim\gon_-=6 $ and $\dim (i(\bar\gon_-)\cap\gon_-)=3$. 
In addition, $\gon_-$ is isomorphic as a $\bar\gos=[\bar\gol,\bar{\gol}]$-module 
to $ 2V_{\frac{\alpha_2 }2}(\bar\gos)\oplus 2V_0(\bar\gos)$ and is 
spanned by the $\bar\goh$-weight vectors $\underbrace{\{-g_{-9}, g_{-8}  \}}_{\subset i(\bar\gon_-)}$, 
$\{\underbrace{-g_{-5}}_{\mathrm{lowest}}, g_{-3}\}$,  
$\underbrace{\{g_{-6}+g_{-7}\}}_{\subset i(\bar\gon_-)}$,  
$\{ g_{-6}-2g_{-7}\}$. The coefficients of the preceding elements 
are chosen so that $\ad (i(\bar {g}_{-2}))$  sends the highest weight vectors to the lower ones.

The set $\weights_{\goh} (\gon_-/N)$ equals $\{ -\psi_1\}$.
\item
Let $\gop\simeq \gop_{(1,0,0)}$. Then $\dim\gon_-=5 $ and $\dim (i(\bar\gon_-)\cap\gon_-)=2$. 
In addition, $\gon_-$ is isomorphic as a $\bar\gos$-module to 
$ 2V_{\frac{\alpha_2} 2}(\bar \gos) \oplus V_0 (\bar \gos)$, 
spanned by the $\bar\goh$-weight vectors 
$\underbrace{\{-g_{-9}, g_{-8}  \}}_{\subset i(\bar\gon_-)}$, 
$\{\underbrace{g_{-4}}_{\mathrm{lowest}}, g_{-1}\}$, $\{g_{-6}\}$. 
The coefficients of the preceding elements are chosen so that 
$\ad (i(\bar {g}_{-2}))$  sends the highest weight vectors to the lower ones.
\end{enumerate}
\end{enumerate}
\end{lemma}

\begin{corollary}\label{corParabolicsB3weaklyCompatibleWithG2}
All 8 parabolic subalgebras $\gop\supset\gob\supset\goh$ of $so(7)$ are 
weakly compatible with $i(\LieGtwo)$. Furthermore, 
$\gop_{(1,0,1)}, \gop_{(0,1,0)}, \gop_{(1,1,1)}$ and 
$\gop_{(0,0,0)}$ are compatible with $i(\LieGtwo)$ in the sense of 
\cite[Section 3]{Kobayashi:branching} and the remaining 4 parabolic subalgebras of $so(7)$ are not.
\end{corollary}

%%%%%%%%%%%%%%%%%%%%%%%%%%%%%%%%%%%%%%%%%%%%%%%%%%%%%%%%%%%%%%%%%
%%%%%%%%%%%%%%%%%%%%%%%%%%%%%%%%%%%%%%%%%%%%%%%%%%%%%%%%%%%%%%%%%

\subsection{Constructing $\bar\gob$-singular vectors in $M_\lambda(so(7),\gop)$ for $\LieAlgPair{\LieGtwo}{so(7)}$ }\label{secConstructingSingularG2inB3}
The remarks after Definition \ref{defFiniteBranchingProblem} and 
Lemma \ref{leStructureB3overG2} imply the following Corollary.
\begin{corollary}
The parabolic subalgebra $\gop\subset so(7)$ has a 
finite branching problem over $i(\LieGtwo)$ if and only if
$\gop\simeq \gop_{(1,0,0)}$ or $\gop\simeq  so(7)=\gop_{(0,0,0)}$.
\end{corollary}
In the present Section we apply Section \ref{secConstructingSingularVectors} 
to the pair $\LieAlgPair{\LieGtwo}{so(7)}$. Let $\bar{c}_1$ be the quadratic Casimir element of 
$U(\LieGtwo)$.  A formula for a degree 6 homogeneous element, 
linearly (and algebraically) independent on $c_1^3$ can be found in \cite{BincerRiesselmann:CasimirG2}.

In order to simplify computations, we only use 
the quadratic Casimir operator and the strong Condition B (Definition \ref{defSufficientlyGenericWeight}). 
As it turns out in all of our examples, for a fixed value of $\mathbf z$, 
the set of weights $\lambda(\mathbf x, \mathbf z)$ satisfying the strong Condition B is 
non-empty and Zariski open, and therefore includes almost all values of $x_1,x_2,x_3$.

\subsection{$\gop_{(1,0,0)}\subset so(7)$} %- parabolic subalgebra}

We devote special attention to the finite branching problem, as here,
by Proposition \ref{propFiniteBranchingVermaEqualsTopPart}
the ``top level'' 
$Q_\lambda$ equals $M_\lambda(so(7),\gop_{(1,0,0)})$. Furthermore, whenever 
the requirements of Corollary \ref{corGenVermaDecompo} are satisfied, the branching 
problem of $M_\lambda(so(7),\gop_{(1,0,0)})$ 
is reduced to problems in single central character blocks of Category $\mathcal O^{\bar{\gop}_{(1,0)}}$ 
(Corollary \ref{corGenVermaDecomposB3G2} below).
In Theorem \ref{leBranchingExplicit(x_1,0,0),(1,a,b)}, for $\gop\simeq\gop_{(1,0,0)}$  and highest weight of the form 
$x_1\omega_1$, $x_1\omega_1+\omega_2$,  $x_1\omega_1+\omega_3$,  $x_1\omega_1+2\omega_2$,  
$x_1\omega_1+\omega_1+\omega_2$,  $x_1\omega_1+2\omega_3$, we produce a 
spanning set of the $\bar \gob$-singular vectors with $m(\mu,\lambda)\neq 0$, 
except for finitely many values of the variable $x_1$. 
Among those 5 families of weights, except for $\lambda=x_1\omega_1+\omega_1+\omega_2$ and the 
finitely many values of $x_1$ excluded by Corollary \ref{corGenVermaDecomposB3G2}, 
we prove that the generalized Verma $so(7) $-module $M_\lambda(so(7),\gop_{(1,0,0)})$ is 
a direct sum of generalized Verma $\LieGtwo$-modules $M_\mu(\LieGtwo, \gop_{(1,0)})$.

%In the following theorem we construct the ``top level'' 
%$\bar{\gob}$-singular vectors of $M_\lambda(so(7), \gop_{(1,0,0)})$ with $\lambda$ of the form 
%$x_1\omega_1+\omega_2$, $x_1\omega_1+\omega_3$, $x_1\omega_1+2\omega_2$, 
%$x_1\omega_1+\omega_2+\omega_3$, $x_1\omega_1+\omega_2+\omega_3$. 
%For  $\lambda\neq x_1\omega_1+\omega_2+\omega_3$ and the excluded values of $x_1$ 
%below, Corollary \ref{corGenVermaDecomposB3G2} below shows 
%that the singular vectors given in Theorem \ref{leBranchingExplicit(x_1,0,0),(1,a,b)} 
%give a splitting of $M_\lambda(so(7),\gop_{(1,0,0)})$ into a direct sum of 
%generalized Verma $\LieGtwo$-modules.
\begin{theorem}\label{leBranchingExplicit(x_1,0,0),(1,a,b)}
Let $v_\lambda$ be the highest weight vector of the  generalized Verma module 
$M_{\lambda}(so(7),\gop_{(1,0,0)})$.
We recall $\omega_1= \varepsilon_1$, $\omega_2=\varepsilon_1+\varepsilon_2$, 
$\omega_3=1/2(\varepsilon_1+\varepsilon_2+\varepsilon_3)$ 
are the fundamental weights of $so(7)$. Recall $\bar{\gob}$ is the Borel subalgebra of $\LieGtwo$.
\begin{enumerate}
\item Suppose $\lambda=x_1\omega_1+\omega_2$. Then
\begin{eqnarray}\notag
v_{\lambda, 0}&:=&v_\lambda \\\notag
v_{\lambda, 1}&:=&\left((-x_1-2)g_{-3}g_{-2}-4g_{-4}+2g_{-2}g_{-1}\right) \cdot v_\lambda \\\notag
v_{\lambda, 2}&:=&\left((x_1^{2}+2x_1)g_{-3}^{2}g_{-2}+(x_1-1)g_{-6}\right.\\\notag&&
\left.+(-2x_1-1)g_{-1}g_{-3}g_{-2}-2g_{-4}g_{-1}+2g_{-1}^{2}g_{-2}\right)\cdot v_{\lambda} \\\notag
\end{eqnarray}
are  linearly independent $\bar\gob$-singular vectors. Moreover, for $x_1\notin\{-1, -7/2, -6\} $, 
the $\LieGtwo$-submodules generated by $v_{\lambda, i}$ have pairwise empty intersections.
%of respective $\bar\goh$- weights $2\alpha_1+3\alpha_2+x_1(\alpha_1+2\alpha_2)$, $\alpha_1+2\alpha_2+x_1(\alpha_1+2\alpha_2)$, $\alpha_1+\alpha_2+x_1(\alpha_1+2\alpha_2)$
\item Suppose $\lambda=x_1\omega_1+\omega_3$. Then
\begin{eqnarray}\notag
v_{\lambda, 0}&:=&v_\lambda \\\notag
v_{\lambda, 1}&:=&\left(-x_1g_{-3}+g_{-1} \right) \cdot v_\lambda \\\notag
v_{\lambda, 2}&:=&\left((2x_1+5)g_{-3}g_{-2}g_{-3}-g_{-6} \right.\\\notag&&\left.+2g_{-4}g_{-3}-2g_{-1}g_{-2}g_{-3}\right)\cdot v_{\lambda}
\end{eqnarray}
are linearly independent $\bar\gob$-singular vectors. Moreover, for $x_1\notin\{ -5, -3, -1\} $, 
the $\LieGtwo$-submodules generated by $v_{\lambda, i}$ 
have pairwise empty intersections. 
\offlineVersionOnly{
\item Suppose $\lambda=x_1\omega_1+2\omega_2$, $\lambda=x_1\omega_1+\omega_2+\omega_3$ 
or $\lambda=x_1\omega_1+2\omega_3$. Then there exist formulas similar 
to the preceding ones that give linearly independent $\bar\gob$-singular vectors 
in $M_{\lambda}(so(7),\gop_{(1,0,0)})$.
Moreover, if  $\lambda=x_1\omega_1+2\omega_3$ or  $\lambda=x_1\omega_1+2\omega_2$, 
for all but finitely many values of $x_1$, the $\bar\gob$-singulars have pairwise empty intersections.

The formulas, as well as the critical values for $x_1$ can be found in the extended version of this 
text available online, \cite{MilevSomberg:BranchingExtended}.
}
\onlineVersionOnly{
\item Suppose $\lambda=x_1\omega_1+2\omega_2$.  Then
\begin{eqnarray}\notag
v_{\lambda,0}&:=&v_\lambda \\\notag
v_{\lambda,1}&:=&\left(-4 g_{-4}+(-x_{1}-3) g_{-3} g_{-2}+2 g_{-1} g_{-2}\right)\cdot v_\lambda\\\notag
v_{\lambda,2}&:=&
\left((4x_{1}-4) g_{-6}-8 g_{-4} g_{-1}+(2x_{1}^{2}+6x_{1}) g_{-3}^{2} g_{-2}\right. \\\notag&&
\left.+(-4x_{1}-4) g_{-1} g_{-3} g_{-2}+4 g_{-1}^{2} g_{-2}\right)\cdot v_\lambda\\\notag
v_{\lambda,3}&:=&\left( (-10x_{1}-20) g_{-4} g_{-3} g_{-2}+20 g_{-4} g_{-1} g_{-2}-20 g_{-4}^{2}\right.\\\notag&&
+(-x_{1}^{2}-5x_{1}-6) g_{-3}^{2} g_{-2}^{2}+(x_{1}^{2}+5x_{1}+6) g_{-2} g_{-3}^{2} g_{-2}\\\notag&&
\left. +(5x_{1}+10) g_{-1} g_{-3} g_{-2}^{2}-10 g_{-1}^{2} g_{-2}^{2}\right)\cdot v_\lambda\\\notag
v_{\lambda,4}&:=&
\left(
(12x_{1}^{2}+48x_{1}+48) g_{-9}+(6x_{1}^{2}+24x_{1}+24) g_{-8} g_{-2}+(12x_{1}^{2}\right.\\\notag&&
+60x_{1}-168) g_{-6} g_{-4}+(3x_{1}^{3}+24x_{1}^{2}-72) g_{-6} g_{-3} g_{-2}\\\notag&& 
+(-6x_{1}^{2}-30x_{1}+84) g_{-6} g_{-1} g_{-2}+(6x_{1}^{3}+51x_{1}^{2}+72x_{1}-12) g_{-4} g_{-3}^{2} g_{-2}\\\notag&&
+(-18x_{1}^{2}-138x_{1}-84) g_{-4} g_{-1} g_{-3} g_{-2}+(24x_{1}+168) g_{-4} g_{-1}^{2} g_{-2}\\\notag&&
+(-24x_{1}-168) g_{-4}^{2} g_{-1}+(1/2x_{1}^{4}+6x_{1}^{3}+41/2x_{1}^{2}+21x_{1}) g_{-3}^{3} g_{-2}^{2}\\\notag&&
+(-3x_{1}^{3}-57/2x_{1}^{2}-54x_{1}-18) g_{-1} g_{-3}^{2} g_{-2}^{2}\\\notag&&
+(3x_{1}^{2}+18x_{1}+24) g_{-1} g_{-2} g_{-3}^{2} g_{-2}+(-12x_{1}-84) g_{-1}^{3} g_{-2}^{2}\\\notag&&
\left.+(9x_{1}^{2}+69x_{1}+42) g_{-1}^{2} g_{-3} g_{-2}^{2}\right)\cdot v_\lambda
\\\notag
v_{\lambda,5}&:=&
\left(
(8x_{1}^{2}+24x_{1}+16) g_{-9} g_{-1}+(-24x_{1}^{2}-72x_{1}+112) g_{-8} g_{-4}\right.\\\notag&&
+ (-8x_{1}^{3}-36x_{1}^{2}-12x_{1}+56) g_{-8} g_{-3} g_{-2}+(16x_{1}^{2}+48x_{1}-48) g_{-8} g_{-1} g_{-2}\\\notag&& 
+(16x_{1}^{2}+8x_{1}-168) g_{-6} g_{-4} g_{-1}\\\notag&&
+(-4x_{1}^{4}-14x_{1}^{3}+12x_{1}^{2}+34x_{1}-28) g_{-6} g_{-3}^{2} g_{-2}\\\notag&& 
+(8x_{1}^{3}+12x_{1}^{2}-60x_{1}-24) g_{-6} g_{-1} g_{-3} g_{-2}+(-8x_{1}^{2}-4x_{1}+84) g_{-6} g_{-1}^{2} g_{-2}\\\notag&& 
+(-4x_{1}^{3}+28x_{1}-56) g_{-6}^{2}+(8x_{1}^{3}+28x_{1}^{2}-12x_{1}-32) g_{-4} g_{-1} g_{-3}^{2} g_{-2}\\\notag&& 
+(16x_{1}+56) g_{-4} g_{-1}^{3} g_{-2}+(-16x_{1}^{2}-48x_{1}+28) g_{-4} g_{-1}^{2} g_{-3} g_{-2}\\\notag&& 
+(-16x_{1}-56) g_{-4}^{2} g_{-1}^{2}+(-1/3x_{1}^{5}-5/2x_{1}^{4}-5x_{1}^{3}+5/6x_{1}^{2}+7x_{1}) g_{-3}^{4} g_{-2}^{2}\\\notag&& 
+(4/3x_{1}^{4}+20/3x_{1}^{3}+5x_{1}^{2}-25/3x_{1}-14/3) g_{-1} g_{-3}^{3} g_{-2}^{2}+(-8x_{1}-28) g_{-1}^{4} g_{-2}^{2}\\\notag&& 
+(8x_{1}^{2}+24x_{1}-14) g_{-1}^{3} g_{-3} g_{-2}^{2}+(-4x_{1}^{3}-14x_{1}^{2}+2x_{1}+8) g_{-1}^{2} g_{-3}^{2} g_{-2}^{2}\\\notag&& 
\left.+(4x_{1}+8) g_{-1}^{2} g_{-2} g_{-3}^{2} g_{-2} \right)\cdot v_\lambda\\\notag
\end{eqnarray}
are linearly independent $\bar\gob$-singular vectors.

Moreover, for $x_1\notin\{ 0, -1, -4$, $-3$, $-9/2, -2, -8$, $-6, -7, -5, -9\}$, 
the $\LieGtwo$-submodules generated by $v_{\lambda, i}$ have pairwise empty intersections. 

\item Suppose $\lambda=x_1\omega_1+\omega_2+\omega_3$. Then
\begin{eqnarray}\notag
v_{\lambda, 0}&:=&v_\lambda \\\notag
v_{\lambda, 1}&:=&\left( -x_1g_{-3}+g_{-1}\right) \cdot v_\lambda \\\notag
v_{\lambda, 2}&:=&\left( 5g_{-1}g_{-2}+(-2x_1-4)g_{-3}g_{-2}+(x_1+2)g_{-2}g_{-3}-5g_{-4}  \right) \cdot v_\lambda \\\notag
v_{\lambda, 3}&:=&(  (1/6x_1^{3}+5/3x_1^{2}+8/3x_1)g_{-3}^{2}g_{-2}\\\notag&&
+(-1/2x_1^{2}-x_1)g_{-3}g_{-2}g_{-3}\\\notag&&
+(1/6x_1^{2}+4/3x_1-3)g_{-6}+(2/3x_1^{2}+10/3x_1-1)g_{-4}g_{-3}\\\notag&&
+(-2/3x_1^{2}-16/3x_1-3)g_{-1}g_{-3}g_{-2}+(-x_1-7)g_{-4}g_{-1}\\\notag&&
+(x_1+7)g_{-1}^{2}g_{-2}+(x_1+2)g_{-1}g_{-2}g_{-3}) \cdot v_\lambda \\\notag
v_{\lambda, 4}&:=&((1/6x_1^{3}+5/3x_1^{2}+17/3x_1+6)g_{-3}g_{-2}g_{-3}\\\notag&&
+(-1/8x_1^{2}-5/4x_1-2)g_{-3}^{2}g_{-2}\\\notag&&
+(-1/12x_1^{2}-2/3x_1-9/4)g_{-6}+(1/2x_1^{2}+35/12x_1+31/12)g_{-4}g_{-3}\\\notag&&
+(-1/6x_1^{2}-11/12x_1-7/6)g_{-1}g_{-2}g_{-3}\\\notag&&
+(-1/6x_1^{2}-13/12x_1-1/4)g_{-1}g_{-3}g_{-2}\\\notag&&
+(-1/3x_1-23/12)g_{-4}g_{-1}+(1/3x_1+23/12)g_{-1}^{2}g_{-2}) \cdot v_\lambda \\\notag
v_{\lambda,5}&:=& ((-x_1^{4}-9/2x_1^{3}-3/2x_1^{2}+7x_1)g_{-3}^{3}g_{-2}+(-3x_1^{3}-3x_1^{2} \\\notag&& 
+18x_1-12)g_{-6}g_{-3}+(3x_1^{3}+21/2x_1^{2}-9/2x_1-9)g_{-1}g_{-3}^{2}g_{-2} \\\notag
&& +(-3x_1^{2}-15x_1+18)g_{-8}+(3x_1^{2}+3x_1-30)g_{-6}g_{-1}\\\notag&&
+(6x_1^{2}+12x_1-24)g_{-4}g_{-1}g_{-3} \\\notag&& 
+(-6x_1^{2}-18x_1+12)g_{-1}^{2}g_{-3}g_{-2}+(-3x_1^{2}-3x_1+6)g_{-1}g_{-3}g_{-2}g_{-3} \\\notag&& 
+(-6x_1-21)g_{-4}g_{-1}^{2}+(6x_1+21)g_{-1}^{3}g_{-2}+(3x_1+6)g_{-1}^{2}g_{-2}g_{-3})\cdot v_\lambda \\\notag
v_{\lambda,6}&:=& ((1/2x_1^{4}+29/4x_1^{3}+153/4x_1^{2}+173/2x_1+70)g_{-3}^{2}g_{-2}^{2}g_{-3}\\\notag&& 
+(1/2x_1^{3}+11/2x_1^{2}+19x_1+20)g_{-6}g_{-2}g_{-3} \\\notag&&
+(-x_1^{3}-11x_1^{2}-38x_1-40)g_{-6}g_{-3}g_{-2}\\\notag&& 
+(-1/6x_1^{4}-29/12x_1^{3}-51/4x_1^{2}-173/6x_1-70/3)g_{-2}g_{-3}^{3}g_{-2}\\\notag&&  
+(4x_1^{3}+93/2x_1^{2}+349/2x_1+210)g_{-4}g_{-3}g_{-2}g_{-3}\\\notag&&
+(-1/2x_1^{3}-27/4x_1^{2}-121/4x_1-45)g_{-4}g_{-3}^{2}g_{-2}\\\notag&& 
+(-2x_1^{3}-93/4x_1^{2}-349/4x_1-105)g_{-1}g_{-3}g_{-2}^{2}g_{-3}\\\notag&&
+(-5/2x_1^{2}-45/2x_1-50)g_{-9}\\\notag&& 
+(-5/2x_1^{2}-45/2x_1-50)g_{-8}g_{-2} +(5x_1^{2}+45x_1+100)g_{-4}^{2}g_{-3}\\\notag&&
+(-5/2x_1^{2}-45/2x_1-50)g_{-6}g_{-4}\\\notag&& 
+(5/2x_1^{2}+45/2x_1+50)g_{-6}g_{-1}g_{-2}\\\notag&&
 +(-5x_1^{2}-45x_1-100)g_{-4}g_{-1}g_{-2}g_{-3}\\\notag&& 
+(1/2x_1^{3}+27/4x_1^{2}+121/4x_1+45)g_{-1}g_{-2}g_{-3}^{2}g_{-2}\\\notag&&
+(5/2x_1^{2}+45/2x_1+50)g_{-1}^{2}g_{-2}^{2}g_{-3}) \cdot v_\lambda \\\notag
\end{eqnarray}
\begin{eqnarray}\notag
v_{\lambda,7}&:=&((-1/24x_1^{6}-37/48x_1^{5}-265/48x_1^{4}-115/6x_1^{3}\\\notag&&
-129/4x_1^{2}-21x_1)g_{-3}^{3}g_{-2}^{2}g_{-3}\\\notag&& 
+(-1/4x_1^{5}-15/4x_1^{4}-77/4x_1^{3}-147/4x_1^{2}-6x_1+36)g_{-6}g_{-3}g_{-2}g_{-3}\\\notag&& 
+(1/8x_1^{5}+15/8x_1^{4}+10x_1^{3}+87/4x_1^{2}+51/4x_1-9)g_{-6}g_{-3}^{2}g_{-2}\\\notag&&  
+(-1/12x_1^{5}-7/6x_1^{4}-139/24x_1^{3}-287/24x_1^{2}-17/2x_1)g_{-4}g_{-3}^{3}g_{-2}\\\notag&&  
+(1/8x_1^{5}+2x_1^{4}+187/16x_1^{3}+487/16x_1^{2}+33x_1+9)g_{-1}g_{-3}^{2}g_{-2}^{2}g_{-3}\\\notag&&  
+(-1/4x_1^{4}-9/4x_1^{3}-8x_1^{2}-33/2x_1-18)g_{-9}g_{-3}\\\notag&&
+(-1/4x_1^{4}-3x_1^{3}-53/4x_1^{2}-51/2x_1-18)g_{-8}g_{-2}g_{-3}\\\notag&&  
+(1/4x_1^{4}+15/4x_1^{3}+37/2x_1^{2}+69/2x_1+18)g_{-8}g_{-3}g_{-2}\\\notag&&  
+(1/8x_1^{4}+3/2x_1^{3}+41/8x_1^{2}+9/4x_1-9)g_{-6}^{2}\\\notag&&
+(-1/4x_1^{4}-11/4x_1^{3}-7x_1^{2}+9x_1+36)g_{-6}g_{-4}g_{-3}\\\notag&&  
+(1/4x_1^{4}+3x_1^{3}+21/2x_1^{2}+25/4x_1-15)g_{-6}g_{-1}g_{-2}g_{-3}\\\notag&&  
+(-1/4x_1^{4}-13/4x_1^{3}-14x_1^{2}-43/2x_1-6)g_{-6}g_{-1}g_{-3}g_{-2}\\\notag&&  
+(-1/4x_1^{4}-29/8x_1^{3}-145/8x_1^{2}-73/2x_1-24)g_{-1}^{2}g_{-3}g_{-2}^{2}g_{-3}\\\notag&&  
+(1/2x_1^{4}+29/4x_1^{3}+145/4x_1^{2}+73x_1+48)g_{-4}g_{-1}g_{-3}g_{-2}g_{-3}\\\notag&&  
+(-1/12x_1^{4}-x_1^{3}-25/6x_1^{2}-27/4x_1-3)g_{-1}g_{-2}g_{-3}^{3}g_{-2}\\\notag&& 
+(-1/2x_1^{3}-31/4x_1^{2}-143/4x_1-51)g_{-8}g_{-1}g_{-2}\\\notag&&  
+(1/4x_1^{3}+23/8x_1^{2}+87/8x_1+27/2)g_{-1}^{2}g_{-2}g_{-3}^{2}g_{-2}\\\notag&&  
+(1/4x_1^{3}+13/4x_1^{2}+27/2x_1+18)g_{-1}^{3}g_{-2}^{2}g_{-3}\\\notag&&
+(1/2x_1^{3}+13/2x_1^{2}+27x_1+36)g_{-8}g_{-4}\\\notag&&  
+(-1/4x_1^{3}-23/8x_1^{2}-87/8x_1-27/2)g_{-4}g_{-1}g_{-3}^{2}g_{-2} \\\notag&&
+ (1/4x_1^{3}+13/4x_1^{2}+27/2x_1+18)g_{-6}g_{-1}^{2}g_{-2} \\\notag&&  
+(1/2x_1^{3}+13/2x_1^{2}+27x_1+36)g_{-4}^{2}g_{-1}g_{-3}\\\notag&&
+(-1/2x_1^{3}-13/2x_1^{2}-27x_1-36)g_{-4}g_{-1}^{2}g_{-2}g_{-3}\\\notag&&  
+(-1/4x_1^{3}-13/4x_1^{2}-27/2x_1-18)g_{-6}g_{-4}g_{-1}\\\notag&&
+(-5/4x_1^{2}-35/4x_1-15)g_{-9}g_{-1})    \cdot v_\lambda
\end{eqnarray}
are linearly independent $\bar\gob$-singular vectors. 
\item Suppose $\lambda=x_1\omega_1+2\omega_3$. Then
\begin{eqnarray}\notag
v_{\lambda,0}&:=&v_\lambda \\\notag
v_{\lambda,1}&:=&\left((-x_{1})g_{-3}+2 g_{-1}\right)\cdot v_\lambda
\\\notag
v_{\lambda,2}&:=&\left((2x_{1}^{2}-2x_{1}) g_{-3}^{2}+(-4x_{1}+4) g_{-1} g_{-3}+4 g_{-1}^{2}\right)\cdot v_\lambda
\\\notag
v_{\lambda,3}&:=&\left((-2x_{1}-8) g_{-6}+(2x_{1}+8) g_{-4} g_{-3}+(2x_{1}^{2}+14x_{1}+24) g_{-3} g_{-2} g_{-3}\right. \\\notag&&
\left.+(-x_{1}^{2}-7x_{1}-12) g_{-2} g_{-3}^{2}+(-2x_{1}-8) g_{-1} g_{-2} g_{-3}\right)\cdot  v_\lambda
\\\notag
v_{\lambda,4}&:=& \left((-x_{1}^{2}-10x_{1}-21) g_{-8}\right.\\\notag&&
+(-1/2x_{1}^{3}-7/2x_{1}^{2}-15/2x_{1}-9/2) g_{-6} g_{-3}\\\notag&&
+(x_{1}^{2}+7x_{1}+12) g_{-6} g_{-1}+(x_{1}^{3}+13/2x_{1}^{2}+10x_{1}-3/2) g_{-4} g_{-3}^{2}\\\notag&&
+(-x_{1}^{2}-7x_{1}-12) g_{-4} g_{-1} g_{-3}+(1/2x_{1}^{4}+5x_{1}^{3}\\\notag&&
+33/2x_{1}^{2}+18x_{1}) g_{-3}^{2} g_{-2} g_{-3}\\\notag&&
+(-x_{1}^{3}-19/2x_{1}^{2}-28x_{1}-51/2) g_{-1} g_{-3} g_{-2} g_{-3}\\\notag&&
+(3/2x_{1}^{2}+9x_{1}+27/2) g_{-1} g_{-2} g_{-3}^{2}\\\notag&&
\left.+(x_{1}^{2}+7x_{1}+12) g_{-1}^{2} g_{-2} g_{-3}\right)\cdot v_\lambda
\\\notag
v_{\lambda,5}&:=&\left(
(4/3x_{1}^{3}+38/3x_{1}^{2}+118/3x_{1}+40) g_{-9} g_{-3}\right. \\\notag&&
+(-4/3x_{1}^{2}-28/3x_{1}-16) g_{-9} g_{-1} \\\notag&&
+(-4/3x_{1}^{2}-28/3x_{1}-16) g_{-8} g_{-4}\\\notag&&
+(4/3x_{1}^{3}+38/3x_{1}^{2}+118/3x_{1}+40) g_{-8} g_{-2} g_{-3}\\\notag&&
+(4/3x_{1}^{3}+14x_{1}^{2}+146/3x_{1}+56) g_{-6} g_{-4} g_{-3}\\\notag&&
+(4/3x_{1}^{4}+18x_{1}^{3}+90x_{1}^{2}+592/3x_{1}+160) g_{-6} g_{-3} g_{-2} g_{-3}\\\notag&&
+(-2/3x_{1}^{4}-9x_{1}^{3}-45x_{1}^{2}-296/3x_{1}-80) g_{-6} g_{-2} g_{-3}^{2}\\\notag&&
+(-4/3x_{1}^{3}-14x_{1}^{2}-146/3x_{1}-56) g_{-6} g_{-1} g_{-2} g_{-3}\\\notag&&
+(-2/3x_{1}^{3}-22/3x_{1}^{2}-80/3x_{1}-32) g_{-6}^{2}\\\notag&&
+(-4/3x_{1}^{4}-52/3x_{1}^{3}-251/3x_{1}^{2}-533/3x_{1}-140) g_{-4} g_{-3}^{2} g_{-2} g_{-3}\\\notag&&
+(4/3x_{1}^{3}+14x_{1}^{2}+146/3x_{1}+56) g_{-4} g_{-1} g_{-2} g_{-3}^{2}\\\notag&&
+(-4/3x_{1}^{3}-14x_{1}^{2}-146/3x_{1}-56) g_{-4}^{2} g_{-3}^{2}\\\notag&&
+(-1/3x_{1}^{5}-11/2x_{1}^{4}-36x_{1}^{3}-701/6x_{1}^{2}-188x_{1}-120) g_{-3}^{2} g_{-2}^{2} g_{-3}^{2}\\\notag&&
+(2/3x_{1}^{4}+26/3x_{1}^{3}+251/6x_{1}^{2}+533/6x_{1}+70) g_{-1} g_{-3} g_{-2}^{2} g_{-3}^{2}\\\notag&&
\left.+(-2/3x_{1}^{3}-7x_{1}^{2}-73/3x_{1}-28) g_{-1}^{2} g_{-2}^{2} g_{-3}^{2} \right)\cdot v_\lambda\notag
\end{eqnarray}
are linearly independent $\bar\gob$-singular vectors. Moreover, for $x_1\notin$ $\{ -5, -3, -6,$ $-4, -7/2,$ $-1, -7, 0, -2\}$, 
the $\LieGtwo$-submodules generated by $v_{\lambda, i}$ have pairwise empty intersections. 
} %onlineVersionOnly
\end{enumerate}

\end{theorem}
\offlineVersionOnly{
The proof of this theorem can be found in the extended version of this 
text available online, \cite{MilevSomberg:BranchingExtended}.
}
\onlineVersionOnly{
\begin{proof}
Fix the base field to be the field $\mathbb C(x_1)$ of rational functions in the variable $x_1$.

We recall $\gol$ denotes the reductive Levi part of $\gop_{(1,0,0)}$. 
Then $\gol\simeq\mathbb C h_1 \oplus so(5)$, 
where the $so(5)$-part is generated by the simple Lie algebra generators $g_{ 2}$, $g_{-2}$, 
$ g_{3}$ and $g_{-3}$. The finite dimensional $\gol$-module $V_{\lambda}(\gol)$, 
inducing the generalized Verma module $M_\lambda(\gog, \gop)$, has a basis of the form 
$u_i\cdot v_\lambda$, where each element $u_i $ is a product of elements of the 
form  $g_{-2}^{s_j}g_{-3}^{t_j}$, where $s_j, t_j\in \mathbb Z_{\geq 0}$. 
Such monomial bases as well as the actions of $g_2$ and $g_3$ on them are given in the following table.

\footnotesize
\begin{tabular}{lllll} \end{tabular}

\begin{longtable}{|lllll|}
\caption{Monomial bases of the inducing modules of certain generalized Verma modules for $\gop_{(1,0,0)}$ \label{tableMonBases(1,0,0)}  
The weights are sorted in increasing gr-lex-order 
with respect to their coordinates in simple basis 
(weights that are lower in the table have gr-lex-higher weight). } \\
\hline \multicolumn{5}{|c|}{ Highest weight $\lambda=x_{1}\omega_{1}+\omega_{2}$ (restricts to natural $so(5)$-module)}\\\hline Element& weight & monomial expression&action of $g_{2}$&action of $g_{3}$\\ $m_{1}$& $(x_{1}+2)\omega_{1}-\omega_{2}$&$g_{-2}g_{-3}^{2}g_{-2} v_\lambda$&$m_{2}$&$0$\\ $m_{2}$& $(x_{1}+1)\omega_{1}+\omega_{2}-2\omega_{3}$&$g_{-3}^{2}g_{-2} v_\lambda$&$0$&$2m_{3}$\\ $m_{3}$& $(x_{1}+1)\omega_{1}$&$g_{-3}g_{-2} v_\lambda$&$0$&$2m_{4}$\\ $m_{4}$& $(x_{1}+1)\omega_{1}-\omega_{2}+2\omega_{3}$&$g_{-2} v_\lambda$&$m_{5}$&$0$\\ $m_{5}$& $x_{1}\omega_{1}+\omega_{2}$&$ v_\lambda$&$0$&$0$\\
\hline\hline \multicolumn{5}{|c|}{ Highest weight $\lambda=x_{1}\omega_{1}+\omega_{3}$ (restricts to spinor $so(5)$-module)}\\\hline Element& weight & monomial expression&action of $g_{2}$&action of $g_{3}$\\ $m_{1}$& $(x_{1}+1)\omega_{1}-\omega_{3}$&$g_{-3}g_{-2}g_{-3} v_\lambda$&$0$&$m_{2}$\\ $m_{2}$& $(x_{1}+1)\omega_{1}-\omega_{2}+\omega_{3}$&$g_{-2}g_{-3} v_\lambda$&$m_{3}$&$0$\\ $m_{3}$& $x_{1}\omega_{1}+\omega_{2}-\omega_{3}$&$g_{-3} v_\lambda$&$0$&$m_{4}$\\ $m_{4}$& $x_{1}\omega_{1}+\omega_{3}$&$ v_\lambda$&$0$&$0$\\ 
\hline\hline \multicolumn{5}{|c|}{ Highest weight $\lambda=x_{1}\omega_{1}+2\omega_{2}$}\\\hline Element& weight & monomial expression&action of $g_{2}$&action of $g_{3}$\\ $m_{1}$& $(x_{1}+4)\omega_{1}-2\omega_{2}$&$g_{-2}^{2}g_{-3}^{4}g_{-2}^{2} v_\lambda$&$2m_{2}$&$0$\\ $m_{2}$& $(x_{1}+3)\omega_{1}-2\omega_{3}$&$g_{-2}g_{-3}^{4}g_{-2}^{2} v_\lambda$&$2m_{3}$&$4m_{4}$\\ $m_{3}$& $(x_{1}+2)\omega_{1}+2\omega_{2}-4\omega_{3}$&$g_{-3}^{4}g_{-2}^{2} v_\lambda$&$0$&$4m_{5}$\\ $m_{4}$& $(x_{1}+3)\omega_{1}-\omega_{2}$&$g_{-2}g_{-3}^{3}g_{-2}^{2} v_\lambda$&$m_{5}$&$6m_{6}$\\ $m_{5}$& $(x_{1}+2)\omega_{1}+\omega_{2}-2\omega_{3}$&$g_{-3}^{3}g_{-2}^{2} v_\lambda$&$0$&$6m_{8}$\\ $m_{6}$& $(x_{1}+3)\omega_{1}-2\omega_{2}+2\omega_{3}$&$g_{-2}^{2}g_{-3}^{2}g_{-2} v_\lambda$&$2m_{7}$&$0$\\ $m_{7}$& $(x_{1}+2)\omega_{1}$&$g_{-2}g_{-3}^{2}g_{-2} v_\lambda$&$2m_{9}$&$m_{10}$\\ $m_{8}$& $(x_{1}+2)\omega_{1}$&$g_{-3}^{2}g_{-2}^{2} v_\lambda$&$2m_{9}$&$6m_{10}$\\ $m_{9}$& $(x_{1}+1)\omega_{1}+2\omega_{2}-2\omega_{3}$&$g_{-3}^{2}g_{-2} v_\lambda$&$0$&$2m_{11}$\\ $m_{10}$& $(x_{1}+2)\omega_{1}-\omega_{2}+2\omega_{3}$&$g_{-3}g_{-2}^{2} v_\lambda$&$2m_{11}$&$4m_{12}$\\ $m_{11}$& $(x_{1}+1)\omega_{1}+\omega_{2}$&$g_{-3}g_{-2} v_\lambda$&$0$&$2m_{13}$\\ $m_{12}$& $(x_{1}+2)\omega_{1}-2\omega_{2}+4\omega_{3}$&$g_{-2}^{2} v_\lambda$&$2m_{13}$&$0$\\ $m_{13}$& $(x_{1}+1)\omega_{1}+2\omega_{3}$&$g_{-2} v_\lambda$&$2m_{14}$&$0$\\ $m_{14}$& $x_{1}\omega_{1}+2\omega_{2}$&$ v_\lambda$&$0$&$0$\\
\hline\hline \multicolumn{5}{|c|}{ Highest weight $\lambda=x_{1}\omega_{1}+\omega_{2}+\omega_{3}$}\\\hline Element& weight & monomial expression&action of $g_{2}$&action of $g_{3}$\\ $m_{1}$& $(x_{1}+3)\omega_{1}-\omega_{2}-\omega_{3}$&$g_{-3}g_{-2}^{2}g_{-3}^{3}g_{-2} v_\lambda$&$m_{2}$&$m_{3}$\\ $m_{2}$& $(x_{1}+2)\omega_{1}+\omega_{2}-3\omega_{3}$&$g_{-3}^{3}g_{-2}^{2}g_{-3} v_\lambda$&$0$&$3m_{5}$\\ $m_{3}$& $(x_{1}+3)\omega_{1}-2\omega_{2}+\omega_{3}$&$g_{-2}^{2}g_{-3}^{3}g_{-2} v_\lambda$&$2m_{4}$&$0$\\ $m_{4}$& $(x_{1}+2)\omega_{1}-\omega_{3}$&$g_{-2}g_{-3}^{3}g_{-2} v_\lambda$&$2m_{6}$&$3m_{7}$\\ $m_{5}$& $(x_{1}+2)\omega_{1}-\omega_{3}$&$g_{-3}^{2}g_{-2}^{2}g_{-3} v_\lambda$&$2/3m_{6}$&$4m_{8}$\\ $m_{6}$& $(x_{1}+1)\omega_{1}+2\omega_{2}-3\omega_{3}$&$g_{-3}^{3}g_{-2} v_\lambda$&$0$&$3m_{10}$\\ $m_{7}$& $(x_{1}+2)\omega_{1}-\omega_{2}+\omega_{3}$&$g_{-2}g_{-3}^{2}g_{-2} v_\lambda$&$m_{10}$&$2m_{11}$\\ $m_{8}$& $(x_{1}+2)\omega_{1}-\omega_{2}+\omega_{3}$&$g_{-3}g_{-2}^{2}g_{-3} v_\lambda$&$2m_{9}$&$3m_{11}$\\ $m_{9}$& $(x_{1}+1)\omega_{1}+\omega_{2}-\omega_{3}$&$g_{-3}g_{-2}g_{-3} v_\lambda$&$0$&$m_{12}+m_{13}$\\ $m_{10}$& $(x_{1}+1)\omega_{1}+\omega_{2}-\omega_{3}$&$g_{-3}^{2}g_{-2} v_\lambda$&$0$&$4m_{13}$\\ $m_{11}$& $(x_{1}+2)\omega_{1}-2\omega_{2}+3\omega_{3}$&$g_{-2}^{2}g_{-3} v_\lambda$&$2m_{12}$&$0$\\ $m_{12}$& $(x_{1}+1)\omega_{1}+\omega_{3}$&$g_{-2}g_{-3} v_\lambda$&$2m_{14}$&$m_{15}$\\ $m_{13}$& $(x_{1}+1)\omega_{1}+\omega_{3}$&$g_{-3}g_{-2} v_\lambda$&$m_{14}$&$3m_{15}$\\ $m_{14}$& $x_{1}\omega_{1}+2\omega_{2}-\omega_{3}$&$g_{-3} v_\lambda$&$0$&$m_{16}$\\ $m_{15}$& $(x_{1}+1)\omega_{1}-\omega_{2}+3\omega_{3}$&$g_{-2} v_\lambda$&$m_{16}$&$0$\\ $m_{16}$& $x_{1}\omega_{1}+\omega_{2}+\omega_{3}$&$ v_\lambda$&$0$&$0$\\ 
\hline\hline \multicolumn{5}{|c|}{ Highest weight $\lambda=x_{1}\omega_{1}+2\omega_{3}$}\\\hline Element& weight & monomial expression&action of $g_{2}$&action of $g_{3}$\\ $m_{1}$& $(x_{1}+2)\omega_{1}-2\omega_{3}$&$g_{-3}^{2}g_{-2}^{2}g_{-3}^{2} v_\lambda$&$0$&$2m_{2}$\\ $m_{2}$& $(x_{1}+2)\omega_{1}-\omega_{2}$&$g_{-3}g_{-2}^{2}g_{-3}^{2} v_\lambda$&$2m_{3}$&$2m_{4}$\\ $m_{3}$& $(x_{1}+1)\omega_{1}+\omega_{2}-2\omega_{3}$&$g_{-3}^{2}g_{-2}g_{-3} v_\lambda$&$0$&$2m_{5}$\\ $m_{4}$& $(x_{1}+2)\omega_{1}-2\omega_{2}+2\omega_{3}$&$g_{-2}^{2}g_{-3}^{2} v_\lambda$&$2m_{6}$&$0$\\ $m_{5}$& $(x_{1}+1)\omega_{1}$&$g_{-3}g_{-2}g_{-3} v_\lambda$&$m_{7}$&$2m_{8}$\\ $m_{6}$& $(x_{1}+1)\omega_{1}$&$g_{-2}g_{-3}^{2} v_\lambda$&$2m_{7}$&$2m_{8}$\\ $m_{7}$& $x_{1}\omega_{1}+2\omega_{2}-2\omega_{3}$&$g_{-3}^{2} v_\lambda$&$0$&$2m_{9}$\\ $m_{8}$& $(x_{1}+1)\omega_{1}-\omega_{2}+2\omega_{3}$&$g_{-2}g_{-3} v_\lambda$&$m_{9}$&$0$\\ $m_{9}$& $x_{1}\omega_{1}+\omega_{2}$&$g_{-3} v_\lambda$&$0$&$2m_{10}$\\ $m_{10}$& $x_{1}\omega_{1}+2\omega_{3}$&$ v_\lambda$&$0$&$0$\\ \hline
\end{longtable}
\normalsize
By Lemma \ref{leParabolicsG2inParabolicsB3}, 
$\bar\gop=\bar{\gop}_{(1,0)}\subset\LieGtwo$, 
and $i(\bar \gos)=i([\bar \gol,\bar \gol])$ equals 
the $sl(2)$-subalgebra of $so(7)$ generated by $g_{ 2}$ 
and $g_{-2}$. %In the same time, $i([\bar \gol,\bar \gol])\subset \gol\simeq \mathbb C h_1\oplus so(5)$. 

We note that $\mathbb C h_2\oplus\mathbb C  g_2$ is a 
Borel subalgebra of $\bar\gos=[\bar \gol,\bar \gol]$. 
We can decompose $V_\lambda(\gol)$ over $\goh+i(\bar \gol) $ 
to obtain the decomposition indicated in the following table. 

\noindent
\begin{longtable}{|ll|}   \caption{Decomposition of $V_\lambda(\gol)$ over $\goh+i(\bar{\gol})$. For $\nu\in\goh^*$, we abbreviate $V_{\nu}(\goh+i(\bar{\gol}))$ as $V_{\nu}$. \label{tableDecompositionsOver(1,0,1)}}\\\endhead
\hline
$\lambda$  & Decomposition of $V_{\lambda}$ over $i(\bar{\gol})+\goh$ \\\hline
$x_1\omega_1+\omega_2$ &$\begin{array}{l}V_{x_1\omega_1+\omega_{2}}\oplus V_{(x_1+1)\omega_{1}}\\ \oplus V_{(x_1+1)\omega_{1}+\omega_{2}-2\omega_{3}}\end{array}$\\\hline
$x_1\omega_1+\omega_3$ &$\begin{array}{l}V_{x_1\omega_1+\omega_{3}}\oplus V_{x_1\omega_1+\omega_{2}-\omega_{3}}\\ \oplus V_{(x_1+1)\omega_{1}-\omega_{3}}\end{array}$\\\hline
$x_1\omega_1+2\omega_2$&$\begin{array}{l}
V_{x_{1}\omega_{1}+2\omega_{2}}\oplus V_{(x_{1}+1)\omega_{1} +\omega_{2}} \\ \oplus V_{(x_{1}+2)\omega_{1}}\oplus V_{(x_{1}+1)\omega_{1}+2\omega_{2}-2\omega_{3}} \\\oplus  V_{(x_{1}+2)\omega_{1}+\omega_{2}-2\omega_{3}}\oplus V_{(x_{1}+2)\omega_{1}+2\omega_{2}-4\omega_{3}}
\end{array}
$\\\hline
$x_1\omega_1+\omega_2+\omega_3$ &
$\begin{array}{l}
V_{x_1\omega_1+\omega_{2}+\omega_{3}}\oplus V_{x_1\omega_1+\omega_{1}+\omega_{3}}\\ 
\oplus  V_{x_1\omega_1+2\omega_{2}-\omega_{3}}\oplus  2V_{x_1\omega_1+\omega_{1}+\omega_{2}-\omega_{3}} \\ 
\oplus V_{x_1\omega_1+2\omega_{1}-\omega_{3}}\oplus V_{x_1\omega_1+\omega_{1}+2\omega_{2}-3\omega_{3}} \\
\oplus V_{x_1\omega_1+2\omega_{1}+\omega_{2}-3\omega_{3}}
\end{array}$ 
\\\hline
$x_1\omega_1+2\omega_3$ &$\begin{array}{l}
V_{x_{1}\omega_{1}+2\omega_{3}}\oplus V_{x_{1}\omega_{1}+\omega_{2}}\\ 
\oplus V_{(x_{1}+1)\omega_{1}}\oplus V_{x_{1}\omega_{1}+2\omega_{2}-2\omega_{3}} \\ 
\oplus V_{(x_{1}+1)\omega_{1}+\omega_{2}-2\omega_{3}}\oplus V_{(x_{1}+2)\omega_{1}-2\omega_{3}}
\end{array}$\\\hline
\end{longtable}

The $\mathbb C h_2\oplus \mathbb C g_2$-singular vectors  realizing the above decomposition 
are computed as the eigenspace of the action of operator $g_2$ using 
the fourth column of Table \ref{tableMonBases(1,0,0)}. 
The resulting $\mathbb C h_2\oplus \mathbb C g_2$-singular vectors are given in the following table.

\footnotesize
\noindent\begin{longtable}{|llll|}
\caption{$\mathbb C h_2\oplus \mathbb C g_2$-singular vectors and $p_{1}(\mu)$ \label{tableHWVover(1,0,1)}}\\\endhead
\hline\multicolumn{4}{|c|}{Highest weight $\lambda= x_{1}\omega_{1}+\omega_{2}$}\\ weight & sing. vect.& projection $\mu\in\bar h^*$ & $p_1(\mu)$ \\\hline 
$(x_{1}+1)\omega_{1}+\omega_{2}-2\omega_{3} $ & $m_{2} $ & $ (2x_{1}+1)\alpha_{1}+(x_{1}+1)\alpha_{2} $  & $ 1/12x_{1}^{2}+1/2x_{1}+5/12 $ \\ 
$(x_{1}+1)\omega_{1} $ & $m_{3} $ & $ (2x_{1}+2)\alpha_{1}+(x_{1}+1)\alpha_{2} $ & $ 1/12x_{1}^{2}+7/12x_{1}+1/2 $ \\ 
$x_{1}\omega_{1}+\omega_{2} $ & $m_{5} $ & $ (2x_{1}+3)\alpha_{1}+(x_{1}+2)\alpha_{2} $ & $ 1/12x_{1}^{2}+2/3x_{1}+1 $ \\ 
\hline 
\hline\multicolumn{4}{|c|}{Highest weight $\lambda= x_{1}\omega_{1}+\omega_{3}$}\\ weight & sing. vect.& projection $\mu\in\bar h^*$ & $p_1(\mu)$ \\\hline 
$(x_{1}+1)\omega_{1}-\omega_{3} $ & $m_{1} $ & $ 2x_{1}\alpha_{1}+x_{1}\alpha_{2} $ & $ 1/12x_{1}^{2}+5/12x_{1} $\\ 
$x_{1}\omega_{1}+\omega_{2}-\omega_{3} $ & $m_{3} $ & $ (2x_{1}+1)\alpha_{1}+(x_{1}+1)\alpha_{2} $ & $ 1/12x_{1}^{2}+1/2x_{1}+5/12 $\\ 
$x_{1}\omega_{1}+\omega_{3} $ & $m_{4} $ & $ (2x_{1}+2)\alpha_{1}+(x_{1}+1)\alpha_{2} $ & $ 1/12x_{1}^{2}+7/12x_{1}+1/2 $\\ 
\hline 
\hline\multicolumn{4}{|c|}{Highest weight $\lambda= x_{1}\omega_{1}+2\omega_{2}$}\\ weight & sing. vect.& projection $\mu\in\bar h^*$ & $p_1(\mu)$ \\\hline 
$(x_{1}+2)\omega_{1}+2\omega_{2}-4\omega_{3} $ & $m_{3} $ & $ (2x_{1}+2)\alpha_{1}+(x_{1}+2)\alpha_{2} $ & $ 1/12x_{1}^{2}+7/12x_{1}+1 $\\ 
$(x_{1}+2)\omega_{1}+\omega_{2}-2\omega_{3} $ & $m_{5} $ & $ (2x_{1}+3)\alpha_{1}+(x_{1}+2)\alpha_{2} $ & $ 1/12x_{1}^{2}+2/3x_{1}+1 $\\ 
$(x_{1}+2)\omega_{1} $ & $-m_{7}+m_{8} $ & $ (2x_{1}+4)\alpha_{1}+(x_{1}+2)\alpha_{2} $ & $ 1/12x_{1}^{2}+3/4x_{1}+7/6 $\\ 
$(x_{1}+1)\omega_{1}+2\omega_{2}-2\omega_{3} $ & $m_{9} $ & $ (2x_{1}+4)\alpha_{1}+(x_{1}+3)\alpha_{2} $ & $ 1/12x_{1}^{2}+3/4x_{1}+5/3 $ \\ 
$(x_{1}+1)\omega_{1}+\omega_{2} $ & $m_{11} $ & $ (2x_{1}+5)\alpha_{1}+(x_{1}+3)\alpha_{2} $ & $ 1/12x_{1}^{2}+5/6x_{1}+7/4 $ \\ 
$x_{1}\omega_{1}+2\omega_{2} $ & $m_{14} $ & $ (2x_{1}+6)\alpha_{1}+(x_{1}+4)\alpha_{2} $ & $ 1/12x_{1}^{2}+11/12x_{1}+5/2 $ \\ 
\hline 
\hline\multicolumn{4}{|c|}{Highest weight $\lambda= x_{1}\omega_{1}+\omega_{2}+\omega_{3}$}\\ weight &sing. vect.& projection $\mu\in\bar h^*$ & $p_1(\mu)$ \\\hline 
$(x_{1}+2)\omega_{1}+\omega_{2}-3\omega_{3} $ & $m_{2} $ & $ (2x_{1}+1)\alpha_{1}+(x_{1}+1)\alpha_{2} $& $ 1/12x_{1}^{2}+1/2x_{1}+5/12 $ \\ 
$(x_{1}+2)\omega_{1}-\omega_{3} $ & $-1/3m_{4}+m_{5} $ & $ (2x_{1}+2)\alpha_{1}+(x_{1}+1)\alpha_{2} $& $ 1/12x_{1}^{2}+7/12x_{1}+1/2 $ \\ 
$(x_{1}+1)\omega_{1}+2\omega_{2}-3\omega_{3} $ & $m_{6} $ & $ (2x_{1}+2)\alpha_{1}+(x_{1}+2)\alpha_{2} $& $ 1/12x_{1}^{2}+7/12x_{1}+1 $ \\ 
$(x_{1}+1)\omega_{1}+\omega_{2}-\omega_{3} $ & $m_{9} $ & $ (2x_{1}+3)\alpha_{1}+(x_{1}+2)\alpha_{2} $& $ 1/12x_{1}^{2}+2/3x_{1}+1 $ \\ 
$(x_{1}+1)\omega_{1}+\omega_{2}-\omega_{3} $ & $m_{10} $ & $ (2x_{1}+3)\alpha_{1}+(x_{1}+2)\alpha_{2} $& $ 1/12x_{1}^{2}+2/3x_{1}+1 $ \\ 
$(x_{1}+1)\omega_{1}+\omega_{3} $ & $-1/2m_{12}+m_{13} $ & $ (2x_{1}+4)\alpha_{1}+(x_{1}+2)\alpha_{2} $& $ 1/12x_{1}^{2}+3/4x_{1}+7/6 $ \\ 
$x_{1}\omega_{1}+2\omega_{2}-\omega_{3} $ & $m_{14} $ & $ (2x_{1}+4)\alpha_{1}+(x_{1}+3)\alpha_{2} $& $ 1/12x_{1}^{2}+3/4x_{1}+5/3 $ \\ 
$x_{1}\omega_{1}+\omega_{2}+\omega_{3} $ & $m_{16} $ & $ (2x_{1}+5)\alpha_{1}+(x_{1}+3)\alpha_{2} $& $ 1/12x_{1}^{2}+5/6x_{1}+7/4 $ \\ 
\hline 
\hline\multicolumn{4}{|c|}{Highest weight $\lambda= x_{1}\omega_{1}+2\omega_{3}$}\\ weight &sing. vect.& projection $\mu\in\bar h^*$ & $p_1(\mu)$ \\\hline 
$(x_{1}+2)\omega_{1}-2\omega_{3} $ & $m_{1} $ & $ 2x_{1}\alpha_{1}+x_{1}\alpha_{2} $& $ 1/12x_{1}^{2}+5/12x_{1} $ \\ 
$(x_{1}+1)\omega_{1}+\omega_{2}-2\omega_{3} $ & $m_{3} $ & $ (2x_{1}+1)\alpha_{1}+(x_{1}+1)\alpha_{2} $& $ 1/12x_{1}^{2}+1/2x_{1}+5/12 $ \\ 
$(x_{1}+1)\omega_{1} $ & $-2m_{5}+m_{6} $ & $ (2x_{1}+2)\alpha_{1}+(x_{1}+1)\alpha_{2} $& $ 1/12x_{1}^{2}+7/12x_{1}+1/2 $ \\ 
$x_{1}\omega_{1}+2\omega_{2}-2\omega_{3} $ & $m_{7} $ & $ (2x_{1}+2)\alpha_{1}+(x_{1}+2)\alpha_{2} $& $ 1/12x_{1}^{2}+7/12x_{1}+1 $ \\ 
$x_{1}\omega_{1}+\omega_{2} $ & $m_{9} $ & $ (2x_{1}+3)\alpha_{1}+(x_{1}+2)\alpha_{2} $& $ 1/12x_{1}^{2}+2/3x_{1}+1 $ \\ 
$x_{1}\omega_{1}+2\omega_{3} $ & $m_{10} $ & $ (2x_{1}+4)\alpha_{1}+(x_{1}+2)\alpha_{2} $& $ 1/12x_{1}^{2}+3/4x_{1}+7/6 $ \\ 
\hline 
\end{longtable}
\normalsize

Let $\bar \lambda:=\pr (\lambda)$ be the projection of $\lambda$ 
onto $\bar h^*$. By Theorem \ref{thBranchingMultsViaSymmetricTensors} 
and Lemma \ref{leParabolicsG2inParabolicsB3} (c).3 we have that 
$\gop_{(1,0,0)}$ has a finite branching problem over $i(\LieGtwo)$  and 
therefore the coefficient 
of $V_\nu$ in Table \ref{tableDecompositionsOver(1,0,1)} 
equals $m(\pr(\nu), \pr(\lambda))$. This correspondence is 
one to one as no two different weights $\nu$ appearing in 
Table \ref{tableDecompositionsOver(1,0,1)} project to the same weight in 
$\bar \goh^*$ by Lemma \ref{leParabolicsG2inParabolicsB3} (c).3, 
as one sees in the third column of Table \ref{tableHWVover(1,0,1)}.

The quadratic Casimir element $\bar c_1$ of $\LieGtwo$  is given by
\begin{eqnarray}\label{eqCasimirG2}
36 \bar c_1&=&{\bar h}_{1}^{2}+3 {\bar h}_{1} {\bar h}_{2}+3 {\bar h}_{2}^{2}+15 {\bar h}_{2}+9 {\bar g}_{-6} {\bar g}_{6}\\\notag&&
+9 {\bar h}_{1}+9 {\bar g}_{-5} {\bar g}_{5}+3 {\bar g}_{-4} {\bar g}_{4}+3 {\bar g}_{-3} {\bar g}_{3}+3 {\bar g}_{-2} {\bar g}_{2}+9 {\bar g}_{-1} {\bar g}_{1}\quad ,
\end{eqnarray}
and its embedding $i(\bar c_1)$ is given by
\begin{eqnarray*}
12 i(\bar c_1)&=&3h_{2}^{2}+3h_{1}h_{2}+6h_{2}h_{3}+h_{1}^{2}+4h_{1}h_{3}+4h_{3}^{2}+10h_{3}+3g_{-9}g_{9}+5h_{1}+9h_{2}\\&&
+3g_{-8}g_{8}+g_{-7}g_{6}+g_{-6}g_{6}+g_{-7}g_{7}+g_{-6}g_{7}-g_{-5}g_{4}+g_{-4}g_{4}\\&&
+g_{-5}g_{5}-g_{-4}g_{5}+g_{-3}g_{1}+g_{-1}g_{1}+g_{-3}g_{3}+g_{-1}g_{3}+3g_{-2}g_{2} \quad .
\end{eqnarray*}

Let $v_\mu$ be the highest weight vector of a $\LieGtwo$-highest 
weight module of highest weight $\mu=y_1\psi_1+y_2\psi_2$, 
(we recall $\psi_1=\alpha_1+2\alpha_2$ and $\psi_2=2\alpha_1+3\alpha_2$ are the 
fundamental weights of $ \LieGtwo$). Then $\bar c_1$ acts on $v_\mu$ by a constant, 
i.e., $\bar c_1\cdot v_\mu=p_1(\mu) v_\mu$ (see Section \ref{secConstructingSingularVectors}). 
To compute the coefficient $p_1(\mu)$, in the expression for $\bar c_1$  
we  carry out the substitutions $\bar h_2\mapsto 3 y_1 , \bar h_1\mapsto y_2$, 
and set all terms involving the generators $\bar g_i$ to zero:
\[
p_1(\mu):=1/12 y_{2}^{2}+ 1/4 y_{1} y_{2}+ 5/12 y_{2}+ 1/4 y_{1}^{2}+ 3/4 y_{1}\quad .
\]

Therefore we can compute the action of $i(\bar c_1)$ 
on each component of a decomposition of $M_\lambda(\gog,\gop)$ 
into $\bar \gob$-highest modules in the last column of Table \ref{tableHWVover(1,0,1)}.

Consider a $\mathbb C h_2\oplus \mathbb C g_2$-singular vector $m$ of 
$\goh$-weight $\mu$ given in the second column of Table \ref{tableHWVover(1,0,1)}. 
As the base field is $\mathbb C(x_1)$, the strong Condition B 
(Definition \ref{defSufficientlyGenericWeight}) is trivially satisfied: 
indeed, for pairwise different values of $\mu$, the values $p_1(\mu)$ in the fourth column of 
Table \ref{tableHWVover(1,0,1)} are pairwise different elements of $C(x_1)$.
Therefore by Theorem \ref{thSufficientlyGenericWeightEnablesFDbranching}, 
the vectors of the form 
$\left(\prod_{\nu\succ\mu} (i(\bar c_1)-p_1(\nu))\right) \cdot m $ are $\bar \gob$-singular. 
The vectors $v_{\lambda,k}$, given in the statement of the current theorem, 
are indeed up to a scalar in $\mathbb C(x_1)$ equal 
to $\left(\prod_{\nu\succ\mu} (i(\bar c_1)-y(\nu))\right) \cdot m $. 
Since the monomials used in the expressions for 
$v_{\lambda,i}$ are linearly independent and the coefficients 
do not have a common zero, substitutions in the 
variable $x_1$ yields non-zero vectors defined over the field $\mathbb C$. 
As the structure constants of $so(7)$ 
are identical over all fields of characteristic 0, 
substitutions of the variable $x_1$ give the desired 
$\bar \gob$-singular vectors. 

For $x_i$ not belonging to the sets listed in the 
statement of the theorem, the values of $p_1(\mu)$ are pairwise different.
Therefore for those values the $\LieGtwo$-modules 
generated by the vectors $v_{i,\lambda}$ have pairwise empty 
intersections, which completes the proof of the Theorem.
\end{proof}
} % onlineVersionOnly.

Except for finitely many values of $x_1$, the formulas for $\gob$-singular 
vectors given in Theorem \ref{leBranchingExplicit(x_1,0,0),(1,a,b)} 
hold in arbitrary highest weight modules of highest weight $\lambda$. 
In fact, the following observation holds.
\begin{corollary}\label{corFormulasEigenVectorsRemainValid}
Let $\lambda\in\goh^*$ be one of the weights
given in Theorem \ref{leBranchingExplicit(x_1,0,0),(1,a,b)}. 
Let $M$ be a $so(7)$-module that has a $\gob$-singular vector $v_\lambda$ 
of $\goh$-weight $\lambda$.
Let the vectors $v_{\lambda,i}$  be defined by the same formulas as in
Theorem \ref{leBranchingExplicit(x_1,0,0),(1,a,b)}. For 
a fixed $v_{\lambda,i}$, let $x_1$ be a number different from the numbers indicated in 
the right column of Table \ref{tableCriticalValuesvlambda} below.
Recall $\bar{\gob}$ is the Borel subalgebra of $\LieGtwo$.

Then  $v_{\lambda,i}$ is non-zero and therefore a $\bar{\gob}$-singular vector 
under the action on $M$ induced by the embedding $\LieAlgPair{\LieGtwo}{so(7)}$.

\begin{longtable}{|rll|}\caption{Values of $x_1$ for each $v_{\lambda,i}$}\label{tableCriticalValuesvlambda}\endhead\hline\multicolumn{3}{|c|}{$\lambda=x_{1}\omega_{1}+\omega_{3}$}
\\
vector& coefficient of $v_\lambda$ in $Sh_{\lambda,i}$ &$x_1\notin $ roots of $Sh_{\lambda,i}$=  \\\hline 
$v_{\lambda,2} $&$\begin{array}{l}2x_{1}^{4}+27x_{1}^{3}  +133x_{1}^{2}+285x_{1}  +225\end{array}$ 
& -3, -5, -5/2 \\ 
$v_{\lambda,1} $&$\begin{array}{l}x_{1}^{2}+x_{1}\end{array}$ 
& 0, -1 \\\hline
\multicolumn{3}{|c|}{$\lambda=x_{1}\omega_{1}+\omega_{2}$}\\vector& coefficient of $v_\lambda$ in $Sh_{\lambda,i}$ &$x_1\notin$ roots of $Sh_{\lambda,i}$= \\\hline $v_{\lambda,2} $&$\begin{array}{l}2x_{1}^{4}+13x_{1}^{3} +25x_{1}^{2}+14x_{1}\end{array}$ 
& 0, -1, -2, -7/2 \\ 
$v_{\lambda,1} $&$\begin{array}{l}x_{1}^{2}+8x_{1}+12\end{array}$ 
& -2, -6 \\\hline
\offlineVersionOnly{
\multicolumn{3}{|c|}{$\lambda=x_{1}\omega_{1}+2\omega_{3}$, $\lambda=x_{1}\omega_{1}+\omega_2+\omega_{3}$, $\lambda=x_{1}\omega_{1}+2\omega_{3}$
}\\
\multicolumn{3}{|c|}{ entries listed in the extended version of this 
text available online, \cite{MilevSomberg:BranchingExtended}.
}\\\hline
} %offlineVersionOnly
\onlineVersionOnly{
\multicolumn{3}{|c|}{$\lambda=x_{1}\omega_{1}+2\omega_{3}$}\\vector& coefficient of $v_\lambda$ in $Sh_{\lambda,i}$ &$x_1\notin$ roots of $Sh_{\lambda,i}$= \\\hline 
$v_{\lambda,5} $
&
$\begin{array}{l}4x_{1}^{10}+152x_{1}^{9}  +2579x_{1}^{8}+25736x_{1}^{7} \\ +167312x_{1}^{6}  +740582x_{1}^{5}  +2260753x_{1}^{4}\\
+4700490x_{1}^{3}  +6371352x_{1}^{2} \\ +5084640x_{1}  +1814400\end{array}$
& -7/2,-5/2,-3,-5,-4,-6 \\ 
$v_{\lambda,4} $&$
\begin{array}{l}2x_{1}^{8}+49x_{1}^{7}  +495x_{1}^{6}+2678x_{1}^{5} \\ +8368x_{1}^{4}+15021x_{1}^{3}   +14175x_{1}^{2} \\ +5292x_{1}\end{array}$ 
& 0, -1, -3, -4, -7, -7/2 \\ 
$v_{\lambda,3} $&$\begin{array}{l}x_{1}^{4}+17x_{1}^{3}  +106x_{1}^{2}+288x_{1}  +288\end{array}$ 
& -3, -4, -6 \\ $v_{\lambda,2} $&$\begin{array}{l}x_{1}^{4}-x_{1}^{2}\end{array}$ 
& 0, 1, -1 \\ $v_{\lambda,1} $&$\begin{array}{l}x_{1}^{2}+2x_{1}\end{array}$ 
& 0, -2 \\\hline\multicolumn{3}{|c|}{$\lambda=x_{1}\omega_{1}+\omega_{2}+\omega_{3}$}\\vector& coefficient of $v_\lambda$ in $Sh_{\lambda,i}$ &$x_1\notin$ roots of $Sh_{\lambda,i}$= \\\hline 
$v_{\lambda,7} $&
$\begin{array}{l}16x_{1}^{14}+784x_{1}^{13} \\ +17496x_{1}^{12}+235424x_{1}^{11} \\ +2130569x_{1}^{10}+13688787x_{1}^{9} \\ +64200218x_{1}^{8}+222353222x_{1}^{7} \\ +568050249x_{1}^{6}+1055574499x_{1}^{5} \\ +1383817036x_{1}^{4}+1208330004x_{1}^{3} \\ +627179616x_{1}^{2}+145212480x_{1}\end{array}$ & 
-7/2,-3,-4,-6,-5,0,-1,-2,-7
\\ 
$v_{\lambda,6} $&
$\begin{array}{l}2x_{1}^{10}+97x_{1}^{9}  +2097x_{1}^{8}+26595x_{1}^{7} \\ +218973x_{1}^{6}+1222044x_{1}^{5} \\ +4676800x_{1}^{4}+12104384x_{1}^{3} \\ +20244528x_{1}^{2}+19716480x_{1} \\ +8467200\end{array}$
& -7/2,-4,-6,-5,-7,-2
\\ $v_{\lambda,5} $&
$\begin{array}{l}2x_{1}^{10}+25x_{1}^{9}  +113x_{1}^{8}+205x_{1}^{7} \\ +53x_{1}^{6}-230x_{1}^{5}  -168x_{1}^{4}\end{array}$ 
& 0, 1, -1, -2, -3, -4, -7/2 \\ 
$v_{\lambda,4} $&
$\begin{array}{l}2x_{1}^{6}+47x_{1}^{5}  +431x_{1}^{4}+1978x_{1}^{3} \\ +4804x_{1}^{2} +5896x_{1}  +2880\end{array}$ 
& -2, -5, -8, -9/2 \\ 
$v_{\lambda,3} $&$\begin{array}{l}2x_{1}^{6}+47x_{1}^{5}  +411x_{1}^{4}+1648x_{1}^{3} \\ +3004x_{1}^{2}+2016x_{1}\end{array}$ 
& 0, -2, -7, -8, -9/2 \\ $v_{\lambda,2} $&$\begin{array}{l}x_{1}^{2}+9x_{1}+14\end{array}$ 
& -2, -7 \\ 
$v_{\lambda,1} $&$\begin{array}{l}x_{1}^{2}+x_{1}\end{array}$ 
& 0, -1 \\\hline\multicolumn{3}{|c|}{$\lambda=x_{1}\omega_{1}+2\omega_{2}$}\\vector& coefficient of $v_\lambda$ in $Sh_{\lambda,i}$ &$x_1\notin$ roots of $Sh_{\lambda,i}$= \\\hline $v_{\lambda,5} $&
$\begin{array}{l}4x_{1}^{10}+80x_{1}^{9} +655x_{1}^{8}+2780x_{1}^{7} \\ +6232x_{1}^{6}+5870x_{1}^{5}  -2355x_{1}^{4}-8730x_{1}^{3} \\ -4536x_{1}^{2}\end{array}$ & 0, 1, -1, -2, -3, -4, -7/2, -9/2 \\ 
$v_{\lambda,4} $&
$\begin{array}{l}2x_{1}^{8}+65x_{1}^{7}  +870x_{1}^{6}+6193x_{1}^{5}  +25234x_{1}^{4} \\ +58716x_{1}^{3}  +72216x_{1}^{2}+36288x_{1}\end{array}$
& -9/2,-2,-6,-7,-3,0,-8 \\ 
$v_{\lambda,3} $&$\begin{array}{l}x_{1}^{4}+20x_{1}^{3}  +137x_{1}^{2}+370x_{1}  +336\end{array}$ 
& -2, -3, -7, -8 \\ 
$v_{\lambda,2} $&$\begin{array}{l}x_{1}^{4}+9x_{1}^{3}+23x_{1}^{2}  +15x_{1}\end{array}$ 
& 0, -1, -3, -5 \\ 
$v_{\lambda,1} $&$\begin{array}{l}x_{1}^{2}+12x_{1}+27\end{array}$ 
& -3, -9 \\\hline
} %onlineVersionOnly
\end{longtable}
\end{corollary}
\begin{proof}
Let $\tau$ denote the transpose anti-automorphism of $U(\gog)$, i.e., 
the linear map of $U(\gog)$ defined by
\[\begin{array}{rclr}
\tau (g_{-\beta})&:=&g_{\beta} & \beta\in \Delta(so(7)) \\
\tau(g_{\beta_1}\dots g_{\beta_k})&:=&\tau(g_{\beta_k})\dots \tau(g_{\beta_1})& \beta_j\in \Delta(so(7)) \\
\tau(h)&:=& h \quad & h\in\goh\quad . 
\end{array}
\]
Let $u_{\lambda,i}\in U(so(7))$ be an element for which 
$v_{\lambda,i}=u_{\lambda,i}\cdot v_\lambda$ (one such element 
is given in Theorem \ref{leBranchingExplicit(x_1,0,0),(1,a,b)}). 
Define $Sh_{\lambda,i}:=\tau(u_{\lambda,i}) u_{\lambda,i} \cdot v_\lambda$. Although we 
will not use this, we note that $Sh_{\lambda,i}$ does 
not depend on the choice of $u_{\lambda,i}$ (see, e.g., \cite{JacksonMilev:ShapovalovForm}).

A short consideration shows that $Sh_{\lambda,i}$ is a multiple of $v_\lambda$. 
Up to a rational scalar  this multiple is indicated in the second column of Table
\ref{tableCriticalValuesvlambda}. The scalar is chosen 
so that all polynomials have integral relatively prime coefficients and 
the leading coefficient  is positive. 
The roots of the polynomials 
$Sh_{\lambda,i}$ are all rational  and are indicated in the last column of  Table \ref{tableCriticalValuesvlambda}
(some of the roots have multiplicity higher than 1). If $Sh_{\lambda,i}$ 
does not equal to zero for a given value of $x_1$, then $v_{\lambda,i}$ cannot vanish for that value of $x_1$. 
This proves the statement.
\end{proof}

Corollary \ref{corGenVermaDecompo} implies the following. The case of $\lambda=x_1\omega_1+\omega_2+\omega_3$ is excluded
as $m(x_1\psi_1+\psi_2, \lambda)=2$.

\begin{corollary}\label{corGenVermaDecomposB3G2}
We have the following $\bar{\gog}$-module isomorphisms.
\begin{enumerate}
\item 
\[
M_{x_1\omega_1}(so(7),\gop_{(1,0,0)})\simeq M_{x_1\psi_1}(\LieGtwo, \gop_{(1,0)})\quad .
\] 
\item Suppose  $x_1\notin\{-1, -7/2, -6\} $. Then
\begin{eqnarray*}
M_{x_1\omega_1+\omega_2}(so(7),\gop_{(1,0,0)}) &\simeq& \phantom{\oplus}
M_{x_1\psi_1+\psi_2}(\LieGtwo, \gop_{(1,0)}) \\&&
\oplus  
M_{(x_1+1)\psi_1}(\LieGtwo, \gop_{(1,0)}) \\&&
\oplus
M_{(x_1-1)\psi_{1}+\psi_{2}}(\LieGtwo, \gop_{(1,0)})\quad .
\end{eqnarray*}
\item Suppose $x_1\notin\{ -5, -3, -1\}$ . Then
\begin{eqnarray*}
M_{x_1\omega_1+\omega_3}(so(7),\gop_{(1,0,0)}) &\simeq& \phantom{\oplus}
M_{(x_1+1)\psi_1}(\LieGtwo, \gop_{(1,0)}) \\&& 
\oplus  
M_{(x_1-1)\psi_1+\psi_2}(\LieGtwo, \gop_{(1,0)}) \\&&
\oplus
M_{x_1\psi_{1}}(\LieGtwo, \gop_{(1,0)})\quad .
\end{eqnarray*}
\item Suppose $x_1\notin\{ 0, -1, -4, -3, -9/2, -2, -8, -6, -7, -5, -9\}$. Then
\begin{eqnarray*}
M_{x_1\omega_1+2\omega_2}(so(7),\gop_{(1,0,0)})&\simeq& \phantom{\oplus}
M_{x_{1}\psi_{1}+2\psi_{2}}(\LieGtwo, \gop_{(1,0)}) \\&&
\oplus
M_{(x_{1}+1)\psi_{1} +\psi_{2}}(\LieGtwo, \gop_{(1,0)}) \\&&
\oplus
M_{(x_{1}+2)\psi_{1}}(\LieGtwo, \gop_{(1,0)}) \\&&
\oplus
M_{(x_{1}-1)\psi_{1}+2\psi_{2}}(\LieGtwo, \gop_{(1,0)}) \\&&
\oplus
M_{x_{1}\psi_{1}+\psi_{2}}(\LieGtwo, \gop_{(1,0)}) \\&&
\oplus
M_{(x_{1}-2)\psi_{1}+2\psi_{2}}(\LieGtwo, \gop_{(1,0)}) \quad .
\end{eqnarray*}
\item Suppose  $x_1\notin \{ -5, -3, -6, -4, -7/2, -1, -7, 0, -2\}$. Then
\begin{eqnarray*}
M_{x_1\omega_1+2\omega_3}(so(7),\gop_{(1,0,0)})&\simeq& \phantom{\oplus}
M_{(x_{1}+2)\psi_{1}}(\LieGtwo, \gop_{(1,0)}) \\&&
\oplus
M_{x_{1}\psi_{1}+\psi_{2}}(\LieGtwo, \gop_{(1,0)}) \\&&
\oplus
M_{(x_{1}+1)\psi_{1}}(\LieGtwo, \gop_{(1,0)}) \\&&
\oplus
M_{(x_{1}-2)\psi_{1}+2\psi_{2}}(\LieGtwo, \gop_{(1,0)}) \\&&
\oplus
M_{(x_{1}-1)\psi_{1}+\psi_{2}}(\LieGtwo, \gop_{(1,0)}) \\&&
\oplus
M_{x_{1}\psi_{1}}(\LieGtwo, \gop_{(1,0)}) \quad .
\end{eqnarray*}
\end{enumerate}
\end{corollary}
%The preceding Corollary reduces the branching problem of $M_{\lambda}(so(7), \gop_{(1,0,0)})$ to the branching problem of $M_{\mu}(\LieGtwo, \gop_{(1,0)})$. For a fixed $\mu$, the number of $\bar{\gob}$-singular vectors in $M_{\mu}(\LieGtwo, \gop_{(1,0)})$ can be deduced from \cite{BoeCollingWood}.
%%%%%%%%%%%%%%%%%%%%%%%%%%%%%%%%%%%%%%%%%%%%%%%%%%%%%%%%%%%%%%%%%%%%%%%%%%%%%%%%%%%%%%%%%%%%

\subsection{Parabolic subalgebras $\gop\subset so(7)$ with non-finite branching problem over $i(\LieGtwo)$ and the case $\gop\simeq so(7)$}
\label{secInfiniteSo7G2branching}

For a parabolic subalgebra $\gop\subset so(7)$, let $\lambda(\mathbf x, \mathbf z)$ be 
as in Section \ref{secConstructingSingularVectors}.
\onlineVersionOnly{
In the current Section, for each parabolic subalgebra 
$\gop\subset so(7)$ other than $\gop_{(1,0,0)}$, for
$z_1+z_2+z_3\leq 5$, $z_i\in \mathbb Z_{\geq 0} $, we compute the strong 
Condition B (Definition \ref{defSufficientlyGenericWeight}). 
For $z_1+z_2+z_3\leq 2$, for each $\mu\in \bar{\goh}^*$ with $n(\mu,\lambda(\mathbf x, \mathbf z))>0$ 
we construct $n(\mu,\lambda(\mathbf x, \mathbf z))$ $\bar\gob$-singular vectors in 
$M_{\lambda(\mathbf x, \mathbf z)}(so(7),\gop)$ of weight $\mu$. 
}
\offlineVersionOnly{
For $z_1+z_2+z_3\leq 2$, for each $\mu\in \bar{\goh}^*$ with $n(\mu,\lambda(\mathbf x, \mathbf z))>0$, 
one can construct $n(\mu,\lambda(\mathbf x, \mathbf z))$ $\bar\gob$-singular vectors in 
$M_{\lambda(\mathbf x, \mathbf z)}(so(7),\gop)$ of weight $\mu$. Tables with the formulas 
arising through this construction can be found in the extended version of this 
text available online, \cite{MilevSomberg:BranchingExtended}.
}

As it turns out, the constructed $\bar\gob$-singular vectors remain $\bar\gob$-singular 
independent of the strong Condition B,
however they may fail to generate the ``top level'' 
$\LieGtwo$-module $Q_{\lambda(\mathbf x, \mathbf z)}$ defined by \eqref{eqTopLevel}.
\offlineVersionOnly{
Further investigation is needed to determine necessary and sufficient conditions
when the formulas produced by our method generate the entire ``top-level''
$\LieGtwo$-module $Q_{\lambda(\mathbf x, \mathbf z)}$ 
(see \eqref{eqTopLevel}), and we postpone to a future work. 

It appears that to achieve a complete understanding of the non-finite branching problem, one
should combine the methods used in the present paper with the study 
of the iterated translation functor obtained by tensoring by $ V_{\psi_1}(\LieGtwo)$,
decomposing into indecomposables, discarding all indecomposables of highest weight higher than the initial,
and iterating the procedure with all remaining indecomposable pieces of 
highest weight strictly lower than the initial. 
The application of the iterated translation functor to $M_{\mu}(\LieGtwo,\gop)$ remains to be investigated.
} %offlineVersionOnly,

\onlineVersionOnly{
By direct observation, for any parabolic subalgebra of $so(7)$
and $\lambda(\mathbf x, \mathbf z)$  with $z_1+z_2+z_3\leq 2$,
the set of weights $ \lambda(\mathbf x,\mathbf z)$ satisfying 
the strong Condition B (Definition \ref{defSufficientlyGenericWeight}) is 
non-empty. Therefore, if we change the base field to $\mathbb C(x_1,x_2,x_3)$, 
the strong Condition B never fails, as in the field $\mathbb C(x_1,x_2,x_3)$ a non-trivial 
$\mathbb Q$-linear combination of the $x_i$
is non-zero by definition. 
Therefore for $z_1+z_2+z_3\leq 2$, Theorem \ref{thSufficientlyGenericWeightEnablesFDbranching} applies over 
$\mathbb C(x_1,x_2,x_3)$, and specializations of the variable $x_i$ need to only exclude a proper 
Zariski-open set for the  $x_i$'s.
We present the computation of the $\bar\gob$-singular vectors 
in $M_{\lambda(\mathbf x,\mathbf z)}(so(7), \gop)$, $z_1+z_2+z_3\leq 2$ over the field
$\mathbb C(x_1,x_2,x_3)$ in Tables 
\ref{tableB3fdsOverG2charsAndHWV(0, 0, 1)}-\ref{tableB3fdsOverG2charsAndHWV(0, 1, 1)} 
(for $\gop\simeq \gop_{(1,0,0)}$ the result is already given in Theorem \ref{leBranchingExplicit(x_1,0,0),(1,a,b)}). 
The computations use Sections \ref{secWeylCharacterFormulas} and \ref{secConstructingSingularVectors} 
in a similar fashion to Theorem \ref{leBranchingExplicit(x_1,0,0),(1,a,b)} and
we do not present details. 
We note that by direct observation, the $\bar\gob$-singular vectors in the tables 
remain linearly independent under 
substitutions of the variables $x_1,x_2, x_3$ with constants. 

The linear $\neq$-inequalities that determine 
the strong Condition B (over the field $\mathbb C$) are computed explicitly 
in Tables \ref{tableB3fdsOverG2charsonly(1, 0, 0)}- \ref{tableB3fdsOverG2charsonly(0, 1, 1)} 
(including the case $\gop\simeq \gop_{(1,0,0)}$). 
In addition we compute the 
numbers $n(\mu,\lambda (\mathbf x, \mathbf z))$ for $z_1+z_2+z_3\leq 5$ 
with coefficients in $\mathbb C(x_1, x_2, x_3)$,  
as described in Section \ref{secConstructingSingularVectors}. 

We supplement the infinite dimensional branching by Table \ref{tableB3fdsOverG2charsAndHWV},
where we list the $\bar\gob$-singular vectors 
that decompose the finite dimensional $so(7)$-modules 
$z_1\omega_1+z_2\omega_2+z_3\omega_3$ with $z_1+z_2+z_3\leq 2$ 
over $\LieGtwo$. The numbers $n(\lambda,\mu)$ for the finite dimensional branching 
can be computed using \cite{McGovernG2InB3Branching} and we omit the corresponding table.

\subsubsection{Tables of $\gob$-singular vectors induced from $V_\lambda(\gol) $} \label{secAppendixSingularVectors}
In the first and second columns we write the module $V_\lambda(\gol)$, 
abbreviated as $V_\lambda$, and its dimension. 
The weight $\lambda\in \goh^*$ is taken with coefficients in 
$\mathbb C(x_1,x_2, x_3)$. In the 
third and fourth columns we list the $\bar\gol$-summands from the decomposition 
of $V_{\lambda}(\gol)$ over $\bar\gol$ and their dimensions. We abbreviate $V_\mu(\bar{\gol})$ as $V_\mu$.
In the fifth column we give an $\bar\gob\cap \bar{\gol}$-singular vector corresponding 
to each summand $V_{\lambda}(\gol)$. 
In the sixth column we give the Casimir projector whose existence 
is given by Theorem \ref{thSufficientlyGenericWeightEnablesFDbranching}, 
and in the last column we give the corresponding $\bar{\gob}$-singular vector 
as given by Theorem \ref{thSufficientlyGenericWeightEnablesFDbranching}. 
We note that for the case of the full parabolic subalgebra $\gop\simeq so(7)$ 
the last two columns are not applicable. We also note that the vectors listed 
in the last column remain $\bar\gob$-singular independent of the strong Condition B.

\begin{longtable}{|ccccl|} \caption{\label{tableB3fdsOverG2charsAndHWV} Decompositions of finite dimensional $so(7)$-modules over $\LieGtwo$}\\\hline so(7)& dim. &$\LieGtwo$&dim.& $\bar{\gob}$-singular vectors\\ \hline\endhead 
\multirow{1}{*}{$V_{0}$}& \multirow{1}{*}{$1$} & \multirow{1}{*}{$V_{0}$}&\multirow{1}{*}{$1$}&$\begin{array}{rcl}&&v_\lambda\end{array}$ \\ 
\hline\multirow{2}{*}{$V_{\omega_{3}}$}& \multirow{2}{*}{$8$} & \multirow{1}{*}{$V_{0}$}&\multirow{1}{*}{$1$}&$\begin{array}{rcl}&&g_{-1}g_{-2}g_{-3}\cdot v_\lambda \\&& -g_{-3}g_{-2}g_{-3}\cdot v_\lambda\end{array}$ \\ 
\cline{3-5}& & \multirow{1}{*}{$V_{\psi_{1}}$}&\multirow{1}{*}{$7$}&$\begin{array}{rcl}&&v_\lambda\end{array}$ \\ 
\hline\multirow{2}{*}{$V_{\omega_{2}}$}& \multirow{2}{*}{$21$} & \multirow{1}{*}{$V_{\psi_{1}}$}&\multirow{1}{*}{$7$}&$\begin{array}{rcl}&&g_{-1}g_{-2}\cdot v_\lambda \\&& -1/2g_{-3}g_{-2}\cdot v_\lambda\end{array}$ \\ 
\cline{3-5}& & \multirow{1}{*}{$V_{\psi_{2}}$}&\multirow{1}{*}{$14$}&$\begin{array}{rcl}&&v_\lambda\end{array}$ \\ 
\hline\multirow{1}{*}{$V_{\omega_{1}}$}& \multirow{1}{*}{$7$} & \multirow{1}{*}{$V_{\psi_{1}}$}&\multirow{1}{*}{$7$}&$\begin{array}{rcl}&&v_\lambda\end{array}$ \\ 
\hline\multirow{3}{*}{$V_{2\omega_{3}}$}& \multirow{3}{*}{$35$} & \multirow{1}{*}{$V_{0}$}&\multirow{1}{*}{$1$}&$\begin{array}{rcl}&&g_{-1}^{2}g_{-2}^{2}g_{-3}^{2}\cdot v_\lambda \\&& -g_{-3}g_{-1}g_{-2}^{2}g_{-3}^{2}\cdot v_\lambda \\&& +2g_{-2}g_{-3}^{2}g_{-1}g_{-2}g_{-3}\cdot v_\lambda \\&& -2g_{-3}g_{-2}g_{-3}g_{-1}g_{-2}g_{-3}\cdot v_\lambda \\&& +g_{-3}^{2}g_{-2}^{2}g_{-3}^{2}\cdot v_\lambda\end{array}$ \\ 
\cline{3-5}& & \multirow{1}{*}{$V_{\psi_{1}}$}&\multirow{1}{*}{$7$}&$\begin{array}{rcl}&&g_{-1}g_{-2}g_{-3}\cdot v_\lambda \\&& +1/2g_{-2}g_{-3}^{2}\cdot v_\lambda \\&& -g_{-3}g_{-2}g_{-3}\cdot v_\lambda\end{array}$ \\
\hline\hline\end{longtable}

\begin{landscape}
\begin{longtable}{|cccclll|} \caption{\label{tableB3fdsOverG2charsAndHWV(0, 1, 0)} $\bar{\gob}$-singular vectors in $M_\lambda(so(7),\gop_{(0,1,0)})$, corresponding to $V_\lambda(\gol)$, where $\gol$ is the reductive Levi part of $\gop_{(0, 1, 0)}$. $\bar\gol$ denotes the reductive Levi part of $\bar\gop_{(0,1)} \subset \LieGtwo$.}\\ \hline  $V_\lambda(\gol)$  $\phantom{ V^{x}}$ & dim. &$\bar{\gol}$-decomp. &dim.& $\mathfrak b \cap \gol$-singular vectors & Casimir projector & Corresp. $\bar{\gob}$-singular vectors \\ \hline\endhead
\multirow{1}{*}{$V_{x_{2}\omega_{2}}$}& \multirow{1}{*}{$1$} & \multirow{1}{*}{$V_{x_{2}\psi_{2}}$}&\multirow{1}{*}{$1$}&$\begin{array}{l}-v_\lambda\end{array}$ & $\begin{array}{l}id\end{array}$&$\begin{array}{l}-v_\lambda\end{array}$\\ 
\hline\multirow{1}{*}{$V_{x_{2}\omega_{2}+\omega_{3}}$}& \multirow{1}{*}{$2$} & \multirow{1}{*}{$V_{\psi_{1}+x_{2}\psi_{2}}$}&\multirow{1}{*}{$2$}&$\begin{array}{l}-v_\lambda\end{array}$ & $\begin{array}{l}id\end{array}$&$\begin{array}{l}-v_\lambda\end{array}$\\ 
\hline\multirow{1}{*}{$V_{\omega_{1}+x_{2}\omega_{2}}$}& \multirow{1}{*}{$2$} & \multirow{1}{*}{$V_{\psi_{1}+x_{2}\psi_{2}}$}&\multirow{1}{*}{$2$}&$\begin{array}{l}-v_\lambda\end{array}$ & $\begin{array}{l}id\end{array}$&$\begin{array}{l}-v_\lambda\end{array}$\\ 
\hline\multirow{1}{*}{$V_{x_{2}\omega_{2}+2\omega_{3}}$}& \multirow{1}{*}{$3$} & \multirow{1}{*}{$V_{2\psi_{1}+x_{2}\psi_{2}}$}&\multirow{1}{*}{$3$}&$\begin{array}{l}-v_\lambda\end{array}$ & $\begin{array}{l}id\end{array}$&$\begin{array}{l}-v_\lambda\end{array}$\\ 
\hline\multirow{2}{*}{$V_{\omega_{1}+x_{2}\omega_{2}+\omega_{3}}$}& \multirow{2}{*}{$4$} & \multirow{1}{*}{$V_{(x_{2}+1)\psi_{2}}$}&\multirow{1}{*}{$1$}&$\begin{array}{l}-g_{-1}\cdot v_\lambda \\ +g_{-3}\cdot v_\lambda\end{array}$ & $\begin{array}{l}12(i(\bar c_1) - (1/4x_{2}^{2}+5/4x_{2} \\ +7/6))\end{array}$&$\begin{array}{l}2g_{-1}\cdot v_\lambda \\ -2g_{-3}\cdot v_\lambda\end{array}$\\ 
\cline{3-7}& & \multirow{1}{*}{$V_{2\psi_{1}+x_{2}\psi_{2}}$}&\multirow{1}{*}{$3$}&$\begin{array}{l}-v_\lambda\end{array}$ & $\begin{array}{l}id\end{array}$&$\begin{array}{l}-v_\lambda\end{array}$\\ 
\hline\multirow{1}{*}{$V_{2\omega_{1}+x_{2}\omega_{2}}$}& \multirow{1}{*}{$3$} & \multirow{1}{*}{$V_{2\psi_{1}+x_{2}\psi_{2}}$}&\multirow{1}{*}{$3$}&$\begin{array}{l}-v_\lambda\end{array}$ & $\begin{array}{l}id\end{array}$&$\begin{array}{l}-v_\lambda\end{array}$\\
\hline\hline\end{longtable}

\begin{longtable}{|cccclll|} \caption{\label{tableB3fdsOverG2charsAndHWV(0, 0, 1)} $\bar{\gob}$-singular vectors in $M_\lambda(so(7),\gop_{(0,0,1)})$, corresponding to $V_\lambda(\gol)$, where $\gol$ is the reductive Levi part of $\gop_{(0, 0, 1)}$. $\bar\gol$ denotes the reductive Levi part of $\bar\gop_{(1,0)} \subset \LieGtwo$. }\\\hline $V_\lambda(\gol)$  $\phantom{ V^{x}}$ & dim. &$\bar{\gol}$-decomp. &dim.& $\mathfrak b \cap \gol$-singular vectors & Casimir projector & Corresp. $\bar{\gob}$-singular vectors \\\hline\endhead 
\multirow{1}{*}{$V_{x_{3}\omega_{3}}$}& \multirow{1}{*}{$1$} & \multirow{1}{*}{$V_{x_{3}\psi_{1}}$}&\multirow{1}{*}{$1$}&$\begin{array}{l}-v_\lambda\end{array}$ & $\begin{array}{l}id\end{array}$&$\begin{array}{l}-v_\lambda\end{array}$\\ 
\hline\multirow{2}{*}{$V_{\omega_{2}+x_{3}\omega_{3}}$}& \multirow{2}{*}{$3$} & \multirow{1}{*}{$V_{(x_{3}+1)\psi_{1}}$}&\multirow{1}{*}{$1$}&$\begin{array}{l}-g_{-1}g_{-2}\cdot v_\lambda\end{array}$ & $\begin{array}{l}12(i(\bar c_1) - (1/12x_{3}^{2}+2/3x_{3} \\ +1))\end{array}$&$\begin{array}{l}-g_{-5}v_\lambda-g_{-3}g_{-2}\cdot v_\lambda \\ +(x_{3}+4)g_{-1}g_{-2}\cdot v_\lambda\end{array}$\\ 
\cline{3-7}& & \multirow{1}{*}{$V_{x_{3}\psi_{1}+\psi_{2}}$}&\multirow{1}{*}{$2$}&$\begin{array}{l}-v_\lambda\end{array}$ & $\begin{array}{l}id\end{array}$&$\begin{array}{l}-v_\lambda\end{array}$\\ 
\hline\multirow{2}{*}{$V_{\omega_{1}+x_{3}\omega_{3}}$}& \multirow{2}{*}{$3$} & \multirow{1}{*}{$V_{(x_{3}-1)\psi_{1}+\psi_{2}}$}&\multirow{1}{*}{$2$}&$\begin{array}{l}-g_{-1}\cdot v_\lambda\end{array}$ & $\begin{array}{l}12(i(\bar c_1) - (1/12x_{3}^{2}+7/12x_{3} \\ +1/2))\end{array}$&$\begin{array}{l}-g_{-3}v_\lambda+x_{3}g_{-1}\cdot v_\lambda\end{array}$\\ 
\cline{3-7}& & \multirow{1}{*}{$V_{(x_{3}+1)\psi_{1}}$}&\multirow{1}{*}{$1$}&$\begin{array}{l}-v_\lambda\end{array}$ & $\begin{array}{l}id\end{array}$&$\begin{array}{l}-v_\lambda\end{array}$\\
\hline\hline\end{longtable}

\begin{longtable}{|cccclll|} \caption{\label{tableB3fdsOverG2charsAndHWV(1, 1, 0)} $\bar{\gob}$-singular vectors in $M_\lambda(so(7),\gop_{(1,1,0)})$, corresponding to $V_\lambda(\gol)$, where $\gol$ is the reductive Levi part of $\gop_{(1, 1, 0)}$. $\bar\gol$ denotes the reductive Levi part of $\bar\gop_{(1,1)} \subset \LieGtwo$.}\\\hline$V_\lambda(\gol)$  $\phantom{ V^{x}}$ & dim. &$\bar{\gol}$-decomp. &dim.& $\mathfrak b \cap \gol$-singular vectors & Casimir projector & Corresp. $\bar{\gob}$-singular vectors \\ \hline\endhead 
\multirow{1}{*}{$V_{x_{1}\omega_{1}+x_{2}\omega_{2}}$}& \multirow{1}{*}{$1$} & \multirow{1}{*}{$V_{x_{1}\psi_{1}+x_{2}\psi_{2}}$}&\multirow{1}{*}{$1$}&$\begin{array}{l}-v_\lambda\end{array}$ & $\begin{array}{l}id\end{array}$&$\begin{array}{l}-v_\lambda\end{array}$\\ 
\hline\multirow{2}{*}{$V_{x_{1}\omega_{1}+x_{2}\omega_{2}+\omega_{3}}$}& \multirow{2}{*}{$2$} & \multirow{1}{*}{$V_{(x_{1}-1)\psi_{1}+(x_{2}+1)\psi_{2}}$}&\multirow{1}{*}{$1$}&$\begin{array}{l}-g_{-3}\cdot v_\lambda\end{array}$ & $\begin{array}{l}12(i(\bar c_1) - (1/4x_{2}^{2}+1/4x_{1}x_{2} \\ +x_{2}+1/12x_{1}^{2}+7/12x_{1} \\ +1/2))\end{array}$&$\begin{array}{l}-g_{-1}v_\lambda+x_{1}g_{-3}\cdot v_\lambda\end{array}$\\ 
\cline{3-7}& & \multirow{1}{*}{$V_{(x_{1}+1)\psi_{1}+x_{2}\psi_{2}}$}&\multirow{1}{*}{$1$}&$\begin{array}{l}-v_\lambda\end{array}$ & $\begin{array}{l}id\end{array}$&$\begin{array}{l}-v_\lambda\end{array}$\\ 
\hline\multirow{3}{*}{$V_{x_{1}\omega_{1}+x_{2}\omega_{2}+2\omega_{3}}$}& \multirow{3}{*}{$3$} & \multirow{1}{*}{$V_{(x_{1}-2)\psi_{1}+(x_{2}+2)\psi_{2}}$}&\multirow{1}{*}{$1$}&$\begin{array}{l}-g_{-3}^{2}\cdot v_\lambda\end{array}$ & $\begin{array}{l}12(i(\bar c_1) - (1/4x_{2}^{2}+1/4x_{1}x_{2} \\ +5/4x_{2}+1/12x_{1}^{2} \\ +2/3x_{1}+1))\\12(i(\bar c_1) - (1/4x_{2}^{2}+1/4x_{1}x_{2} \\ +5/4x_{2}+1/12x_{1}^{2} \\ +3/4x_{1}+7/6))\end{array}$&$\begin{array}{l}-4g_{-1}^{2}v_\lambda \\ +(4x_{1}-4)g_{-1}g_{-3}\cdot v_\lambda \\ +(-2x_{1}^{2}+2x_{1})g_{-3}^{2}\cdot v_\lambda\end{array}$\\ 
\cline{3-7}& & \multirow{1}{*}{$V_{x_{1}\psi_{1}+(x_{2}+1)\psi_{2}}$}&\multirow{1}{*}{$1$}&$\begin{array}{l}-g_{-3}\cdot v_\lambda\end{array}$ & $\begin{array}{l}12(i(\bar c_1) - (1/4x_{2}^{2}+1/4x_{1}x_{2} \\ +5/4x_{2}+1/12x_{1}^{2} \\ +3/4x_{1}+7/6))\end{array}$&$\begin{array}{l}-2g_{-1}v_\lambda+x_{1}g_{-3}\cdot v_\lambda\end{array}$\\
\hline\hline\end{longtable}

\begin{longtable}{|cccclll|} \caption{\label{tableB3fdsOverG2charsAndHWV(1, 0, 1)} $\bar{\gob}$-singular vectors in $M_\lambda(so(7),\gop_{(1,0,1)})$, corresponding to $V_\lambda(\gol)$, where $\gol$ is the reductive Levi part of $\gop_{(1, 0, 1)}$. $\bar\gol$ denotes the reductive Levi part of $\bar\gop_{(0,1)} \subset \LieGtwo$.}\\\hline$V_\lambda(\gol)$  $\phantom{ V^{x}}$ & dim. &$\bar{\gol}$-decomp. &dim.& $\mathfrak b \cap \gol$-singular vectors & Casimir projector & Corresp. $\bar{\gob}$-singular vectors \\ \hline\endhead 
\multirow{1}{*}{$V_{x_{1}\omega_{1}+x_{3}\omega_{3}}$}& \multirow{1}{*}{$1$} & \multirow{1}{*}{$V_{(x_{3}+x_{1})\psi_{1}}$}&\multirow{1}{*}{$1$}&$\begin{array}{l}-v_\lambda\end{array}$ & $\begin{array}{l}id\end{array}$&$\begin{array}{l}-v_\lambda\end{array}$\\ 
\hline\multirow{1}{*}{$V_{x_{1}\omega_{1}+\omega_{2}+x_{3}\omega_{3}}$}& \multirow{1}{*}{$2$} & \multirow{1}{*}{$V_{(x_{3}+x_{1})\psi_{1}+\psi_{2}}$}&\multirow{1}{*}{$2$}&$\begin{array}{l}-v_\lambda\end{array}$ & $\begin{array}{l}id\end{array}$&$\begin{array}{l}-v_\lambda\end{array}$\\ 
\hline\multirow{1}{*}{$V_{x_{1}\omega_{1}+2\omega_{2}+x_{3}\omega_{3}}$}& \multirow{1}{*}{$3$} & \multirow{1}{*}{$V_{(x_{3}+x_{1})\psi_{1}+2\psi_{2}}$}&\multirow{1}{*}{$3$}&$\begin{array}{l}-v_\lambda\end{array}$ & $\begin{array}{l}id\end{array}$&$\begin{array}{l}-v_\lambda\end{array}$\\ 
\hline\hline\end{longtable}

\begin{longtable}{|cccclll|} \caption{\label{tableB3fdsOverG2charsAndHWV(0, 1, 1)} $\bar{\gob}$-singular vectors in $M_\lambda(so(7),\gop_{(0,1,1)})$, corresponding to $V_\lambda(\gol)$, where $\gol$ is the reductive Levi part of $\gop_{(0, 1, 1)}$. $\bar\gol$ denotes the reductive Levi part of $\bar\gop_{(1,1)} \subset \LieGtwo$.}\\\hline $V_\lambda(\gol)$  $\phantom{ V^{x}}$ & dim. &$\bar{\gol}$-decomp. &dim.& $\mathfrak b \cap \gol$-singular vectors & Casimir projector & Corresp. $\bar{\gob}$-singular vectors \\ \hline\endhead 
\multirow{1}{*}{$V_{x_{2}\omega_{2}+x_{3}\omega_{3}}$}& \multirow{1}{*}{$1$} & \multirow{1}{*}{$V_{x_{3}\psi_{1}+x_{2}\psi_{2}}$}&\multirow{1}{*}{$1$}&$\begin{array}{l}-v_\lambda\end{array}$ & $\begin{array}{l}id\end{array}$&$\begin{array}{l}-v_\lambda\end{array}$\\ 
\hline\multirow{2}{*}{$V_{\omega_{1}+x_{2}\omega_{2}+x_{3}\omega_{3}}$}& \multirow{2}{*}{$2$} & \multirow{1}{*}{$V_{(x_{3}-1)\psi_{1}+(x_{2}+1)\psi_{2}}$}&\multirow{1}{*}{$1$}&$\begin{array}{l}-g_{-1}\cdot v_\lambda\end{array}$ & $\begin{array}{l}12(i(\bar c_1) - (1/12x_{3}^{2}+1/4x_{2}x_{3} \\ +7/12x_{3}+1/4x_{2}^{2} \\ +x_{2}+1/2))\end{array}$&$\begin{array}{l}-g_{-3}v_\lambda+x_{3}g_{-1}\cdot v_\lambda\end{array}$\\ 
\cline{3-7}& & \multirow{1}{*}{$V_{(x_{3}+1)\psi_{1}+x_{2}\psi_{2}}$}&\multirow{1}{*}{$1$}&$\begin{array}{l}-v_\lambda\end{array}$ & $\begin{array}{l}id\end{array}$&$\begin{array}{l}-v_\lambda\end{array}$\\ 
\hline\multirow{3}{*}{$V_{2\omega_{1}+x_{2}\omega_{2}+x_{3}\omega_{3}}$}& \multirow{3}{*}{$3$} & \multirow{1}{*}{$V_{(x_{3}-2)\psi_{1}+(x_{2}+2)\psi_{2}}$}&\multirow{1}{*}{$1$}&$\begin{array}{l}-g_{-1}^{2}\cdot v_\lambda\end{array}$ & $\begin{array}{l}12(i(\bar c_1) - (1/12x_{3}^{2}+1/4x_{2}x_{3} \\ +2/3x_{3}+1/4x_{2}^{2} \\ +5/4x_{2}+1))\\12(i(\bar c_1) - (1/12x_{3}^{2}+1/4x_{2}x_{3} \\ +3/4x_{3}+1/4x_{2}^{2} \\ +5/4x_{2}+7/6))\end{array}$&$\begin{array}{l}-4g_{-3}^{2}v_\lambda \\ +(4x_{3}-4)g_{-3}g_{-1}\cdot v_\lambda \\ +(-2x_{3}^{2}+2x_{3})g_{-1}^{2}\cdot v_\lambda\end{array}$\\ 
\cline{3-7}& & \multirow{1}{*}{$V_{x_{3}\psi_{1}+(x_{2}+1)\psi_{2}}$}&\multirow{1}{*}{$1$}&$\begin{array}{l}-g_{-1}\cdot v_\lambda\end{array}$ & $\begin{array}{l}12(i(\bar c_1) - (1/12x_{3}^{2}+1/4x_{2}x_{3} \\ +3/4x_{3}+1/4x_{2}^{2} \\ +5/4x_{2}+7/6))\end{array}$&$\begin{array}{l}-2g_{-3}v_\lambda+x_{3}g_{-1}\cdot v_\lambda\end{array}$\\
\hline\hline\end{longtable}

\end{landscape}
\subsubsection{Tables of $n(\mu,\lambda)$}
In the left column we write the module $V_\lambda(\gol)$, abbreviated as $V_\lambda$. 
The weight $\lambda\in \goh^*$ is taken with coefficients in 
$\mathbb C(x_1,x_2, x_3)$. In the right column we write the decomposition of
$V_{\lambda}(\gol)$ as a direct sum of $V_\mu(\bar{\gol})$-modules, $\mu\in \bar\goh^*$. 
For $\mu\in \bar{\goh}^*$ we abbreviate $V_\mu(\bar{\gol})$ by $V_\mu$.
In the row below each pair of columns, we indicate, for each $\lambda\in \goh^*$ with $\mathbb C(x_1,x_2, x_3)$, the
Zariski open conditions on the $x_i$'s which imply the strong Condition B. 
We recall that $\psi_1,\psi_2$ stand for the fundamental weights of $\LieGtwo$ and 
$\omega_1,\omega_2,\omega_3$ stand for the fundamental weights of $so(7)$.
\begin{longtable}{|p{2cm}l|} \caption{\label{tableB3fdsOverG2charsonly(1, 0, 0)} Decompositions of inducing $\mathfrak{p}_{(1, 0, 0)}$-modules over $\bar{ \mathfrak {l}}$, where $\mathfrak{l}$ is the reductive Levi part of $\mathfrak{p}_{(1, 0, 0)}$ and $\bar{\mathfrak {l}}$ is the reductive Levi part of $\bar{\mathfrak {p}} _{(1, 0)}$ .} \\\hline $V_\lambda(\mathfrak l)$ & Decomposition over $\bar { \mathfrak l}$ \endhead \hline 
$ V_{x_{1}\omega_{1}} $ & $\begin{array}{l} V_{x_{1}\psi_{1}}\end{array} $ \\ 
\hline $ V_{x_{1}\omega_{1}+\omega_{3}} $ & $\begin{array}{l} V_{(x_{1}+1)\psi_{1}}\oplus V_{(x_{1}-1)\psi_{1}+\psi_{2}}\oplus V_{x_{1}\psi_{1}}\end{array} $ \\ 
\multicolumn{2}{|p{11.6cm}|}{ \tiny Strong Condition B: $(-x_{1}-5)\neq 0$, $(-x_{1}-3)\neq 0$, $(-x_{1}-1)\neq 0$\normalsize}\\ 
\hline $ V_{x_{1}\omega_{1}+\omega_{2}} $ & $\begin{array}{l} V_{x_{1}\psi_{1}+\psi_{2}}\oplus V_{(x_{1}+1)\psi_{1}}\oplus V_{(x_{1}-1)\psi_{1}+\psi_{2}}\end{array} $ \\ 
\multicolumn{2}{|p{11.6cm}|}{ \tiny Strong Condition B: $(-x_{1}-6)\neq 0$, $(-x_{1}-1)\neq 0$, $(-2x_{1}-7)\neq 0$\normalsize}\\ 
\hline $ V_{x_{1}\omega_{1}+2\omega_{3}} $ & $\begin{array}{l} V_{x_{1}\psi_{1}+\psi_{2}}\oplus V_{(x_{1}+2)\psi_{1}}\oplus V_{(x_{1}+1)\psi_{1}}\oplus V_{(x_{1}-2)\psi_{1}+2\psi_{2}} \\ \oplus V_{(x_{1}-1)\psi_{1}+\psi_{2}}\oplus V_{x_{1}\psi_{1}}\end{array} $ \\ 
\multicolumn{2}{|p{11.6cm}|}{ \tiny Strong Condition B: $-x_{1}\neq 0$, $(-x_{1}-1)\neq 0$, $(-x_{1}-7)\neq 0$, $(-2x_{1}-7)\neq 0$, $(-x_{1}-3)\neq 0$, $(-x_{1}-2)\neq 0$, $(-x_{1}-6)\neq 0$, $(-x_{1}-5)\neq 0$, $(-x_{1}-4)\neq 0$, $-1\neq 0$\normalsize}\\
\hline $ V_{x_{1}\omega_{1}+\omega_{2}+\omega_{3}} $ & $\begin{array}{l} V_{(x_{1}+1)\psi_{1}+\psi_{2}}\oplus 2V_{x_{1}\psi_{1}+\psi_{2}}\oplus V_{(x_{1}+1)\psi_{1}}\oplus V_{(x_{1}-1)\psi_{1}+2\psi_{2}} \\ \oplus V_{(x_{1}+2)\psi_{1}}\oplus V_{(x_{1}-2)\psi_{1}+2\psi_{2}}\oplus V_{(x_{1}-1)\psi_{1}+\psi_{2}}\end{array} $ \\
\hline\end{longtable}

\begin{longtable}{|p{2cm}l|} \caption{\label{tableB3fdsOverG2charsonly(0, 1, 0)} Decompositions of inducing $\mathfrak{p}_{(0, 1, 0)}$-modules over $\bar{ \mathfrak {l}}$, where $\mathfrak{l}$ is the reductive Levi part of $\mathfrak{p}_{(0, 1, 0)}$ and $\bar{\mathfrak {l}}$ is the reductive Levi part of $\bar{\mathfrak {p}} _{(0, 1)}$ .} \\\hline $V_\lambda(\mathfrak l)$ & Decomposition over $\bar { \mathfrak l}$ \endhead \hline 
$ V_{x_{2}\omega_{2}} $ & $\begin{array}{l} V_{x_{2}\psi_{2}}\end{array} $ \\ 
\hline $ V_{x_{2}\omega_{2}+\omega_{3}} $ & $\begin{array}{l} V_{\psi_{1}+x_{2}\psi_{2}}\end{array} $ \\ 
\hline $ V_{\omega_{1}+x_{2}\omega_{2}} $ & $\begin{array}{l} V_{\psi_{1}+x_{2}\psi_{2}}\end{array} $ \\ 
\hline $ V_{x_{2}\omega_{2}+2\omega_{3}} $ & $\begin{array}{l} V_{2\psi_{1}+x_{2}\psi_{2}}\end{array} $ \\ 
\hline $ V_{\omega_{1}+x_{2}\omega_{2}+\omega_{3}} $ & $\begin{array}{l} V_{(x_{2}+1)\psi_{2}}\oplus V_{2\psi_{1}+x_{2}\psi_{2}}\end{array} $ \\ 
\multicolumn{2}{|p{11.6cm}|}{ \tiny Strong Condition B: $-1\neq 0$\normalsize}\\ 
\hline $ V_{2\omega_{1}+x_{2}\omega_{2}} $ & $\begin{array}{l} V_{2\psi_{1}+x_{2}\psi_{2}}\end{array} $ \\ 
\hline $ V_{x_{2}\omega_{2}+3\omega_{3}} $ & $\begin{array}{l} V_{3\psi_{1}+x_{2}\psi_{2}}\end{array} $ \\
\hline\end{longtable}
\begin{longtable}{|p{2cm}l|} \caption{\label{tableB3fdsOverG2charsonly(0, 0, 1)} Decompositions of inducing $\mathfrak{p}_{(0, 0, 1)}$-modules over $\bar{ \mathfrak {l}}$, where $\mathfrak{l}$ is the reductive Levi part of $\mathfrak{p}_{(0, 0, 1)}$ and $\bar{\mathfrak {l}}$ is the reductive Levi part of $\bar{\mathfrak {p}} _{(1, 0)}$ .} \\\hline $V_\lambda(\mathfrak l)$ & Decomposition over $\bar { \mathfrak l}$ \endhead \hline 
$ V_{x_{3}\omega_{3}} $ & $\begin{array}{l} V_{x_{3}\psi_{1}}\end{array} $ \\ 
\hline $ V_{\omega_{2}+x_{3}\omega_{3}} $ & $\begin{array}{l} V_{x_{3}\psi_{1}+\psi_{2}}\oplus V_{(x_{3}+1)\psi_{1}}\end{array} $ \\ 
\multicolumn{2}{|p{11.6cm}|}{ \tiny Strong Condition B: $(-x_{3}-6)\neq 0$\normalsize}\\ 
\hline $ V_{\omega_{1}+x_{3}\omega_{3}} $ & $\begin{array}{l} V_{(x_{3}+1)\psi_{1}}\oplus V_{(x_{3}-1)\psi_{1}+\psi_{2}}\end{array} $ \\ 
\multicolumn{2}{|p{11.6cm}|}{ \tiny Strong Condition B: $(-x_{3}-1)\neq 0$\normalsize}\\ 
\hline $ V_{2\omega_{2}+x_{3}\omega_{3}} $ & $\begin{array}{l} V_{x_{3}\psi_{1}+2\psi_{2}}\oplus V_{(x_{3}+1)\psi_{1}+\psi_{2}}\oplus V_{(x_{3}+2)\psi_{1}}\end{array} $ \\ 
\multicolumn{2}{|p{11.6cm}|}{ \tiny Strong Condition B: $(-x_{3}-9)\neq 0$, $(-x_{3}-7)\neq 0$, $(-x_{3}-8)\neq 0$\normalsize}\\ 
\hline $ V_{\omega_{1}+\omega_{2}+x_{3}\omega_{3}} $ & $\begin{array}{l} V_{x_{3}\psi_{1}+\psi_{2}}\oplus V_{(x_{3}-1)\psi_{1}+2\psi_{2}}\oplus V_{(x_{3}+1)\psi_{1}+\psi_{2}} \\ \oplus V_{(x_{3}+2)\psi_{1}}\end{array} $ \\
\hline\end{longtable}

\begin{longtable}{|p{2cm}l|} \caption{\label{tableB3fdsOverG2charsonly(1, 1, 0)} Decompositions of inducing $\mathfrak{p}_{(1, 1, 0)}$-modules over $\bar{ \mathfrak {l}}$, where $\mathfrak{l}$ is the reductive Levi part of $\mathfrak{p}_{(1, 1, 0)}$ and $\bar{\mathfrak {l}}$ is the reductive Levi part of $\bar{\mathfrak {p}} _{(1, 1)}$ .} \\\hline $V_\lambda(\mathfrak l)$ & Decomposition over $\bar { \mathfrak l}$ \endhead \hline 
$ V_{x_{1}\omega_{1}+x_{2}\omega_{2}} $ & $\begin{array}{l} V_{x_{1}\psi_{1}+x_{2}\psi_{2}}\end{array} $ \\ 
\hline $ V_{x_{1}\omega_{1}+x_{2}\omega_{2}+\omega_{3}} $ & $\begin{array}{l} V_{(x_{1}+1)\psi_{1}+x_{2}\psi_{2}}\oplus V_{(x_{1}-1)\psi_{1}+(x_{2}+1)\psi_{2}}\end{array} $ \\ 
\multicolumn{2}{|p{11.6cm}|}{ \tiny Strong Condition B: $(-x_{1}-1)\neq 0$\normalsize}\\ 
\hline $ V_{x_{1}\omega_{1}+x_{2}\omega_{2}+2\omega_{3}} $ & $\begin{array}{l} V_{x_{1}\psi_{1}+(x_{2}+1)\psi_{2}}\oplus V_{(x_{1}+2)\psi_{1}+x_{2}\psi_{2}}\oplus V_{(x_{1}-2)\psi_{1}+(x_{2}+2)\psi_{2}}\end{array} $ \\ 
\multicolumn{2}{|p{11.6cm}|}{ \tiny Strong Condition B: $(-x_{1}-2)\neq 0$, $-x_{1}\neq 0$, $(-x_{1}-1)\neq 0$\normalsize}\\ 
\hline $ V_{x_{1}\omega_{1}+x_{2}\omega_{2}+3\omega_{3}} $ & $\begin{array}{l} V_{(x_{1}-1)\psi_{1}+(x_{2}+2)\psi_{2}}\oplus V_{(x_{1}+1)\psi_{1}+(x_{2}+1)\psi_{2}}\oplus V_{(x_{1}+3)\psi_{1}+x_{2}\psi_{2}} \\ \oplus V_{(x_{1}-3)\psi_{1}+(x_{2}+3)\psi_{2}}\end{array} $ \\ 
\multicolumn{2}{|p{11.6cm}|}{ \tiny Strong Condition B: $(-x_{1}-1)\neq 0$, $(-x_{1}-2)\neq 0$, $(-x_{1}-3)\neq 0$, $(-x_{1}+1)\neq 0$, $-x_{1}\neq 0$\normalsize}\\ 
\hline $ V_{x_{1}\omega_{1}+x_{2}\omega_{2}+4\omega_{3}} $ & $\begin{array}{l} V_{(x_{1}+2)\psi_{1}+(x_{2}+1)\psi_{2}}\oplus V_{(x_{1}-2)\psi_{1}+(x_{2}+3)\psi_{2}}\oplus V_{x_{1}\psi_{1}+(x_{2}+2)\psi_{2}} \\ \oplus V_{(x_{1}+4)\psi_{1}+x_{2}\psi_{2}}\oplus V_{(x_{1}-4)\psi_{1}+(x_{2}+4)\psi_{2}}\end{array} $ \\
\hline\end{longtable}

\begin{longtable}{|p{2cm}l|} \caption{\label{tableB3fdsOverG2charsonly(1, 0, 1)} Decompositions of inducing $\mathfrak{p}_{(1, 0, 1)}$-modules over $\bar{ \mathfrak {l}}$, where $\mathfrak{l}$ is the reductive Levi part of $\mathfrak{p}_{(1, 0, 1)}$ and $\bar{\mathfrak {l}}$ is the reductive Levi part of $\bar{\mathfrak {p}} _{(1, 0)}$ .} \\\hline $V_\lambda(\mathfrak l)$ & Decomposition over $\bar { \mathfrak l}$ \endhead \hline 
$ V_{x_{1}\omega_{1}+x_{3}\omega_{3}} $ & $\begin{array}{l} V_{(x_{3}+x_{1})\psi_{1}}\end{array} $ \\ 
\hline $ V_{x_{1}\omega_{1}+\omega_{2}+x_{3}\omega_{3}} $ & $\begin{array}{l} V_{(x_{3}+x_{1})\psi_{1}+\psi_{2}}\end{array} $ \\ 
\hline $ V_{x_{1}\omega_{1}+2\omega_{2}+x_{3}\omega_{3}} $ & $\begin{array}{l} V_{(x_{3}+x_{1})\psi_{1}+2\psi_{2}}\end{array} $ \\ 
\hline $ V_{x_{1}\omega_{1}+3\omega_{2}+x_{3}\omega_{3}} $ & $\begin{array}{l} V_{(x_{3}+x_{1})\psi_{1}+3\psi_{2}}\end{array} $ \\ 
\hline $ V_{x_{1}\omega_{1}+4\omega_{2}+x_{3}\omega_{3}} $ & $\begin{array}{l} V_{(x_{3}+x_{1})\psi_{1}+4\psi_{2}}\end{array} $ \\ 
\hline $ V_{x_{1}\omega_{1}+5\omega_{2}+x_{3}\omega_{3}} $ & $\begin{array}{l} V_{(x_{3}+x_{1})\psi_{1}+5\psi_{2}}\end{array} $ \\ 
\hline\end{longtable}

\begin{longtable}{|p{2cm}l|} \caption{\label{tableB3fdsOverG2charsonly(0, 1, 1)} Decompositions of inducing $\mathfrak{p}_{(0, 1, 1)}$-modules over $\bar{ \mathfrak {l}}$, where $\mathfrak{l}$ is the reductive Levi part of $\mathfrak{p}_{(0, 1, 1)}$ and $\bar{\mathfrak {l}}$ is the reductive Levi part of $\bar{\mathfrak {p}} _{(1, 1)}$ .} \\\hline $V_\lambda(\mathfrak l)$ & Decomposition over $\bar { \mathfrak l}$ \endhead \hline 
$ V_{x_{2}\omega_{2}+x_{3}\omega_{3}} $ & $\begin{array}{l} V_{x_{3}\psi_{1}+x_{2}\psi_{2}}\end{array} $ \\ 
\hline $ V_{\omega_{1}+x_{2}\omega_{2}+x_{3}\omega_{3}} $ & $\begin{array}{l} V_{(x_{3}+1)\psi_{1}+x_{2}\psi_{2}}\oplus V_{(x_{3}-1)\psi_{1}+(x_{2}+1)\psi_{2}}\end{array} $ \\ 
\multicolumn{2}{|p{11.6cm}|}{ \tiny Strong Condition B: $(-x_{3}-1)\neq 0$\normalsize}\\ 
\hline $ V_{2\omega_{1}+x_{2}\omega_{2}+x_{3}\omega_{3}} $ & $\begin{array}{l} V_{x_{3}\psi_{1}+(x_{2}+1)\psi_{2}}\oplus V_{(x_{3}+2)\psi_{1}+x_{2}\psi_{2}}\oplus V_{(x_{3}-2)\psi_{1}+(x_{2}+2)\psi_{2}}\end{array} $ \\ 
\multicolumn{2}{|p{11.6cm}|}{ \tiny Strong Condition B: $(-x_{3}-2)\neq 0$, $-x_{3}\neq 0$, $(-x_{3}-1)\neq 0$\normalsize}\\ 
\hline $ V_{3\omega_{1}+x_{2}\omega_{2}+x_{3}\omega_{3}} $ & $\begin{array}{l} V_{(x_{3}-1)\psi_{1}+(x_{2}+2)\psi_{2}}\oplus V_{(x_{3}+1)\psi_{1}+(x_{2}+1)\psi_{2}}\oplus V_{(x_{3}+3)\psi_{1}+x_{2}\psi_{2}} \\ \oplus V_{(x_{3}-3)\psi_{1}+(x_{2}+3)\psi_{2}}\end{array} $ \\ 
\multicolumn{2}{|p{11.6cm}|}{ \tiny Strong Condition B: $(-x_{3}-1)\neq 0$, $(-x_{3}-2)\neq 0$, $(-x_{3}-3)\neq 0$, $(-x_{3}+1)\neq 0$, $-x_{3}\neq 0$\normalsize}\\ 
\hline $ V_{4\omega_{1}+x_{2}\omega_{2}+x_{3}\omega_{3}} $ & $\begin{array}{l} V_{x_{3}\psi_{1}+(x_{2}+2)\psi_{2}}\oplus V_{(x_{3}+2)\psi_{1}+(x_{2}+1)\psi_{2}}\oplus V_{(x_{3}-2)\psi_{1}+(x_{2}+3)\psi_{2}} \\ \oplus V_{(x_{3}+4)\psi_{1}+x_{2}\psi_{2}}\oplus V_{(x_{3}-4)\psi_{1}+(x_{2}+4)\psi_{2}}\end{array} $ \\
\hline\end{longtable}

} %onlineVersionOnly

%\end{appendix}

%%%%%%%%%%%%%%%%%%%%%%%%%%%%%%%%%%%%%%%%%%%%%%%%%%%%%%%%%%%%%%%%%%%%%%%%%%%%%%%
%%%%%%%%%%%%%%%%%%%%%%%%%%%%%%%%%%%%%%%%%%%%%%%%%%%%%%%%%%%%%%%%%%%%%%%

\medskip
%%%%%%%%%%%%%%%%%%%%%%%%%%%%%%%%%%%%%%%%%%%%%%%%%%%%%%%%%%%%%%%%%%%%%%%%%%%%%%%
%\vspace{1cm}

\flushleft{{\em Acknowledgment}: 
The authors gratefully acknowledge the support by the 
Czech Grant Agency through the grant GA CR P 201/12/G028. 
}
%%%%%%%%%%%%%%%%%%%%%%%%%%%%%%%%%%%%%%%%%%%%%%%%%%%%%%%%%%%%%%%%%%%%%%%%%%%%%%%%%%%%%%%%%%%% 

%\bibliographystyle{plain}
%\bibliography{../vectorpartition/algorithm_paper/TodorMilevsBibliography}

\end{document}